\documentclass[secthm,seceqn,amsthm,ussrhead,reqno, 12pt]{amsart}
\usepackage[utf8]{inputenc}
\usepackage[english]{babel}
\usepackage[symbol]{footmisc}
\usepackage{amssymb,amsmath,amsthm,amsfonts,xcolor,enumerate,hyperref,comment,longtable,cleveref}

\usepackage{times}
\usepackage{cite}
\usepackage{pdflscape}
\usepackage{ulem}
\usepackage[mathcal]{euscript}
\usepackage{tikz}
\usepackage{hyperref}
\usepackage{cancel}
\usepackage{stmaryrd}

\usetikzlibrary{arrows}

\newtheorem{theorem}{Theorem}
\newtheorem{lemma}[theorem]{Lemma}

\newtheorem{definition}[theorem]{Definition}

\newtheorem*{theoremA}{Theorem A}

\newtheorem*{theoremB}{Theorem B}

\usepackage{stmaryrd}
\usepackage{xcolor}

\begin{document}

\noindent{\Large
The algebraic and geometric classification of  \\  nilpotent    Leibniz   algebras}\footnote{
The first part of this work is supported by 
FCT   UIDB/MAT/00212/2020 and UIDP/MAT/00212/2020.
The second part of this work is supported by the Russian Science Foundation under grant 22-71-10001.  
}
\footnote{Corresponding author: Ivan Kaygorodov   (kaygorodov.ivan@gmail.com)}

 \bigskip

\begin{center}

 {\bf
Kobiljon Abdurasulov\footnote{CMA-UBI, Universidade da Beira Interior, Covilh\~{a}, Portugal; \ Institute of Mathematics Academy of
Sciences of Uzbekistan, Tashkent, Uzbekistan; \ abdurasulov0505@mail.ru}, 
Ivan Kaygorodov\footnote{CMA-UBI, Universidade da Beira Interior, Covilh\~{a}, Portugal;   \
Moscow Center for Fundamental and Applied Mathematics, Moscow,   Russia; \
Saint Petersburg  University, Russia; \
    kaygorodov.ivan@gmail.com}
  \&
Abror Khudoyberdiyev\footnote{Institute of Mathematics Academy of
Sciences of Uzbekistan, Tashkent, Uzbekistan; \
National University of Uzbekistan, Tashkent, Uzbekistan, \ khabror@mail.ru}

}

\end{center}

\ 

\medskip

\noindent{\bf Abstract}:
{\it This paper is devoted to the complete algebraic and geometric classification of complex $5$-dimensional nilpotent   Leibniz  algebras.
In particular, 
the variety of complex $5$-dimensional nilpotent  Leibniz algebras has 
dimension $24$  it has 
$10$  irreducible components 
(there is only one rigid algebra in this variety).
 }
 
\medskip

\bigskip



 \bigskip 
 
\section*{Introduction}

\bigskip 

The algebraic classification (up to isomorphism) of algebras of dimension $n$ from a certain variety
defined by a certain family of polynomial identities is a classic problem in the theory of non-associative algebras.
There are many results related to the algebraic classification of small-dimensional algebras in the varieties of
Jordan, Lie, Leibniz, Zinbiel and many other algebras \cite{  klp20, kkl20, ale3,     DGK18,    degr3, usefi1, degr2,    ha16,    kkp20}.
 Geometric properties of a variety of algebras defined by a family of polynomial identities have been an object of study since 1970's (see, \cite{wolf2, wolf1,  ak21,afk21,   chouhy,     BC99, aleis, aleis2,   gabriel,  fkk21, GRH3, ckls,    shaf, GRH, GRH2, ale3,     ikv17,   kppv,    kv16, S90}). 
 Gabriel described the irreducible components of the variety of $4$-dimensional unital associative algebras~\cite{gabriel}.  
 Burde and Steinhoff  constructed the graphs of degenerations for the varieties of    $3$-dimensional and $4$-dimensional Lie algebras~\cite{BC99}. 
 Grunewald and O'Halloran  calculated the degenerations for the variety of $5$-dimensional nilpotent Lie algebras~\cite{GRH}. 
Chouhy  proved that  in the case of finite-dimensional associative algebras,
 the $N$-Koszul property is preserved under the degeneration relation~\cite{chouhy}.
Degenerations have also been used to study a level of complexity of an algebra~\cite{wolf1,wolf2}.
 The study of degenerations of algebras is very rich and closely related to deformation theory, in the sense of Gerstenhaber \cite{ger63}.

\newpage 

In the present paper, we give the algebraic  and geometric classification of
complex $5$-dimensional nilpotent  Leibniz algebras.
 An algebra ${\bf   A}$ is called a (right) {\it   Leibniz  algebra}  if it satisfies the identity
$$(xy)z=(xz)y+x(yz).$$ 

Leibniz algebras present a "non antisymmetric" generalization of Lie algebras.
It appeared in some papers of Bloh [in 1960s] 
and Loday  [in 1990s].
Recently, they appeared in many geometric and physics applications (see, for example, 
\cite{bonez, vert1,leib2, kotov20,  KW01,  lavau, vert2, m22, STZ21,    strow20} and references therein).
A systematic study of algebraic properties of Leibniz algebras is started from the Loday paper \cite{L93}.
So, several classical theorems from Lie algebras theory have been extended
to the Leibniz algebras case;
many classification results regarding nilpotent, solvable, simple, and semisimple Leibniz algebras
are obtained 
(see, for example, \cite{aors21,ak21,vert1,kky22,CKLO13,leib2,  ikv17,kppv,KW01,m22, ms22, STZ21, T21} and references therein).
Leibniz algebras are   terminal algebras. 
Symmetric Leibniz algebras (i.e. left and right Leibniz algebras) are Poisson-admissible algebras.

Our method for classifying nilpotent  Leibniz algebras is based on the calculation of central extensions of nilpotent algebras of smaller dimensions from the same variety.
The algebraic study of central extensions of   algebras has been an important topic for years \cite{ hac16,  ss78}.
First, Skjelbred and Sund used central extensions of Lie algebras to obtain a classification of nilpotent Lie algebras  \cite{ss78}.
Note that the Skjelbred-Sund method of central extensions is an important tool in the classification of nilpotent algebras.
Using the same method,  
 $4$-dimensional nilpotent 
(terminal, 
 commutative  \cite{fkkv}) algebras,
  $5$-dimensional nilpotent 
(Jordan \cite{ha16},
antiassociative \cite{fkk21},
Zinbiel \cite{afk21},
symmetric Leibniz \cite{ak21},
 restricted Lie  \cite{usefi1}) algebras,
 $6$-dimensional nilpotent 
(Lie  \cite{degr3,degr2}, 
 anticommutative) algebras,
 $8$-dimensional   dual Mock-Lie algebras \cite{ckls},
and some others have been described. Our main results related to the algebraic classification of cited varieties are summarized below.

\begin{theoremA}
Up to isomorphism, there are infinitely many isomorphism classes of complex $5$-dimensional nilpotent (non-symmetric) Leibniz algebras, 
described explicitly in  section \ref{secteoA} 
in terms of 2 two-parameter families 18 one-parameter families and 62 additional isomorphism classes.
\end{theoremA}

 The degenerations between the (finite-dimensional) algebras from a certain variety $\mathfrak{V}$ defined by a set of identities have been actively studied in the past decade.
The description of all degenerations allows one to find the so-called rigid algebras and families of algebras, i.e., those whose orbit closures under the action of the general linear group form irreducible components of $\mathfrak{V}$
(with respect to the Zariski topology). 
We list here some works in which the rigid algebras of the varieties of
all $4$-dimensional Leibniz algebras \cite{ikv17},
all $4$-dimensional nilpotent terminal algebras,
all $4$-dimensional nilpotent commutative algebras \cite{fkkv},
all $6$-dimensional nilpotent anticommutative algebras,
all $8$-dimensional dual Mock Lie algebras \cite{ckls}
have been found.
A full description of degenerations has been obtained  
for $2$-dimensional algebras, 
for $4$-dimensional Lie algebras in \cite{BC99},
for $4$-dimensional Zinbiel and  $4$-dimensional nilpotent Leibniz algebras in \cite{kppv},
for $6$-dimensional nilpotent Lie algebras in \cite{S90,GRH},  
for $8$-dimensional $2$-step nilpotent anticommutative algebras,
and so on.
Our main results related to the geometric classification of cited varieties are summarized below. 

\begin{theoremB}
The variety of complex $5$-dimensional nilpotent  Leibniz algebras has 
dimension $24$  it has 
$10$  irreducible components 
(in particular, there is only one rigid algebra in this variety). 
\end{theoremB}

 \newpage
\section{The algebraic classification of nilpotent Leibniz  algebras}
\subsection{Method of classification of nilpotent Leibniz algebras}
Throughout this paper, we use the notations and methods well written in \cite{  hac16},
which we have adapted for the Leibniz case with some modifications.
Further in this section we give some important definitions.

Let $({\bf A}, \cdot)$ be a complex  Leibniz    algebra 
and $\mathbb V$ be a complex  vector space. The $\mathbb C$-linear space ${\rm Z^{2}}\left(
\bf A,\mathbb V \right) $ is defined as the set of all  bilinear maps $\theta  \colon {\bf A} \times {\bf A} \longrightarrow {\mathbb V}$ such that
\[ \theta(xy,z)=\theta(xz,y)+\theta(x,yz). \]

These elements will be called {\it cocycles}. For a
linear map $f$ from $\bf A$ to  $\mathbb V$, if we define $\delta f\colon {\bf A} \times
{\bf A} \longrightarrow {\mathbb V}$ by $\delta f  (x,y ) =f(xy )$, then $\delta f\in {\rm Z^{2}}\left( {\bf A},{\mathbb V} \right) $. We define ${\rm B^{2}}\left({\bf A},{\mathbb V}\right) =\left\{ \theta =\delta f\ : f\in {\rm Hom}\left( {\bf A},{\mathbb V}\right) \right\} $.
We define the {\it second cohomology space} ${\rm H^{2}}\left( {\bf A},{\mathbb V}\right) $ as the quotient space ${\rm Z^{2}}
\left( {\bf A},{\mathbb V}\right) \big/{\rm B^{2}}\left( {\bf A},{\mathbb V}\right) $.

\

Let $\operatorname{Aut}({\bf A}) $ be the automorphism group of  ${\bf A} $ and let $\phi \in \operatorname{Aut}({\bf A})$. For $\theta \in
{\rm Z^{2}}\left( {\bf A},{\mathbb V}\right) $ define  the action of the group $\operatorname{Aut}({\bf A}) $ on ${\rm Z^{2}}\left( {\bf A},{\mathbb V}\right) $ by $\phi \theta (x,y)
=\theta \left( \phi \left( x\right) ,\phi \left( y\right) \right) $.  It is easy to verify that
 ${\rm B^{2}}\left( {\bf A},{\mathbb V}\right) $ is invariant under the action of $\operatorname{Aut}({\bf A}).$  
 So, we have an induced action of  $\operatorname{Aut}({\bf A})$  on ${\rm H^{2}}\left( {\bf A},{\mathbb V}\right)$.

\

Let $\bf A$ be a Leibniz  algebra of dimension $m$ over  $\mathbb C$ and ${\mathbb V}$ be a $\mathbb C$-vector
space of dimension $k$. For the bilinear map $\theta$, define on the linear space ${\bf A}_{\theta } = {\bf A}\oplus {\mathbb V}$ the
bilinear product `` $\left[ -,-\right] _{{\bf A}_{\theta }}$'' by $\left[ x+x^{\prime },y+y^{\prime }\right] _{{\bf A}_{\theta }}=
 xy +\theta(x,y) $ for all $x,y\in {\bf A},x^{\prime },y^{\prime }\in {\mathbb V}$.
The algebra ${\bf A}_{\theta }$ is called a $k$-{\it dimensional central extension} of ${\bf A}$ by ${\mathbb V}$. One can easily check that ${\bf A_{\theta}}$ is a Leibniz 
algebra if and only if $\theta \in {\rm Z^2}({\bf A}, {\mathbb V})$.

Call the
set $\operatorname{Ann}(\theta)=\left\{ x\in {\bf A}:\theta \left( x, {\bf A} \right)+ \theta \left({\bf A} ,x\right) =0\right\} $
the {\it annihilator} of $\theta $. We recall that the {\it annihilator} of an  algebra ${\bf A}$ is defined as
the ideal $\operatorname{Ann}(  {\bf A} ) =\left\{ x\in {\bf A}:  x{\bf A}+ {\bf A}x =0\right\}$. Observe
 that
$\operatorname{Ann}\left( {\bf A}_{\theta }\right) =(\operatorname{Ann}(\theta) \cap\operatorname{Ann}({\bf A}))
 \oplus {\mathbb V}$.

\

The following result shows that every algebra with a non-zero annihilator is a central extension of a smaller-dimensional algebra.

\begin{lemma}
Let ${\bf A}$ be an $n$-dimensional Leibniz  algebra such that $\dim (\operatorname{Ann}({\bf A}))=m\neq0$. Then there exists, up to isomorphism, a unique $(n-m)$-dimensional  Leibniz   algebra ${\bf A}'$ and a bilinear map $\theta \in {\rm Z^2}({\bf A'}, {\mathbb V})$ with $\operatorname{Ann}({\bf A'})\cap\operatorname{Ann}(\theta)=0$, where $\mathbb V$ is a vector space of dimension m, such that ${\bf A} \cong {{\bf A}'}_{\theta}$ and
 ${\bf A}/\operatorname{Ann}({\bf A})\cong {\bf A}'$.
\end{lemma}

\begin{proof}
Let ${\bf A}'$ be a linear complement of $\operatorname{Ann}({\bf A})$ in ${\bf A}$. Define a linear map $P \colon {\bf A} \longrightarrow {\bf A}'$ by $P(x+v)=x$ for $x\in {\bf A}'$ and $v\in\operatorname{Ann}({\bf A})$, and define a multiplication on ${\bf A}'$ by $[x, y]_{{\bf A}'}=P(x y)$ for $x, y \in {\bf A}'$.
For $x, y \in {\bf A}$, we have
\[P(xy)=P((x-P(x)+P(x))(y- P(y)+P(y)))=P(P(x) P(y))=[P(x), P(y)]_{{\bf A}'}. \]

Since $P$ is a homomorphism $P({\bf A})={\bf A}'$ is a  Leibniz  algebra and
 ${\bf A}/\operatorname{Ann}({\bf A})\cong {\bf A}'$, which gives us the uniqueness. Now, define the map $\theta \colon {\bf A}' \times {\bf A}' \longrightarrow\operatorname{Ann}({\bf A})$ by $\theta(x,y)=xy- [x,y]_{{\bf A}'}$.
  Thus, ${\bf A}'_{\theta}$ is ${\bf A}$ and therefore $\theta \in {\rm Z^2}({\bf A'}, {\mathbb V})$ and $\operatorname{Ann}({\bf A'})\cap\operatorname{Ann}(\theta)=0$.
\end{proof}

\

\;
\begin{definition}
Let ${\bf A}$ be an algebra and $I$ be a subspace of $\operatorname{Ann}({\bf A})$. If ${\bf A}={\bf A}_0 \oplus I$
then $I$ is called an {\it annihilator component} of ${\bf A}$.
A central extension of an algebra $\bf A$ without annihilator component is called a {\it non-split central extension}.
\end{definition}

Our task is to find all central extensions of an algebra $\bf A$ by
a space ${\mathbb V}$.  In order to solve the isomorphism problem we need to study the
action of $\operatorname{Aut}({\bf A})$ on ${\rm H^{2}}\left( {\bf A},{\mathbb V}
\right) $. To do that, let us fix a basis $e_{1},\ldots ,e_{s}$ of ${\mathbb V}$, and $
\theta \in {\rm Z^{2}}\left( {\bf A},{\mathbb V}\right) $. Then $\theta $ can be uniquely
written as $\theta \left( x,y\right) =
\displaystyle \sum_{i=1}^{s} \theta _{i}\left( x,y\right) e_{i}$, where $\theta _{i}\in
{\rm Z^{2}}\left( {\bf A},\mathbb C\right) $. Moreover, $\operatorname{Ann}(\theta)=\operatorname{Ann}(\theta _{1})\cap\operatorname{Ann}(\theta _{2})\cap\ldots \cap\operatorname{Ann}(\theta _{s})$. Furthermore, $\theta \in
{\rm B^{2}}\left( {\bf A},{\mathbb V}\right) $\ if and only if all $\theta _{i}\in {\rm B^{2}}\left( {\bf A},
\mathbb C\right) $.
It is not difficult to prove (see \cite[Lemma 13]{hac16}) that given a  Leibniz  algebra ${\bf A}_{\theta}$, if we write as
above $\theta \left( x,y\right) = \displaystyle \sum_{i=1}^{s} \theta_{i}\left( x,y\right) e_{i}\in {\rm Z^{2}}\left( {\bf A},{\mathbb V}\right) $ and 
$\operatorname{Ann}(\theta)\cap \operatorname{Ann}\left( {\bf A}\right) =0$, then ${\bf A}_{\theta }$ has an
annihilator component if and only if $\left[ \theta _{1}\right] ,\left[
\theta _{2}\right] ,\ldots ,\left[ \theta _{s}\right] $ are linearly
dependent in ${\rm H^{2}}\left( {\bf A},\mathbb C\right) $.

\;

Let ${\mathbb V}$ be a finite-dimensional vector space over $\mathbb C$. The {\it Grassmannian} $G_{k}\left( {\mathbb V}\right) $ is the set of all $k$-dimensional
linear subspaces of $ {\mathbb V}$. Let $G_{s}\left( {\rm H^{2}}\left( {\bf A},\mathbb C\right) \right) $ be the Grassmannian of subspaces of dimension $s$ in
${\rm H^{2}}\left( {\bf A},\mathbb C\right) $. There is a natural action of $\operatorname{Aut}({\bf A})$ on $G_{s}\left( {\rm H^{2}}\left( {\bf A},\mathbb C\right) \right) $.
Let $\phi \in \operatorname{Aut}({\bf A})$. For $W=\left\langle
\left[ \theta _{1}\right] ,\left[ \theta _{2}\right] ,\dots,\left[ \theta _{s}
\right] \right\rangle \in G_{s}\left( {\rm H^{2}}\left( {\bf A},\mathbb C
\right) \right) $ define $\phi W=\left\langle \left[ \phi \theta _{1}\right]
,\left[ \phi \theta _{2}\right] ,\dots,\left[ \phi \theta _{s}\right]
\right\rangle $. We denote the orbit of $W\in G_{s}\left(
{\rm H^{2}}\left( {\bf A},\mathbb C\right) \right) $ under the action of $\operatorname{Aut}({\bf A})$ by $\operatorname{Orb}(W)$. Given
\[
W_{1}=\left\langle \left[ \theta _{1}\right] ,\left[ \theta _{2}\right] ,\dots,
\left[ \theta _{s}\right] \right\rangle ,W_{2}=\left\langle \left[ \vartheta
_{1}\right] ,\left[ \vartheta _{2}\right] ,\dots,\left[ \vartheta _{s}\right]
\right\rangle \in G_{s}\left( {\rm H^{2}}\left( {\bf A},\mathbb C\right)
\right),
\]
we easily have that if $W_{1}=W_{2}$, then $ \bigcap\limits_{i=1}^{s}\operatorname{Ann}(\theta _{i})\cap \operatorname{Ann}\left( {\bf A}\right) = \bigcap\limits_{i=1}^{s}
\operatorname{Ann}(\vartheta _{i})\cap\operatorname{Ann}( {\bf A}) $, and therefore we can introduce
the set
\[
{\bf T}_{s}({\bf A}) =\left\{ W=\left\langle \left[ \theta _{1}\right] ,
\left[ \theta _{2}\right] ,\dots,\left[ \theta _{s}\right] \right\rangle \in
G_{s}\left( {\rm H^{2}}\left( {\bf A},\mathbb C\right) \right) : \bigcap\limits_{i=1}^{s}\operatorname{Ann}(\theta _{i})\cap\operatorname{Ann}({\bf A}) =0\right\},
\]
which is stable under the action of $\operatorname{Aut}({\bf A})$.

\

Now, let ${\mathbb V}$ be an $s$-dimensional linear space and let us denote by
${\bf E}\left( {\bf A},{\mathbb V}\right) $ the set of all {\it non-split $s$-dimensional central extensions} of ${\bf A}$ by
${\mathbb V}$. By above, we can write
\[
{\bf E}\left( {\bf A},{\mathbb V}\right) =\left\{ {\bf A}_{\theta }:\theta \left( x,y\right) = \sum_{i=1}^{s}\theta _{i}\left( x,y\right) e_{i} \ \ \text{and} \ \ \left\langle \left[ \theta _{1}\right] ,\left[ \theta _{2}\right] ,\dots,
\left[ \theta _{s}\right] \right\rangle \in {\bf T}_{s}({\bf A}) \right\} .
\]
We also have the following result, which can be proved as in \cite[Lemma 17]{hac16}.

\begin{lemma}
 Let ${\bf A}_{\theta },{\bf A}_{\vartheta }\in {\bf E}\left( {\bf A},{\mathbb V}\right) $. Suppose that $\theta \left( x,y\right) =  \displaystyle \sum_{i=1}^{s}
\theta _{i}\left( x,y\right) e_{i}$ and $\vartheta \left( x,y\right) =
\displaystyle \sum_{i=1}^{s} \vartheta _{i}\left( x,y\right) e_{i}$.
Then the  Leibniz   algebras ${\bf A}_{\theta }$ and ${\bf A}_{\vartheta } $ are isomorphic
if and only if
$$\operatorname{Orb}\left\langle \left[ \theta _{1}\right] ,
\left[ \theta _{2}\right] ,\dots,\left[ \theta _{s}\right] \right\rangle =
\operatorname{Orb}\left\langle \left[ \vartheta _{1}\right] ,\left[ \vartheta
_{2}\right] ,\dots,\left[ \vartheta _{s}\right] \right\rangle .$$
\end{lemma}

This shows that there exists a one-to-one correspondence between the set of $\operatorname{Aut}({\bf A})$-orbits on ${\bf T}_{s}\left( {\bf A}\right) $ and the set of
isomorphism classes of ${\bf E}\left( {\bf A},{\mathbb V}\right) $. Consequently we have a
procedure that allows us, given a  Leibniz  algebra ${\bf A}'$ of
dimension $n-s$, to construct all non-split central extensions of ${\bf A}'$. This procedure is:

\begin{enumerate}
\item For a given Leibniz  algebra ${\bf A}'$ of dimension $n-s $, determine ${\rm H^{2}}( {\bf A}',\mathbb {C}) $, $\operatorname{Ann}({\bf A}')$ and $\operatorname{Aut}({\bf A}')$.

\item Determine the set of $\operatorname{Aut}({\bf A}')$-orbits on ${\bf T}_{s}({\bf A}') $.

\item For each orbit, construct the Leibniz  algebra associated with a
representative of it.
\end{enumerate}

\medskip

The above described method gives all (symmetric and non-symmetric)  Leibniz  algebras. But we are interested in developing this method in such a way that it only gives non-symmetric Leibniz    algebras, because the classification of all symmetric Leibniz  is given in \cite{ak21}. Clearly, any central extension of a non-symmetric  Leibniz  is non-symmetric. But a symmetric Leibniz  algebra may have extensions which are not symmetric Leibniz  algebras. More precisely, let ${\mathbb{S}}$ be a symmetric Leibniz  algebra and $\theta \in {\rm Z^2}({\mathbb{S}}, {\mathbb C}).$ Then ${\mathbb{S}}_{\theta }$ is a symmetric Leibniz  algebra if and only if 
\begin{equation*}
 \theta(x,yz)= \theta(xy,z)+\theta(y,xz). 
 \end{equation*}
for all $x,y,z\in {\mathbb{S}}.$ Define the subspace ${\rm Z_{\mathcal{S}}^2}({\mathbb{S}},{\mathbb C})$ of ${\rm Z^2}({\mathbb{S}},{\mathbb C})$ by
\begin{equation*}
{\rm Z^2}_{\mathcal{S}}({\mathbb{S}},{\mathbb C}) =\left\{\theta \in {\rm Z^2}({\mathbb{S}},{\mathbb C}) : \theta(x,yz)= \theta(xy,z)+\theta(y,xz), \text{ for all } x, y,z\in {\mathbb{S}}\right\}.
\end{equation*}

Observe that ${\rm B^2}({\mathbb{S}},{\mathbb C})\subseteq{\rm Z_{\mathcal{S}}^2}({\mathbb{S}},{\mathbb C}).$
Let ${\rm H_{\mathcal{S}}^2}({\mathbb{S}},{\mathbb C}) =
{\rm Z_{\mathcal{S}}^2}({\mathbb{S}},{\mathbb C}) \big/{\rm B^2}({\mathbb{S}},{\mathbb C}).$ 
Then ${\rm H_{\mathcal{S}}^2}({\mathbb{S}},{\mathbb C})$ is a subspace of $
{\rm H^2}({\mathbb{S}},{\mathbb C}).$ Define 
\begin{eqnarray*}
{\bf R}_{s}({\mathbb{S}})  &=&\left\{ {\mathcal W}\in {\bf T}_{s}({\mathcal{S}}) :{\mathcal W}\in G_{s}({\rm H_{\mathcal{S}}^2}({\mathbb{S}},{\mathbb C}) ) \right\}, \\
{\bf U}_{s}({\mathbb{S}})  &=&\left\{ {\mathcal W}\in {\bf T}_{s}({\mathbb{S}}) :{\mathcal W}\notin G_{s}({\rm H_{\mathcal{S}}^2}({\mathbb{S}},{\mathbb C}) ) \right\}.
\end{eqnarray*}
Then ${\bf T}_{s}({\mathbb{S}}) ={\bf R}_{s}(
{\mathbb{S}}) \mathbin{\mathaccent\cdot\cup} {\bf U}_{s}(
{\mathbb{S}}).$ The sets ${\bf R}_{s}({\mathbb{S}}) $
and ${\bf U}_{s}({\mathbb{S}})$ are stable under the action
of $\operatorname{Aut}({\mathbb{S}}).$ Thus, the  Leibniz  algebras
corresponding to the representatives of $\operatorname{Aut}({\mathbb{S}}) $%
-orbits on ${\bf R}_{s}({\mathbb{S}})$ are symmetric Leibniz  algebras,
while those corresponding to the representatives of $\operatorname{Aut}({\mathbb{S}}%
) $-orbits on ${\bf U}_{s}({\mathbb{S}})$ are non-symmetric Leibniz algebras. Hence, we may construct all non-split non-symmetric Leibniz algebras $%
\bf{A}$ of dimension $n$ with $s$-dimensional annihilator 
from a given Leibniz algebra $\bf{A}%
^{\prime }$ of dimension $n-s$ in the following way:

\begin{enumerate}
\item If $\bf{A}^{\prime }$ is a non-symmetric Leibniz algebra, then apply the procedure.

\item Otherwise, do the following:

\begin{enumerate}
\item Determine ${\bf U}_{s}(\bf{A}^{\prime })$ and $%
\operatorname{Aut}(\bf{A}^{\prime }).$

\item Determine the set of $\operatorname{Aut}(\bf{A}^{\prime })$-orbits on ${\bf U%
}_{s}(\bf{A}^{\prime }).$

\item For each orbit, construct the   Leibniz  algebra corresponding to one of its
representatives.
\end{enumerate}
\end{enumerate}

\subsubsection{Notations}
Let us introduce the following notations. Let ${\bf A}$ be a nilpotent algebra with
a basis $\{e_{1},e_{2}, \ldots, e_{n}\}.$ Then by $\Delta_{ij}$\ we will denote the
bilinear form
$\Delta_{ij}:{\bf A}\times {\bf A}\longrightarrow \mathbb C$
with $\Delta_{ij}(e_{l},e_{m}) = \delta_{il}\delta_{jm}.$
The set $\left\{ \Delta_{ij}:1\leq i, j\leq n\right\}$ is a basis for the linear space of
bilinear forms on ${\bf A},$ so every $\theta \in
{\rm Z^2}({\bf A},\bf \mathbb V )$ can be uniquely written as $
\theta = \displaystyle \sum_{1\leq i,j\leq n} c_{ij}\Delta _{{i}{j}}$, where $
c_{ij}\in \mathbb C$.
Let us fix the following notations for our nilpotent algebras:

$$\begin{array}{lll}

{\mathcal N}_{j}& \mbox{---}& j\mbox{th }3\mbox{-dimensional        $2$-step nilpotent algebra.} \\

\mathcal{L}_{j}& \mbox{---}& j\mbox{th }3\mbox{-dimensional    Leibniz (non-2-step nilpotent) algebra.} \\

{\mathfrak N}_{j}& \mbox{---}& j\mbox{th }4\mbox{-dimensional   $2$-step nilpotent algebra.} \\

\mathfrak{L}_{j}& \mbox{---}& j\mbox{th }4\mbox{-dimensional    Leibniz (non-2-step nilpotent) algebra.} \\

\mathbb{L}_{j}& \mbox{---}& j\mbox{th }5\mbox{-dimensional   non-split nilpotent non-symmetric Leibniz algebra.} \\

\end{array}$$

\subsection{$2$-dimensional central extensions of $2$-dimensional    nilpotent Leibniz algebras}

\subsubsection{ $3$-dimensional    nilpotent Leibniz algebras and their cohomology spaces}
Let us use the algebraic classification of $3$-dimensional nilpotent Leibniz algebras from the Corrigendum to \cite{kppv}.

\begin{longtable}{ll llllll} 
\hline
\multicolumn{8}{c}{{\bf The list of 2-step nilpotent 3-dimensional Leibniz algebras}}  \\
\hline
 
{${\mathcal N}_{01}$} &$:$ &  $e_1e_1 = e_2$ &&&&\\ 
\multicolumn{8}{l}{
${\rm H}^2({\mathcal N}_{01})=
\Big\langle  
[\Delta_{13}],[\Delta_{21}],[\Delta_{31}],[\Delta _{33}]
\Big\rangle $}\\

\hline
{${\mathcal N}_{02}$} &$:$ &  $e_1e_2=  e_3$ & $e_2e_1=-e_3$ &&& \\ 

\multicolumn{8}{l}{
${\rm H}^2({\mathcal N}_{02})={\rm H}_{\mathcal{S}}^2({\mathcal N}_{02})=
\Big\langle 
[\Delta _{11}],[\Delta _{12}],[\Delta_{13}-\Delta _{31}],[\Delta_{22}],[\Delta _{23}-\Delta_{32}]
\Big\rangle $}\\

\hline
${\mathcal N}_{03}^{\alpha}$ &$:$ & $e_1e_1=  e_3$ & $e_1e_2=e_3$& $e_2e_2=\alpha e_3$  &&\\  

\multicolumn{8}{l}{
${\rm H}^2({\mathcal N}_{03}^{\alpha\neq0})=
\Big\langle 
[\Delta_{11}],[\Delta_{21}],[\Delta_{22}]
\Big\rangle $}\\
 
\multicolumn{8}{l}{
${\rm H}^2({\mathcal N}_{03}^{0})=
\Big\langle 
[\Delta_{11}],[\Delta_{21}],[\Delta_{22}], [\Delta_{31}+\Delta_{32}]
\Big\rangle $}\\

\hline
${\mathcal N}_{04}$ &$:$ & $e_1e_1=  e_3$& $e_1e_2=e_3$&  $e_2e_1=e_3$ &&\\ 

\multicolumn{8}{l}{
${\rm H}^2({\mathcal N}_{04})=
\Big\langle 
[\Delta_{11}],[\Delta_{12}],[\Delta_{22}]
\Big\rangle $}\\

\hline
\multicolumn{8}{c}{
\bf The list of 3-step nilpotent 3-dimensional Leibniz algebras} \\
\hline

$\mathcal{L}_1$ &$:$ & $e_1e_1 = e_2$ &$ e_2e_1=e_3$\\ 

\multicolumn{8}{l}{
${\rm H}^2(\mathcal{L}_1)=
\Big\langle 
[\Delta_{31}]
\Big\rangle $}\\

\hline
\end{longtable}

From the previous table we obtain that any extension of the algebra ${\mathcal N}_{02}$ is a symmetric Leibniz algebra and any extension of the algebras ${\mathcal N}_{03}^{\alpha\neq 0}$ and ${\mathcal N}_{04}$ is split and the algebra $\mathcal{L}_1$ has no non-split 2-dimensional central extension. Therefore, is sufficient to consider for the algebras ${\mathcal N}_{01}$ and ${\mathcal N}_{03}^{0}.$  
Note that these algebras are non-Lie symmetric Leibniz algebras and any non-split extension of non-Lie Leibniz algebra is a non-symmetric Leibniz algebra\cite{Benayadi}. 
Thus, non-split extension of the algebras ${\mathcal N}_{01}$ and ${\mathcal N}_{03}^{0}$ gives us non-symmetric Leibniz algebras.

\subsubsection{Central extensions of ${\mathcal N}_{01}$}
	Let us use the following notations:
	$$	\nabla_1 = [\Delta_{13}], \quad \nabla_2 = [\Delta_{21}], \quad \nabla_3 = [\Delta_{31}],\quad \nabla_4 = [\Delta_{33}].
	$$

Take $\theta=\sum\limits_{i=1}^4\alpha_i\nabla_i\in {\rm H^2}({\mathcal N}_{01}).$
	The automorphism group of ${\mathcal N}_{01}$ consists of invertible matrices of the form
	$$\phi=
	\begin{pmatrix}
	x &    0    &  0\\
	y &    x^2  &  u\\
	z &    0    &  t\\
	\end{pmatrix}.$$

 Since
	$$
	\phi^T\begin{pmatrix}
	0 & 0  & \alpha_1\\
	\alpha_2  & 0 & 0 \\
	\alpha_3&  0    & \alpha_4\\
	\end{pmatrix} 
	\phi=\begin{pmatrix}
	\alpha^*   & 0 & \alpha_1^*\\
	\alpha_2^* & 0 & 0 \\
	\alpha_3^* & 0 & \alpha_4^*
	\end{pmatrix},$$
	 we have that the action of ${\rm Aut} ({\mathcal N}_{01})$ on the subspace
$\langle \sum\limits_{i=1}^4\alpha_i\nabla_i  \rangle$
is given by
$\langle \sum\limits_{i=1}^4\alpha_i^{*}\nabla_i\rangle,$
where
\begin{longtable}{ll}
$\alpha^*_1=tx \alpha_1+tz \alpha_4,$ &
$\alpha^*_2=x^3 \alpha_2,$ \\
$\alpha^*_3=u x \alpha _2+t x \alpha _3+t z \alpha _4,$ &
$\alpha_4^*=t^2 \alpha _4.$\\
\end{longtable}

We are interested  only in $\alpha_2\neq 0$
  and consider the vector space generated by the following two cocycles:
$$ \theta_1=\alpha_1\nabla_1+\alpha_2\nabla_2+\alpha_3\nabla_3+\alpha_4\nabla_4 \ \ \text{and} \ \  \theta_2=\beta_1\nabla_1+\beta_3\nabla_3+\beta_4\nabla_4.$$

Thus, we have
\begin{longtable}{ll}
$\alpha^*_1=tx \alpha_1+tz \alpha_4,$ & $\beta^*_1=tx \beta_1+tz \beta_4,$ \\
$\alpha^*_2=x^3 \alpha_2,$ & \\ 
$\alpha^*_3=u x \alpha _2+t x \alpha _3+t z \alpha _4,$ & $\beta^*_3=t x \beta_3+t z \beta_4,$\\ 
$\alpha_4^*=t^2 \alpha _4.$ & $\beta_4^*=t^2 \beta_4.$ \\
\end{longtable}

Consider the following cases.

\begin{enumerate}
	\item $\beta_4=0,$ $\alpha_4=0,$ then:
	
\begin{enumerate}
	\item if $\beta_1=0,$ $\alpha_1=0,$ then we have the   representative $ \left\langle \nabla_2,\nabla_3 \right\rangle; $
	
	\item if $\beta_1=0,$ $\alpha_1\neq0,$ then choosing $x=1,$ $t = \frac{\alpha_2}{\alpha_1},$ we have the   representative $ \left\langle \nabla_1+\nabla_2,\nabla_3 \right\rangle; $
	
	\item if $\beta_1\neq0,$ then choosing $t=\alpha_2\beta_1,$ $u = \alpha_1\beta_3-\alpha_3\beta_1,$ we have the representative $ \left\langle \nabla_2, \nabla_1 + \alpha \nabla_3 \right\rangle.$
	
\end{enumerate}	
	
\item $\beta_4=0,$ $\alpha_4\neq 0,$ then choosing $x=1,$ $z = - \frac{ \alpha_1}{\alpha_4,}$ $u=\frac{t(\alpha_1-\alpha_3)}{\alpha_2},$ $t = \sqrt{\frac{\alpha_2}{\alpha_4}},$ we have the representatives $ \left\langle \nabla_2+\nabla_4, \nabla_3 \right\rangle $ and $ \left\langle \nabla_2+\nabla_4, \nabla_1 + \alpha \nabla_3 \right\rangle$ depending on whether $ \beta_1=0 $ or not.

\item $\beta_4\neq 0,$ then we may suppose $\alpha_4=0$ and 	
\begin{enumerate}
	\item if $\beta_1=\beta_3$ and $\alpha_1=0,$ then 
choosing $x=1,$ $t=1,$ $z = - \frac{\beta_3}{\beta_4},$ 
$u=- \frac{\alpha_3}{\alpha_2},$  we have the representative $ \left\langle \nabla_2, \nabla_4 \right\rangle;$

\item if $\beta_1=\beta_3$ and $\alpha_1\neq0,$ then 
choosing $x=1,$ $z = - \frac{\beta_3}{\beta_4},$ 
$u=-\frac{\alpha_3}{\alpha_1}$ and $t=\frac{\alpha_2}{\alpha_1},$ we have the representative $ \left\langle \nabla_1+\nabla_2, \nabla_4 \right\rangle;$

\item if $\beta_1\neq\beta_3 $ and $\alpha_1=0,$ then 
choosing $x=\beta_4,$ 
$z = - \beta_3,$ 
$u=\frac{\alpha_3(\beta_3-\beta_1)}{\alpha_2}$ and 
$t=\beta_1-\beta_3,$ we have the representative 
$ \left\langle \nabla_2, \nabla_1+\nabla_4 \right\rangle;$

\item if $\beta_1\neq\beta_3$ and $\alpha_1\neq0,$ then 
choosing 
$x=\frac{\alpha_1(\beta_1-\beta_3)}{\alpha_2\beta_4},$ 
$z= - \frac{\alpha_1(\beta_1- \beta_3)\beta_3}{\alpha_2\beta_4^2},$ 
$u=-\frac{\alpha_1\alpha_3(\beta_1-\beta_3)^2}{\alpha_2^2\beta_4^2},$ and 
$t=\frac{\alpha_1(\beta_1-\beta_3)^2}{\alpha_2\beta_4^2}$ 
and we have the representative $ \left\langle \nabla_1+\nabla_2, \nabla_1+\nabla_4 \right\rangle.$

\end{enumerate}

\end{enumerate}

Now we have the following distinct orbits:
\begin{longtable} {lll}
$\langle \nabla_2, \nabla_3 \rangle$ & $\langle \nabla_1+\nabla_2,\nabla_3 \rangle$ & $\langle \nabla_2,\nabla_1 + \alpha \nabla_3\rangle$ \\
$\langle \nabla_2+\nabla_4, \nabla_3 \rangle$ & $\langle \nabla_2+\nabla_4, \nabla_1 + \alpha \nabla_3 \rangle$ & $\langle \nabla_2,\nabla_4 \rangle$ \\
$\langle \nabla_1+\nabla_2, \nabla_4 \rangle$ & $\langle \nabla_2, \nabla_1+\nabla_4 \rangle$ & $\langle \nabla_1+\nabla_2,\nabla_1+\nabla_4 \rangle.$ \\
\end{longtable}

Note that the orbit $\langle \nabla_2, \nabla_1+\nabla_4 \rangle$ gives the algebra isomorphic to the algebra $\mathbb{L}_{47}^{0,0}$ in section \ref{sec1.3.7}.
Hence, we have the following new $5$-dimensional nilpotent Leibniz algebras (see Section \ref{secteoA}):
\begin{longtable}{lllllllllllllll}
$\mathbb{L}_{01},$ &
$\mathbb{L}_{02},$ & 
$\mathbb{L}_{03}^\alpha,$ & $\mathbb{L}_{04},$ & 
$\mathbb{L}_{05}^\alpha,$ & 
$\mathbb{L}_{06},$ & 
$\mathbb{L}_{07},$ & 
$\mathbb{L}_{08}.$ & 
\end{longtable}


\subsubsection{Central extensions of ${\mathcal N}_{03}^0$}
	Let us use the following notations:
	$$	\nabla_1 = [\Delta_{11}], \quad \nabla_2 = [\Delta_{21}], \quad \nabla_3 = [\Delta_{22}],\quad \nabla_4 = [\Delta_{31}+\Delta_{32}].
	$$

Take $\theta=\sum\limits_{i=1}^4\alpha_i\nabla_i\in {\rm H^2}({\mathcal N}_{03}^0).$
	The automorphism group of ${\mathcal N}_{03}$ consists of invertible matrices of the form
	$$	\phi=
	\begin{pmatrix}
	x &    0  &  0\\
	y &  x+y  & 0\\
	z &  t    &  x(x+y)\\
	\end{pmatrix}.$$
	
 Since
	$$
	\phi^T\begin{pmatrix}
	\alpha_1  & 0  & 0\\
	\alpha_2  & \alpha_3 & 0 \\
	\alpha_4  & \alpha_4 & 0\\
	\end{pmatrix} \phi=	\begin{pmatrix}
	\alpha_1+\alpha^*  & \alpha^*  & 0\\
	\alpha_2  & \alpha_3 & 0 \\
	\alpha_4  & \alpha_4 & 0\\
	\end{pmatrix},
	$$
	 we have that the action of ${\rm Aut} ({\mathcal N}_{03}^0)$ on the subspace
$\langle \sum\limits_{i=1}^4\alpha_i\nabla_i  \rangle$
is given by
$\langle \sum\limits_{i=1}^4\alpha_i^{*}\nabla_i\rangle,$
where
\begin{longtable}{lcllcl}
$\alpha^*_1$&$=$&$x(x\alpha_1+y\alpha_2-y\alpha_3),$ &
$\alpha^*_2$&$=$&$(x+y)(x\alpha_2+y\alpha_3+t\alpha_4),$ \\
$\alpha^*_3$&$=$&$(x+y)((x+y)\alpha_3+t\alpha_4),$ &
$\alpha_4^*$&$=$&$x(x+y)^2\alpha_4.$\\
\end{longtable}

We are interested  only in $\alpha_4\neq 0$ and consider the vector space generated by the following two cocycles:
$$ \theta_1=\alpha_1\nabla_1+\alpha_2\nabla_2+\alpha_3\nabla_3+\alpha_4\nabla_4 \ \ \text{and} \ \  \theta_2=\beta_1\nabla_1+\beta_2\nabla_2+\beta_3\nabla_3.$$
Thus, we have

\begin{longtable}{lcllcl}
$\alpha^*_1$&$=$&$x(x\alpha_1+y\alpha_2-y\alpha_3),$ & $\beta^*_1$&$=$&$x(x\beta_1+y\beta_2-y\beta_3),$ \\
$\alpha^*_2$&$=$&$(x+y)(x\alpha_2+y\alpha_3+t\alpha_4),$ & 
$\beta^*_2$&$=$&$(x+y)(x\beta_2+y\beta_3),$\\
$\alpha^*_3$&$=$&$(x+y)((x+y)\alpha_3+t\alpha_4),$ &
$\beta^*_3$&$=$&$(x+y)^2\beta_3,$ \\
$\alpha_4^*$&$=$&$x(x+y)^2\alpha_4,$ & $\beta^*_4$&$=$&$0.$ \\
\end{longtable}

Consider the following cases.

\begin{enumerate}
	\item $\beta_3=0,$ $\beta_2=0,$ then $\beta_1\neq 0$ and choosing 
	$t =- \frac{(x+y)\alpha_3}{\alpha_4},$ we have $\alpha_3^*=0$ and 
\begin{longtable}{lcllcl}
$\alpha^*_2$&$=$&$x(x+y)(\alpha_2-\alpha_3),$ &
$\alpha_4^*$&$=$&$x(x+y)^2\alpha_4.$\\
\end{longtable}
\begin{enumerate}
	\item if $\alpha_2=\alpha_3,$ then we have the representative $ \left\langle \nabla_4,\nabla_1 \right\rangle; $
	\item if $\alpha_2\neq\alpha_3,$ then choosing $x = \frac{\alpha_2-\alpha_3} {\alpha_4},$ $y=0,$ we have the representative $ \left\langle \nabla_2+\nabla_4,\nabla_1 \right\rangle.$
\end{enumerate}
	\item $\beta_3=0,$ $\beta_2\neq0,$ then  choosing 
	$t =- \frac{(x+y)\alpha_3}{\alpha_4},$ we have $\alpha_3^*=0$ and 
\begin{longtable}{lcllcl}
$\alpha^*_1$&$=$&$x(x\alpha_1+y\alpha_2-y\alpha_3),$ & $\beta^*_1$&$=$&$x(x\beta_1+y\beta_2),$ \\
$\alpha^*_2$&$=$&$x(x+y)(\alpha_2-\alpha_3),$ & 
$\beta^*_2$&$=$&$x(x+y)\beta_2,$\\
$\alpha_4^*$&$=$&$x(x+y)^2\alpha_4.$
\end{longtable}

\begin{enumerate}
	\item if $\beta_1=\beta_2,$ then we have the representatives $ \left\langle \nabla_4,\nabla_1 + \nabla_2\right\rangle$ and $ \left\langle \nabla_1 + \nabla_4,\nabla_1 + \nabla_2\right\rangle$  depending on $\alpha_1=\alpha_2-\alpha_3$ or not; 
	\item if $\beta_1\neq\beta_2,$ then choosing 
	choosing $x =\beta_2,$
	$y =- \beta_1,$ we have $\beta_1^*=0$ and
	obtain the representatives $ \left\langle \nabla_4, \nabla_2\right\rangle$ and $ \left\langle \nabla_1 + \nabla_4,\nabla_2\right\rangle$  depending on 
	$\alpha_1\beta_2 = (\alpha_2-\alpha_3)\beta_1$ or not.
\end{enumerate}
\item $\beta_3\neq 0,$ then we consider following subcases.
\begin{enumerate}
	\item if $\beta_2\neq\beta_3,$ then choosing $y =- \frac{x\beta_2}{\beta_3},$ $t =- \frac{x\alpha_2+y\alpha_3}{\alpha_4},$ 
	we have $\alpha_2^*=\beta_2^*=0.$ Thus we can suppose $\alpha_2=\beta_2=0,$ moreover we may assume $\alpha_3=0$ and obtain \begin{longtable}{lcllcl}
$\alpha^*_1$&$=$&$x^2\alpha_1,$ & $\beta^*_1$&$=$&$x^2\beta_1,$ \\
$\alpha_4^*$&$=$&$x^3\alpha_4,$ & $\beta^*_3$&$=$&$x^2\beta_3.$ \end{longtable}

Then we have the representatives $ \left\langle \nabla_1 +\nabla_4,\alpha \nabla_1 + \nabla_3\right\rangle$ and $ \left\langle \nabla_4,\alpha \nabla_1 + \nabla_3\right\rangle$  depending on $\alpha_1=0$ or not. 

\item if $\beta_2=\beta_3,$ then we have 

\begin{longtable}{lcllcl}
$\alpha^*_1$&$=$&$x(x\alpha_1+y\alpha_2-y\alpha_3),$ & $\beta^*_1$&$=$&$x^2\beta_1,$ \\
$\alpha^*_2$&$=$&$(x+y)(x\alpha_2+y\alpha_3+t\alpha_4),$ & 
$\beta^*_2$&$=$&$(x+y)^2\beta_3,$\\
$\alpha^*_3$&$=$&$(x+y)((x+y)\alpha_3+t\alpha_4),$ &
$\beta^*_3$&$=$&$(x+y)^2\beta_3,$ \\
$\alpha_4^*$&$=$&$x(x+y)^2\alpha_4,$ & 
\end{longtable}

\begin{enumerate} \item if $\beta_1\neq 0,$ then choosing $t= \frac{(x+y)(x\alpha_1+y\alpha_2 - y\alpha_3)\beta_3 - x(x+y)\alpha_3\beta_1}{x\alpha_4\beta_1},$ $y=\frac{\sqrt{\beta_3} -\sqrt{\beta_1}}{\sqrt{\beta_1}},$ we may assume $\alpha_1=\alpha_3 =0,$
$\beta_1=\beta_2=\beta_3=1.$
Thus, we have 
\begin{longtable}{lcllcl}
$\alpha^*_1$&$=$&$0,$ & $\beta^*_1$&$=$&$x^2,$ \\
$\alpha^*_2$&$=$&$x^2\alpha_2,$ & 
$\beta^*_2$&$=$&$x^2,$\\
$\alpha^*_3$&$=$&$0$ & $\beta^*_3$&$=$&$x^2,$ \\
$\alpha_4^*$&$=$&$x^3\alpha_4.$ &  \end{longtable}

Then we have the representatives $ \left\langle \nabla_4, \nabla_1 + \nabla_2 +\nabla_3\right\rangle$ and $ \left\langle \nabla_2 +\nabla_4, \nabla_1 + \nabla_2 + \nabla_3\right\rangle$  depending on $\alpha_2=0$ or not. 

\item if $\beta_1= 0,$ then choosing $t = - \frac{(x+y)\alpha_3}{\alpha_4},$ we may suppose $\alpha_3=0$ and obtain 
\begin{longtable}{lcllcl}
$\alpha^*_1$&$=$&$x(x\alpha_1+y\alpha_2),$ & $\beta^*_1$&$=$&$0,$ \\
$\alpha^*_2$&$=$&$(x+y)x\alpha_2,$ & 
$\beta^*_2$&$=$&$(x+y)^2\beta_3,$\\
$\alpha^*_3$&$=$&$0,$ & $\beta^*_3$&$=$&$(x+y)^2\beta_3,$ \\
$\alpha_4^*$&$=$&$x(x+y)^2\alpha_4,$ &  \end{longtable}

\begin{enumerate}
	\item if $\alpha_1=\alpha_2 =0,$ then we have the representative $ \left\langle \nabla_4,\nabla_2 + \nabla_3\right\rangle;$

    \item if $\alpha_1=\alpha_2 \neq0,$ then choosing  
    $x=\frac{\alpha_2}{\alpha_4}$ and 
    $y=0$, we have the representative $ \left\langle \nabla_1 + \nabla_2 +\nabla_4, \nabla_2 + \nabla_3\right\rangle;$

     \item if $\alpha_2=0,$ $\alpha_1 \neq0,$ then choosing 
     $x=\frac{\alpha_1}{\alpha_4}$ and 
     $y=0$, 
      we have the representative $ \left\langle \nabla_1 + \nabla_4, \nabla_2 + \nabla_3\right\rangle;$

      \item if $\alpha_2\neq0,$ $\alpha_1 \neq\alpha_2,$  then choosing $x=\frac{\alpha_2^2}{(\alpha_2-\alpha_1)\alpha_4}$ and 
      $y=\frac{\alpha_1\alpha_2}{(\alpha_2-\alpha_1)\alpha_4},$ 
      we have the representative $ \left\langle \nabla_2 + \nabla_4, \nabla_2 + \nabla_3\right\rangle.$     
\end{enumerate}

\end{enumerate}

\end{enumerate}
\end{enumerate}

Summarizing, we have the following distinct orbits
\begin{center} 
$\left\langle \nabla_4,\nabla_1 \right\rangle,$  
$ \left\langle \nabla_2+\nabla_4,\nabla_1 \right\rangle,$  
$ \left\langle \nabla_4,\nabla_1 + \nabla_2\right\rangle,$  
$ \left\langle \nabla_1 + \nabla_4,\nabla_1 + \nabla_2\right\rangle,$  
$ \left\langle \nabla_4, \nabla_2\right\rangle,$  
$ \left\langle \nabla_1 + \nabla_4,\nabla_2\right\rangle,$ 
$ \left\langle \nabla_1 +\nabla_4,\alpha \nabla_1 + \nabla_3\right\rangle,$ 
$ \left\langle \nabla_4,\alpha \nabla_1 + \nabla_3\right\rangle,$  
$\left\langle \nabla_4,\nabla_1 + \nabla_2 + \nabla_3\right\rangle,$   
$ \left\langle \nabla_2 + \nabla_4, \nabla_1 + \nabla_2 + \nabla_3\right\rangle,$  
$ \left\langle \nabla_4,\nabla_2 + \nabla_3\right\rangle,$  
$ \left\langle \nabla_1 + \nabla_2 +\nabla_4, \nabla_2 + \nabla_3\right\rangle,$ 
  $\left\langle \nabla_1 + \nabla_4, \nabla_2 + \nabla_3\right\rangle,$  
  $ \left\langle \nabla_2 + \nabla_4, \nabla_2 + \nabla_3\right\rangle.$
\end{center}

Note that the orbit $\left\langle \nabla_4,\nabla_1 \right\rangle$ give the algebra isomorphic to the algebra $\mathbb{L}_{47}^{1,0}$, the orbit $\left\langle \nabla_4, \nabla_1 + \nabla_3\right\rangle$ 
give the algebra $\mathbb{L}_{63}^{0}$ and  
the orbit $\left\langle \nabla_4,\alpha \nabla_1 + \nabla_3\right\rangle_{\alpha \neq 0; 1}$ give the algebra $\mathbb{L}_{47}^{\alpha,0}$ in section \ref{sec1.3.7} and \ref{sec1.3.11}. Thus, we obtain following new algebras (see Section \ref{secteoA}):
\begin{longtable}{llllllllllllll}
$\mathbb{L}_{09},$ & $\mathbb{L}_{10},$ & 
$\mathbb{L}_{11},$ & 
$\mathbb{L}_{12},$ & 
$\mathbb{L}_{13},$ & 
$\mathbb{L}_{14}^{\alpha},$ & 
$\mathbb{L}_{15},$ & 
$\mathbb{L}_{16},$ & 
$\mathbb{L}_{17},$ & 
$\mathbb{L}_{18},$ & 
$\mathbb{L}_{19},$ & $\mathbb{L}_{20},$ & 
$\mathbb{L}_{21}.$ 
\end{longtable}
 

\subsection{$1$-dimensional central extensions of $4$-dimensional    nilpotent Leibniz algebras}

\subsubsection{ $4$-dimensional    nilpotent Leibniz algebras and their cohomology spaces}
Let us use the algebraic classification of $4$-dimensional nilpotent Leibniz algebras from the Corrigendum to \cite{kppv}.

\begin{longtable}{llllllll} 
\hline
\multicolumn{8}{c}{{\bf The list of 2-step nilpotent 4-dimensional Leibniz algebras}}  \\
\hline
 
{${\mathfrak N}_{01}$} &$:$ &  $e_1e_1 = e_2$ &&&&\\ 
\multicolumn{8}{l}{
${\rm H}^2({\mathfrak N}_{01})=
\Big\langle  
[\Delta_{13}],[\Delta_{14}],[\Delta_{21}],[\Delta_{31}],[\Delta _{33}],[\Delta _{34}],[\Delta _{41}],[\Delta _{43}],[\Delta _{44}]
\Big\rangle $}\\
 
\hline
{${\mathfrak N}_{02}$} &$:$ & $e_1e_1 = e_3$& $e_2e_2=e_4$  &&&  \\ 

\multicolumn{8}{l}{
${\rm H}^2({\mathfrak N}_{02})=
\Big\langle 
[\Delta_{12}],[\Delta_{21}],[\Delta_{31}],[\Delta _{42}]
\Big\rangle $}\\ 

\hline
{${\mathfrak N}_{03}$} &$:$ &  $e_1e_2=  e_3$ & $e_2e_1=-e_3$ &&& \\ 
\multicolumn{8}{l}{
${\rm H}_{\mathcal S}^2({\mathfrak N}_{03})=
\Big\langle 
[\Delta_{11}],[\Delta_{12}],[\Delta_{14}],[\Delta_{22}],[\Delta_{24}],[\Delta_{13}-\Delta_{31}],[\Delta_{23}-\Delta_{32}],[\Delta _{41}],[\Delta _{42}],[\Delta_{44}]
\Big\rangle $}\\
\multicolumn{8}{l}{${\rm H}^2({\mathfrak N}_{03})={\rm H}_{\mathcal S}^2({\mathfrak N}_{03})$}\\

\hline
${\mathfrak N}_{04}^{\alpha}$ &$:$ & $e_1e_1=  e_3$ & $e_1e_2=e_3$& $e_2e_2=\alpha e_3$  &&\\  

\multicolumn{8}{l}{
${\rm H}^2({\mathfrak N}_{04}^{\alpha\neq0})=
\Big\langle 
[\Delta_{12}],[\Delta_{14}],[\Delta_{21}],[\Delta_{22}],[\Delta_{24}],
[\Delta _{41}],[\Delta _{42}],[\Delta_{44}]
\Big\rangle $}\\

\multicolumn{8}{l}{
${\rm H}^2({\mathfrak N}_{04}^{0})=
\Big\langle 
[\Delta_{12}],[\Delta_{14}],[\Delta_{21}],[\Delta_{22}],[\Delta_{24}],
[\Delta_{31}+\Delta_{32}],[\Delta _{41}],[\Delta _{42}],[\Delta_{44}]
\Big\rangle $}\\

\hline
${\mathfrak N}_{05}$ &$:$ & $e_1e_1=  e_3$& $e_1e_2=e_3$&  $e_2e_1=e_3$ &&\\ 

\multicolumn{8}{l}{
${\rm H}^2({\mathfrak N}_{05})=
\Big\langle 
[\Delta_{12}],[\Delta_{14}],[\Delta_{21}],[\Delta_{22}],[\Delta_{24}],
[\Delta _{41}],[\Delta _{42}],[\Delta_{44}]
\Big\rangle $}\\

\hline
${\mathfrak N}_{06}$ &$:$ & $e_1e_2 = e_4$& $e_3e_1 = e_4$   &&&\\ 

\multicolumn{8}{l}{
${\rm H}^2({\mathfrak N}_{06})=
\Big\langle 
[\Delta_{11}],[\Delta_{13}],[\Delta_{21}],[\Delta_{22}],
[\Delta _ {23}],[\Delta_{31}], [\Delta _{32}],[\Delta_{33}]
\Big\rangle $}\\

\hline
{${\mathfrak N}_{07}$} &$:$ & $e_1e_2 = e_3$ & $e_2e_1 = e_4$ &  $e_2e_2 = -e_3$ &&\\ 

 \multicolumn{8}{l}{
${\rm H}^2({\mathfrak N}_{07})=
\Big\langle 
[\Delta_{11}],[\Delta_{22}],[\Delta_{32}],[\Delta_{41}]
\Big\rangle $}\\

\hline
${\mathfrak N}_{08}^{\alpha}$ &$:$ & $e_1e_1 = e_3$ & $e_1e_2 = e_4$ & $e_2e_1 = -\alpha e_3$ & $e_2e_2 = -e_4$& \\ 
 
\multicolumn{8}{l}{
${\rm H}^2({\mathfrak N}_{08}^{\alpha \neq 1 })=
\Big\langle 
[\Delta_{21}],[\Delta_{22}],[\Delta_{31}],[\Delta_{42}]
\Big\rangle $}\\

\multicolumn{8}{l}{
${\rm H}^2({\mathfrak N}_{08}^{1})=
\Big\langle 
[\Delta_{21}],[\Delta_{22}],[\Delta_{31}],[\Delta_{42}], [\Delta_{32}+\Delta_{41}]
\Big\rangle $}\\ 

\hline
${\mathfrak N}_{09}^{\alpha}$ &$:$ & $e_1e_1 = e_4$ & $e_1e_2 = \alpha e_4$ &  $e_2e_1 = -\alpha e_4$ & $e_2e_2 = e_4$ &  $e_3e_3 = e_4$\\

\multicolumn{8}{l}{
${\rm H}^2({\mathfrak N}_{09}^{\alpha})=
\Big\langle 
[\Delta_{12}],[\Delta_{13}],[\Delta_{21}],[\Delta_{22}],[\Delta_{23}],
[\Delta_{31}],[\Delta_{32}],[\Delta_{33}]
\Big\rangle $}\\

\hline
${\mathfrak N}_{10}$ &$:$ &  $e_1e_2 = e_4$ & $e_1e_3 = e_4$ & $e_2e_1 = -e_4$ & $e_2e_2 = e_4$ & $e_3e_1 = e_4$  \\ 

\multicolumn{8}{l}{
${\rm H}^2({\mathfrak N}_{10})=
\Big\langle 
[\Delta_{11}],[\Delta_{13}],[\Delta_{21}],[\Delta_{22}],[\Delta_{23}],
[\Delta_{31}],[\Delta_{32}],[\Delta_{33}]
\Big\rangle $}\\

\hline
${\mathfrak N}_{11}$ &$:$ &  $e_1e_1 = e_4$ & $e_1e_2 = e_4$ & $e_2e_1 = -e_4$ & $e_3e_3 = e_4$&  \\

\multicolumn{8}{l}{
${\rm H}^2({\mathfrak N}_{11})=
\Big\langle 
[\Delta_{12}],[\Delta_{13}],[\Delta_{21}],[\Delta_{22}],[\Delta_{23}],
[\Delta_{31}],[\Delta_{32}],[\Delta_{33}]
\Big\rangle $}\\

\hline
{${\mathfrak N}_{12}$} &$:$ &  $e_1e_2 = e_3$ & $e_2e_1 = e_4$  &&& \\ 

 \multicolumn{8}{l}{
${\rm H}^2({\mathfrak N}_{12})=
\Big\langle 
[\Delta_{11}],[\Delta_{22}],[-\Delta_{13}+\Delta_{14}+\Delta_{31}],
[\Delta_{32}],[\Delta _{41}],[\Delta_{23}-\Delta_{24}+\Delta_{42}]
\Big\rangle $}\\

\hline
${\mathfrak N}_{13}$ &$:$ & $e_1e_1 = e_4$ & $e_1e_2 = e_3$ & $e_2e_1 = -e_3$ & 
\multicolumn{2}{l}{$e_2e_2=2e_3+e_4$} \\
 
\multicolumn{8}{l}{
${\rm H}^2({\mathfrak N}_{13})=
\Big\langle 
[\Delta_{21}],[\Delta_{22}],[\Delta_{32}-\Delta_{41}],[\Delta_{31}-2\Delta_{32}+\Delta_{42}]
\Big\rangle $}\\

\hline
{${\mathfrak N}_{14}^{\alpha}$} &$:$ &   $e_1e_2 = e_4$ & $e_2e_1 =\alpha e_4$ & $e_2e_2 = e_3$&& \\ 

\multicolumn{8}{l}{
${\rm H}^2({\mathfrak N}_{14}^{\alpha\neq -1})=
\Big\langle 
[\Delta_{11}],[\Delta_{21}],[\Delta_{32}],[\alpha \Delta_{31}+\Delta_{42}]
\Big\rangle $}\\ 

\multicolumn{8}{l}{
${\rm H}^2({\mathfrak N}_{14}^{-1})=
\Big\langle 
[\Delta_{11}],[\Delta_{21}],[\Delta_{32}],[-\Delta_{31}+\Delta_{42}],
[\Delta_{14}-\Delta_{41}], [\Delta_{24}-\Delta_{42}]
\Big\rangle $}\\ 

\hline

${\mathfrak N}_{15}$ &$:$ &  $e_1e_2 = e_4$ & $e_2e_1 = -e_4$ & $e_3e_3 = e_4$ && \\

\multicolumn{8}{l}{
${\rm H}^2({\mathfrak N}_{15})=
\Big\langle 
[\Delta_{12}],[\Delta_{13}],[\Delta_{21}],[\Delta_{22}],[\Delta_{23}],
[\Delta_{31}],[\Delta_{32}],[\Delta_{33}]
\Big\rangle $}\\
\hline
\multicolumn{8}{c}{
\bf The list of 3-step nilpotent 4-dimensional Leibniz algebras} \\
\hline

$\mathfrak{L}_1$ &$:$ & $e_1e_1 = e_2$ &$ e_2e_1=e_3$\\ 

\multicolumn{8}{l}{
${\rm H}^2({\mathfrak L}_{1})=
\Big\langle 
[\Delta_{14}],[\Delta_{31}],[\Delta_{41}],[\Delta_{44}]
\Big\rangle $}\\

\hline
$\mathfrak{L}_2$ &$:$ & $e_1e_2 = e_3,$& $  e_1e_3 = e_4$ & $ e_2e_1 = -e_3$ & $ e_3e_1 = -e_4$\\ 

\multicolumn{8}{l}{
${\rm H}_{\mathcal S}^2({\mathfrak L}_{2})=
\Big\langle 
[\Delta_{11}],[\Delta_{21}],[\Delta_{22}],[\Delta_{23}-\Delta_{32}], [\Delta_{14}-\Delta_{41}]
\Big\rangle $}\\
\multicolumn{8}{l}{${\rm H}^2({\mathfrak L}_{2})={\rm H}_{\mathcal S}^2({\mathfrak L}_{2})$}\\
 
\hline
$\mathfrak{L}_3$ &$:$ &$e_1e_1 = e_3$ & $ e_1e_2 = e_4$ & $ e_2e_1 = e_3$ & $   e_3e_1 = e_4$\\ 

\multicolumn{8}{l}{
${\rm H}^2({\mathfrak L}_{3})=
\Big\langle 
[\Delta_{21}],[\Delta_{22}],[\Delta_{31}]
\Big\rangle $}\\

\hline

$\mathfrak{L}_4$ &$:$& $e_1e_1 = e_3$ &   $e_2e_1 = e_3$ &$ e_3e_1 = e_4$ \\

\multicolumn{8}{l}{
${\rm H}^2({\mathfrak L}_{4})=
\Big\langle 
[\Delta_{12}],[\Delta_{21}],[\Delta_{22}],[\Delta_{41}]
\Big\rangle $}\\

\hline

$\mathfrak{L}_5$ &$:$& $e_1e_1 = e_3$ &$  e_2e_1 = e_3$ & $ e_2e_2 = e_4$ & $ e_3e_1=e_4$\\

\multicolumn{8}{l}{
${\rm H}^2({\mathfrak L}_{5})=
\Big\langle 
[\Delta_{12}],[\Delta_{21}],[\Delta_{31}]
\Big\rangle $}\\

\hline
$\mathfrak{L}_6$ &$:$ &$e_1e_1 = e_3$ & $ e_1e_2 =  e_4$ & $ e_2e_1 = e_3$ & $ e_2e_2 = e_4$ & $e_3e_1=e_4$ \\ 

\multicolumn{8}{l}{
${\rm H}^2({\mathfrak L}_{6})=
\Big\langle 
[\Delta_{21}],[\Delta_{22}],[\Delta_{31}],[\Delta_{32}+\Delta_{41}]
\Big\rangle $}\\

\hline
$\mathfrak{L}_7$ &$:$ &$e_1e_1 = e_3$& $e_1e_2 = e_4$ & $e_3e_1 = e_4$  \\

\multicolumn{8}{l}{
${\rm H}^2({\mathfrak L}_{7})=
\Big\langle 
[\Delta_{21}],[\Delta_{22}],[\Delta_{31}],[\Delta_{32}+\Delta_{41}]
\Big\rangle $}\\

\hline
$\mathfrak{L}_8$ &$:$ &$e_1e_1 = e_3$ & $   e_2e_2 = e_4$ & $ e_3e_1=e_4$\\ 

\multicolumn{8}{l}{
${\rm H}^2({\mathfrak L}_{8})=
\Big\langle 
[\Delta_{12}],[\Delta_{21}],[\Delta_{31}]
\Big\rangle $}\\

\hline
$\mathfrak{L}_9$ &$:$&$e_1e_2=-e_3+e_4$& $e_1e_3=-e_4$ & $ e_2e_1 = e_3$& $ e_3e_1=e_4$ \\

\multicolumn{8}{l}{
${\rm H}^2({\mathfrak L}_{9})=
\Big\langle 
[\Delta_{11}],[\Delta_{21}],[\Delta_{22}],[\Delta_{23}-\Delta_{32}]
\Big\rangle $}\\

\hline
$\mathfrak{L}_{10}$ &$:$& $e_1e_2=-e_3$ & $e_1e_3=-e_4$ & $ e_2e_1 = e_3$ & $e_2e_2=e_4$ & $e_3e_1=e_4$\\ 

\multicolumn{8}{l}{
${\rm H}^2({\mathfrak L}_{10})=
\Big\langle 
[\Delta_{11}],[\Delta_{21}],[\Delta_{22}],[\Delta_{23}-\Delta_{32}]
\Big\rangle $}\\

\hline
$\mathfrak{L}_{11}$ &$:$ &$ e_1e_1 = e_4$ & $  e_1e_2=-e_3$ & $e_1e_3=-e_4$ & $  e_2e_1 = e_3$ & $e_2e_2=e_4$ & $ e_3e_1=e_4$ \\ 

\multicolumn{8}{l}{
${\rm H}^2({\mathfrak L}_{11})=
\Big\langle 
[\Delta_{21}],[\Delta_{22}],[\Delta_{13}-\Delta_{31}],[\Delta_{23}-\Delta_{32}]
\Big\rangle $}\\

\hline
$\mathfrak{L}_{12}$ &$:$ & $e_1e_1 = e_4$ & $  e_1e_2=-e_3 $& $ e_1e_3=-e_4$ & $ e_2e_1 = e_3$ & $   e_3e_1=e_4$\\ 

\multicolumn{8}{l}{
${\rm H}^2({\mathfrak L}_{12})=
\Big\langle 
[\Delta_{21}],[\Delta_{22}],[\Delta_{13}-\Delta_{31}],[\Delta_{23}-\Delta_{32}]
\Big\rangle $}\\

\hline

\multicolumn{8}{c}{\bf The list of 4-step nilpotent 4-dimensional Leibniz algebras}   \\
\hline
 
$\mathfrak{L}_{13}$ &$:$ &$e_1e_1 = e_2$ & $ e_2e_1 = e_3$ & $  e_3e_1 = e_4$ \\ 

\multicolumn{8}{l}{
${\rm H}^2(\mathfrak{L}_{13})=
\Big\langle 
 [\Delta_{41}]
\Big\rangle $}\\

\hline
\end{longtable}

\subsubsection{Central extensions of ${\mathfrak N}_{01}$}
	Let us use the following notations:
	\begin{longtable}{lllllll} 
	$\nabla_1 = [\Delta_{13}],$ & $\nabla_2 = [\Delta_{14}],$ &$\nabla_3 = [\Delta_{21}],$ & $\nabla_4 = [\Delta_{31}],$ & $\nabla_5 = [\Delta_{33}],$ \\
$\nabla_6 = [\Delta_{34}],$ & $\nabla_7 = [\Delta_{41}],$ &$\nabla_8 = [\Delta_{43}],$ & $\nabla_9 = [\Delta_{44}].$ 
	\end{longtable}	
	
Take $\theta=\sum\limits_{i=1}^9\alpha_i\nabla_i\in {\rm H^2}({\mathfrak N}_{01}).$
	The automorphism group of ${\mathfrak N}_{01}$ consists of invertible matrices of the form
	$$\phi=
	\begin{pmatrix}
	x &  0  & 0 & 0\\
	q &  x^2& r & u\\
	w &  0  & t & k\\
    z &  0  & y & l
	\end{pmatrix}.
	$$

 Since
	$$
	\phi^T\begin{pmatrix}
	0 & 0  & \alpha_1 & \alpha_2\\
	\alpha_3  & 0 & 0 & 0\\
	\alpha_4&  0    & \alpha_5 & \alpha_6\\
	\alpha_7&  0    & \alpha_8 & \alpha_9
	\end{pmatrix} \phi=	\begin{pmatrix}
	\alpha^* & 0  & \alpha_1^* & \alpha_2^*\\
	\alpha_3^*  & 0 & 0 & 0\\
	\alpha_4^*&  0    & \alpha_5^* & \alpha_6^*\\
	\alpha_7^*&  0    & \alpha_8^* & \alpha_9^*
	\end{pmatrix},
	$$
	 we have that the action of ${\rm Aut} ({\mathfrak N}_{01})$ on the subspace
$\langle \sum\limits_{i=1}^9\alpha_i\nabla_i  \rangle$
is given by
$\langle \sum\limits_{i=1}^9\alpha_i^{*}\nabla_i\rangle,$
where
\begin{longtable}{lcl}
$\alpha^*_1$&$=$&$t x \alpha _1+x y \alpha _2+t w \alpha _5+w y \alpha _6+t z \alpha _8+y z \alpha _9,$ \\
$\alpha^*_2$&$=$&$k x \alpha _1+l x \alpha _2+k w \alpha _5+l w \alpha _6+k z \alpha _8+l z \alpha _9,$ \\
$\alpha^*_3$&$=$&$x^3 \alpha _3,$ \\
$\alpha_4^*$&$=$&$r x \alpha _3+t x \alpha _4+t w \alpha _5+t z \alpha _6+x y \alpha _7+w y \alpha _8+y z \alpha _9,$\\
$\alpha_5^*$&$=$&$t^2 \alpha _5+y t(\alpha_6+\alpha_8)+y^2 \alpha_9,$\\
$\alpha_6^*$&$=$&$k t \alpha _5+l t \alpha _6+k y \alpha _8+l y \alpha _9,$\\
$\alpha_7^*$&$=$&$u x \alpha _3+k x \alpha _4+k w \alpha _5+k z \alpha _6+l x \alpha _7+l w \alpha _8+l z \alpha _9,$\\
$\alpha_8^*$&$=$&$k t \alpha _5+k y \alpha _6+l t \alpha _8+l y \alpha _9,$\\
$\alpha_9^*$&$=$&$k^2 \alpha _5+lk(\alpha_6+\alpha_8)+l^2\alpha_9.$\\
\end{longtable}

 We are interested only in the cases with 
 \begin{center}
$\alpha_3\neq 0,$ 
$(\alpha_1,\alpha_4,\alpha_5, \alpha_6,\alpha_8) \neq (0,0,0,0,0)$ 
and $(\alpha_2,\alpha_6,\alpha_7, \alpha_8, \alpha_9) \neq (0,0,0,0,0).$
 \end{center} 
 
It is easy to see that choosing 
\begin{center}$r = - \frac{t x \alpha _4+t w \alpha _5+t z \alpha _6+x y \alpha _7+w y \alpha _8+y z \alpha _9} {x \alpha_3}$ and
$u  = -\frac{k x \alpha _4+k w \alpha _5+k z \alpha _6+l x \alpha _7+l w \alpha _8+l z \alpha _9}{x \alpha _3},$ 
\end{center}
we have $\alpha^*_4=\alpha^*_7=0.$ Hence, we can suppose $\alpha_4=\alpha_7=0.$

The family of orbits $\langle\alpha_5\nabla_5+\alpha_6\nabla_6+\alpha_8\nabla_8+\alpha_9\nabla_9\rangle$ gives us characterised structure of three dimensional ideal whose a one dimensional extension of two dimensional subalgebra with basis $\{e_3, e_4\}.$ 

Using the classification of three dimensional nilpotent algebras, we may consider following cases.

\begin{enumerate}
	\item $\alpha_5=\alpha_6=\alpha_8=\alpha_9=0,$ i.e., three dimensional ideal is abelian. Then 
	we may suppose $\alpha_1 \neq 0$ and choosing $y=0,$ $l=\alpha_1,$ $k =-\alpha_2,$ we obtain that  $\alpha_2^*=0$ which implies $(\alpha_2^*,\alpha_6^*,\alpha_7^*, \alpha_8^*, \alpha_9^*) = (0,0,0,0,0).$ Thus, in this case we do not have new algebras. 
	
	\item $\alpha_5=1,$ $\alpha_6=\alpha_8=\alpha_9=0,$ i.e., three dimensional ideal is isomorphic to ${\mathcal N}_{01}$. Then $\alpha_2 \neq 0$ and choosing $x=1,$ $k=0,$ $y=0,$ $w = -\alpha_1,$ $l = \frac{\alpha_3}{\alpha_2}$ and  $t=\sqrt{\alpha_3},$ we have the representative 
	$\langle \nabla_2+ \nabla_3+\nabla_5\rangle.$

\item $\alpha_5=\alpha_9=0,$ $\alpha_6=1,$ $\alpha_8=-1,$ i.e., three dimensional ideal is isomorphic to ${\mathcal N}_{02}$. Then choosing 
$x=1,$ $k=0,$ $y=-1,$ $w = -\alpha_2,$ $z = \alpha_1,$ $l = \alpha_3$ and  $t=1,$ 
we have the representative 
	$\langle \nabla_3+ \nabla_6-\nabla_8\rangle.$
	
\item $\alpha_5=1, \alpha_6=1, \alpha_8=0, \alpha_9=\alpha,$ i.e., three dimensional ideal is isomorphic to ${\mathcal N}_{03}^{\alpha}$.

\begin{enumerate}

\item If  $\alpha \neq 0,$ then choosing $x=1,$ $k=0,$ $y=0,$ $w = -\alpha_1,$ $z = \frac{\alpha_1-\alpha_2}{\alpha_9},$
$l=t$ and $t=\sqrt{\alpha_3},$ we have the representative 
	$\langle \nabla_3+ \nabla_5+\nabla_6+\alpha\nabla_9\rangle_{\alpha\neq 0}.$

	\item If $\alpha = 0$  
and  $\alpha_1=\alpha_2,$ then choosing $x=1,$ $k=0,$ $y=0,$ $w = -\alpha_2,$ $l=\sqrt{\alpha_3}$ and $t=\sqrt{\alpha_3},$ we have the representative 
	$\langle \nabla_3+ \nabla_5+\nabla_6\rangle.$ 

	\item If $\alpha = 0$  
and $\alpha_1\neq \alpha_2,$ then choosing 
$x=\frac{(\alpha_1-\alpha_2)^2}{\alpha_3},$  $y=0,$ $z=0,$ 
$w = -\frac{\alpha_2(\alpha_1-\alpha_2)^2}{\alpha_3},$ $l=\frac{(\alpha_1-\alpha_2)^2}{\alpha_3},$  
$u=0,$
$k=0$ and
$t=\frac{(\alpha_1-\alpha_2)^2}{\alpha_3},$
we have the representative 
	$\langle \nabla_1+\nabla_3+ \nabla_5+\nabla_6\rangle.$ 
\end{enumerate}

\item $\alpha_5=1, \alpha_6=1, \alpha_8=1, \alpha_9=0,$ i.e., three dimensional ideal is isomorphic to ${\mathcal N}_{04}$. Then choosing $x=1,$ $k=0,$ $y=0,$ $w = -\alpha_2,$ $z = \alpha_2-\alpha_1,$
$l=\sqrt{\alpha_3}$ and $t=\sqrt{\alpha_3},$ we have the representative 
	$\langle \nabla_3+ \nabla_5+\nabla_6 +\nabla_8\rangle.$

 \end{enumerate}
 
Summarizing, we have the following distinct orbits
\begin{center} $\langle \nabla_2+ \nabla_3+\nabla_5\rangle,$ 
$\langle \nabla_3+ \nabla_6-\nabla_8\rangle,$
$\langle \nabla_3+ \nabla_5+\nabla_6+\alpha\nabla_9\rangle,$ 
$\langle \nabla_1+\nabla_3+ \nabla_5+\nabla_6\rangle,$ 
$\langle \nabla_3+ \nabla_5+\nabla_6 +\nabla_8\rangle,$
\end{center}
which gives the following new algebras (see Section \ref{secteoA}):
$\mathbb{L}_{22},$ 
$\mathbb{L}_{23},$  
$\mathbb{L}_{24}^\alpha,$  
$\mathbb{L}_{25}$ and 
$\mathbb{L}_{26}.$


 \subsubsection{Central extensions of ${\mathfrak N}_{02}$}
	Let us use the following notations:
	\begin{longtable}{lllllll} $\nabla_1 = [\Delta_{12}],$ & $\nabla_2 = [\Delta_{21}],$ &$\nabla_3 = [\Delta_{31}],$ & $\nabla_4 = [\Delta_{42}].$
	\end{longtable}	
	
Take $\theta=\sum\limits_{i=1}^4\alpha_i\nabla_i\in {\rm H^2}({\mathfrak N}_{02}).$
	The automorphism group of ${\mathfrak N}_{02}$ consists of invertible matrices of the form
	$$\phi_1=
	\begin{pmatrix}
	x &  0  & 0 & 0\\
	0 &  y  & 0 & 0\\
	z &  u  & x^2 & 0\\
    t &  v  & 0 & y^2
	\end{pmatrix}, \quad \phi_2=
	\begin{pmatrix}
	0 &  x  & 0 & 0\\
	y &  0  & 0 & 0\\
	z &  u  & 0 & x^2\\
    t &  v  & y^2 & 0
	\end{pmatrix}.
	$$

 Since
	$$
	\phi_1^T\begin{pmatrix}
	0 & \alpha_1  & 0 & 0\\
	\alpha_2  & 0 & 0 & 0\\
	\alpha_3 &  0    & 0 & 0\\
	0 &  \alpha_4    & 0 & 0
	\end{pmatrix} \phi_1=	\begin{pmatrix}
	\alpha^* & \alpha_1^*  & 0 & 0\\
	\alpha_2^*  & \alpha ^{**} & 0 & 0\\
	\alpha_3^* &  0    & 0 & 0\\
	0 &  \alpha_4^*    & 0 & 0
	\end{pmatrix},
	$$
	 we have that the action of ${\rm Aut} ({\mathfrak N}_{02})$ on the subspace
$\langle \sum\limits_{i=1}^4\alpha_i\nabla_i  \rangle$
is given by
$\langle \sum\limits_{i=1}^4\alpha_i^{*}\nabla_i\rangle,$
where
\begin{longtable}{lcllcl}
$\alpha^*_1$&$ =$&$ y (x\alpha _1+t\alpha _4),$ &
$\alpha^*_2$&$ =$&$ x (y\alpha _2+u\alpha _3),$ \\
$\alpha^*_3$&$ =$&$ x^3 \alpha _3,$ &
$\alpha_4^*$&$ =$&$ y^3 \alpha _4.$
\end{longtable}
We are interested only in the cases with $\alpha_3\alpha_4\neq 0.$ Then putting 
$x=\sqrt[3]{\alpha_3^2}\alpha_4,$
$y=\alpha_3\sqrt[3]{\alpha_4^2},$
$t=-\alpha_1\sqrt[3]{\alpha_3^2}$
and
$u=-\alpha_2\sqrt[3]{\alpha_4^2},$
we have the representative
 $\langle \nabla_3+\nabla_4\rangle.$ Hence we obtain the algebra $\mathbb{L}_{27}.$ 
 
 
  \subsubsection{Central extensions of ${\mathfrak N}_{04}^{0}$}
	Let us use the following notations:
	\begin{longtable}{lllllll} 
	$\nabla_1=[\Delta_{12}],$ & $\nabla_2 = [\Delta_{14}],$ & $\nabla_3 = [\Delta_{21}],$ & $\nabla_4 = [\Delta_{22}],$ &
	$\nabla_5=[\Delta_{24}],$\\
	$\nabla_6=[\Delta_{31}+\Delta_{32}],$ & $\nabla_7=[\Delta_{41}],$ & $\nabla_8=[\Delta_{42}],$ & $\nabla_9=[\Delta_{44}].$ \\
	\end{longtable}	
	
Take $\theta=\sum\limits_{i=1}^9\alpha_i\nabla_i\in {\rm H^2}({\mathfrak N}_{04}^{0}).$
	The automorphism group of ${\mathfrak N}_{04}^{0}$ consists of invertible matrices of the form
	$$\phi=
	\begin{pmatrix}
	x &  0  & 0 & 0\\
	y &  x+y  & 0 & 0\\
	z &  t  & x(x+y) & w\\
    u &  v  & 0 & r
	\end{pmatrix}.
	$$

 Since
	$$
	\phi^T\begin{pmatrix}
	0 & \alpha_1  & 0 & \alpha_2\\
	\alpha_3 &  \alpha_4 & 0 & \alpha_5\\
	\alpha_6 &  \alpha_6 & 0 & 0\\
	\alpha_7 &  \alpha_8   & 0 & \alpha_9
	\end{pmatrix} \phi=	\begin{pmatrix}
    \alpha^* & \alpha_1^*+\alpha^*  & 0 & \alpha_2^*\\
	\alpha_3^* &  \alpha_4^* & 0 & \alpha_5^*\\
	\alpha_6^* &  \alpha_6^* & 0 & 0\\
	\alpha_7^* &  \alpha_8^*   & 0 & \alpha_9^*
	\end{pmatrix},
	$$
	 we have that the action of ${\rm Aut} ({\mathfrak N}_{04}^{0})$ on the subspace
$\langle \sum\limits_{i=1}^9\alpha_i\nabla_i  \rangle$
is given by
$\langle \sum\limits_{i=1}^9\alpha_i^{*}\nabla_i\rangle,$
where
\begin{longtable}{lcl}
$\alpha^*_1$&$=$&$x^2\alpha_1-x(u-v)\alpha_2-xy(\alpha_3-\alpha_4)-y(u-v)\alpha_5-ux(\alpha_7-\alpha_8)-u(u-v)\alpha_9,$ \\
$\alpha^*_2$&$=$&$r(x\alpha_2+y\alpha_5+u \alpha_9),$ \\
$\alpha^*_3$&$=$&$x((x+y)\alpha_3+t\alpha_6+v\alpha_7)+y((x+y)\alpha_4+
t\alpha_6+v\alpha_8)+u((x+y)\alpha_5+v\alpha_9),$ \\
$\alpha_4^*$&$=$&$(x+y)((x+y)\alpha_4+t\alpha_6+v\alpha_8) +v((x+y)\alpha_5+v\alpha_9)$\\
$\alpha_5^*$&$=$&$r((x+y)\alpha_5+v\alpha_9),$\\
$\alpha_6^*$&$=$&$x (x+y)^2\alpha_6,$\\
$\alpha_7^*$&$=$&$w (x+y)\alpha_6+r(x\alpha_7+y\alpha_8+u\alpha_9),$\\
$\alpha_8^*$&$=$&$(x+y)(w\alpha_6+r\alpha_8)+rv\alpha_9,$\\
$\alpha_9^*$&$=$&$r^2 \alpha_9.$\\
\end{longtable}

We are interested only in the cases with $\alpha_6\neq 0$ and  $(\alpha_2,\alpha_5,\alpha_7,\alpha_8,\alpha_9)\neq(0,0,0,0,0).$ Then  choosing 
\begin{center}
$t=-\frac{x (x+y)\alpha_3+y(x+y)\alpha _4+ux\alpha_5+uy\alpha_5+vx \alpha_7+v y \alpha_8+uv\alpha_9}{(x+y)\alpha_6},\ w=-\frac{r(x\alpha_7+y\alpha_8+u \alpha_9)}{(x+y)\alpha_6},$     
\end{center}
we have $\alpha_3^*=\alpha_7^*=0$. Thus, without loss of generality, we can suppose $\alpha_3=\alpha_7=0$ and we
have the following 

\begin{longtable}{lcl}
$\alpha^*_1$&$=$&$x^2\alpha_1-(u-v)(x\alpha_2+y\alpha_5+u\alpha_9)+x(y\alpha_4+u\alpha_8),$ \\
$\alpha^*_2$&$=$&$r(x\alpha_2+y\alpha_5+u \alpha_9),$ \\
$\alpha_4^*$&$=$&$x(x+y)\alpha_4-(u-v)((x+y)\alpha_5+v\alpha_9)+vx\alpha_8$\\
$\alpha_5^*$&$=$&$r((x+y)\alpha_5+v\alpha_9),$\\
$\alpha_6^*$&$=$&$x (x+y)^2\alpha_6,$\\
$\alpha_8^*$&$=$&$r(x\alpha_8-(u-v)\alpha_9),$\\
$\alpha_9^*$&$=$&$r^2 \alpha_9.$\\
\end{longtable}

\begin{enumerate}
\item $\alpha_9\neq0$, then choosing $u=\frac{x\alpha_8-(x+y)\alpha_5}{\alpha_9},\ v=-\frac{(x+y) \alpha_5}{\alpha_9},$ we have $\alpha_5^*=\alpha_8^*=0$ and without loss of generality, we can get $\alpha_5=\alpha_8=0.$ Thus, we get
\begin{longtable}{lcllcllcl}
$\alpha^*_1$&$=$&$x(x\alpha_1+y\alpha_4),$ &
$\alpha^*_2$&$=$&$r x\alpha_2,$ &
$\alpha_4^*$&$=$&$x(x+y)\alpha_4,$\\
$\alpha_6^*$&$=$&$x (x+y)^2\alpha_6,$ &
$\alpha_9^*$&$=$&$r^2 \alpha_9.$\\
\end{longtable}

\begin{enumerate}
	\item Let $\alpha_4\neq 0,$ then we consider following cases: 
		\begin{enumerate}
	    \item if $\alpha_1\neq\alpha_4,$ then choosing $x=\frac{\alpha_4^2}{(\alpha_4-\alpha_1)\alpha_6},\  y=\frac{\alpha_1\alpha_4}{(\alpha_1-\alpha_4)\alpha_6},\ r=\frac{\alpha_4^2}{\alpha_6\sqrt{(\alpha_4-\alpha_1)\alpha_9}},$ we have the family of representatives $\langle \alpha \nabla_2+\nabla_4+\nabla_6+\nabla_9\rangle;$
	    \item if $\alpha_1=\alpha_4,$ then we have the following cases:
	    	    \begin{enumerate}
	        \item if $\alpha_2\neq0$ then choosing $x=\frac{\alpha_4^2\alpha_9}{\alpha_2^2\alpha_6},\ y=\frac{\alpha_4(\alpha_2^2-\alpha_1\alpha_9)}{\alpha_2^2\alpha_6},\ r=\frac{\alpha_4^2}{\alpha_2\alpha_6},$ we have the representative $\langle \nabla_1+\nabla_2+\nabla_4+\nabla_6+\nabla_9\rangle;$
	        \item if $\alpha_2=0$ then choosing $x=\frac{\alpha_4}{\alpha_6},\ y=0,\ r=\frac{\alpha_4}{\alpha_6}\sqrt{\frac{\alpha_4}{\alpha_9}},$ we have the representative $\langle \nabla_1+\nabla_4+\nabla_6+\nabla_9\rangle.$
	    \end{enumerate}
	\end{enumerate}
	
	\item Let $\alpha_4=0,$ then we consider following cases:
\begin{enumerate}
    \item if $\alpha_1=0,$ then we have:
    \begin{enumerate}
        \item if $\alpha_2=0,$ then choosing $x=1,\ y=0,\ r=\sqrt{\frac{\alpha_6}{\alpha_9}},$ we have the representative $\langle \nabla_6+\nabla_9\rangle;$
        \item if $\alpha_2\neq0,$ then choosing $x=\frac{\alpha_2^2}{\alpha_6\alpha_9},\ y=0,\ r=\frac{\alpha_2^3}{\alpha_6\alpha_9^2},$ we have the representative $\langle \nabla_2+\nabla_6+\nabla_9\rangle;$
    \end{enumerate}
    \item $\alpha_1\neq0,$ then choosing $x=\frac{\alpha_1}{\alpha_6},\ y=0,\ r=\frac{\alpha_1}{\alpha_6}\sqrt{\frac{\alpha_1}{\alpha_9}},$  we have the family of representatives $\langle \nabla_1+\alpha \nabla_2+\nabla_6+\nabla_9\rangle.$
\end{enumerate}

\end{enumerate}	

\item $\alpha_9=0,\ \alpha_8\neq0,$ then $\alpha_9^*=0$ and we have the following  
\begin{longtable}{lcllcllcl}
$\alpha^*_1$&$=$&$x^2\alpha_1-(u-v)(x\alpha_2+y\alpha_5)+x(y\alpha_4+u\alpha_8),$ &
$\alpha^*_2$&$=$&$r(x\alpha_2+y\alpha_5),$ \\
$\alpha_4^*$&$=$&$x(x+y)\alpha_4-(u-v)(x+y)\alpha_5+vx\alpha_8$&
$\alpha_5^*$&$=$&$r(x+y)\alpha_5,$\\
$\alpha_6^*$&$=$&$x(x+y)^2\alpha_6,$&
$\alpha_8^*$&$=$&$r x\alpha_8.$\\
\end{longtable}

\begin{enumerate}
\item Let  $\alpha_2=\alpha_5$, then:

\begin{enumerate}
    \item if $\alpha_2=0,$ then choosing $x=1,$ $y=0,$ $u=-\frac{\alpha_1}{\alpha_8},\ v=-\frac{\alpha_4}{\alpha_8},  r=\frac{\alpha_6}{\alpha_8},$ we have the representative 
    $\langle \nabla_6+\nabla_8\rangle;$
        \item if $\alpha_2\neq0,$ then choosing 
       $x=\alpha_2 \alpha_8,$
        $y=\alpha_8(\alpha_8-\alpha_2),$ 
        $r= \alpha_6\alpha_8^3,$ 
        $u=-2\alpha_1\alpha_2+2\alpha_2\alpha_4-\alpha_4\alpha_8$
        and $v=-\alpha_1\alpha_2+\alpha_4(\alpha_2-\alpha_8),$ 
        we have  the representative 
    $\langle \nabla_2+\nabla_5+\nabla_6+\nabla_8\rangle.$
\end{enumerate}

\item  Let $\alpha_2\neq\alpha_5$, then:

\begin{enumerate}
    \item if $\alpha_5=0,$ then $\alpha_2\neq0$ and putting $v=-\frac{(x+y)\alpha_4}{\alpha_8},$ we 
    can suppose $\alpha_4=0.$ Thus, consider following cases:
    
\begin{enumerate}
\item if $\alpha_2=\alpha_8,\ \alpha_1=0,$ then putting $x=1,\ y=0,\ r=\frac{\alpha_6}{\alpha_2},$ we have the representative 
$\langle \nabla_2+\nabla_6+\nabla_8\rangle;$

\item if $\alpha_2=\alpha_8,\ \alpha_1\neq0,$ then putting $x=\frac{\alpha_1}{\alpha_6},\ y=0,\  r=\frac{\alpha_1^2}{\alpha_2\alpha_6},$ we have the representative 
$\langle \nabla_1+\nabla_2+\nabla_6+\nabla_8\rangle.$

    \item if $\alpha_2\neq\alpha_8,$ then putting $u=\frac{\alpha_1}{\alpha_2-\alpha_8},\ x=1,\ y=0,\ r=\frac{\alpha_6}{\alpha_8},$ we have the family of representatives 
$\langle \alpha\nabla_2+\nabla_6+\nabla_8\rangle_{\alpha\neq1};$
\end{enumerate}

\item if $\alpha_5\neq 0,$ then putting 
$u=\frac{x\alpha_4}{\alpha_5}+\frac{v(\alpha_2-(\alpha_8+\alpha_5))}{\alpha_2-\alpha_5},$  
$y=-\frac{x \alpha _2}{\alpha_5},$ we can suppose 
$\alpha_2=0,$ $\alpha_4=0$. Consider following subcases: 
\begin{enumerate}
    \item if $\alpha_8=-\alpha_5,\ \alpha_1=0,$ then choosing $x=1,\ r=\frac{\alpha_6}{\alpha_8},$ we have the representative $\langle -\nabla_5+\nabla_6+\nabla_8\rangle;$

    \item if $\alpha_8=-\alpha_5,\ \alpha_1\neq0,$ then choosing $x=\frac{\alpha_1}{\alpha_6},\ r=\frac{\alpha_1^2}{\alpha_6\alpha_8},$ we have the representative $\langle \nabla_1-\nabla_5+\nabla_6+\nabla_8\rangle;$

    \item if $\alpha_8\neq -\alpha_5,$ then choosing 
    $x=1,$  
    $v=-\frac{\alpha_1\alpha_5}{\alpha_8(\alpha_5+\alpha_8)}$ and
    $r=\frac{\alpha_6}{\alpha_8},$ we have the family of representatives $\langle \alpha\nabla_5+\nabla_6+\nabla_8\rangle_{\alpha\neq-1}.$
\end{enumerate}

\end{enumerate}

\end{enumerate}

\item $\alpha_9=0,\ \alpha_8=0, \alpha_5\neq0,$ then $\alpha_9^*=\alpha_8^*=0$ and
putting $u=v+\frac{x\alpha_4}{\alpha_5},$ we have $\alpha_4^*=0.$ Thus, without lost the generality we can suppose $\alpha_4=0$ and 
we have the following  
\begin{longtable}{lcllcl}
$\alpha^*_1$&$=$&$x^2\alpha_1,$ &
$\alpha^*_2$&$=$&$r(x\alpha_2+y\alpha_5),$ \\
$\alpha_5^*$&$=$&$r(x+y)\alpha_5,$ &
$\alpha_6^*$&$=$&$x(x+y)^2\alpha_6.$\\
\end{longtable}

\begin{enumerate}
    \item If  $\alpha_2=\alpha_5,\ \alpha_1=0$, then choosing $x=1,\ y=0,\  r=\frac{\alpha_6}{\alpha_5},$ we have the representative 
    $\langle \nabla_2+\nabla_5+\nabla_6\rangle.$

    \item If $\alpha_2=\alpha_5,\ \alpha_1\neq0,$ then choosing $x=\frac{\alpha_1}{\alpha_6},\ y=0,\ r=\frac{\alpha_1^2}{\alpha_5\alpha_6},$ we have  the representative 
    $\langle \nabla_1+\nabla_2+\nabla_5+\nabla_6\rangle.$
    
    \item If $\alpha_2\neq \alpha_5,\ \alpha_1=0,$ then choosing $x=1,\ y=-\frac{\alpha_2}{\alpha_5},\ r=\frac{(\alpha_5-\alpha_2)\alpha_6}{\alpha_5},$ we have  the representative $\langle \nabla_5+\nabla_6\rangle.$

    \item If $\alpha_2\neq\alpha_5,\ \alpha_1\neq0,$ then choosing $x=\frac{\alpha_1\alpha_5^2}{(\alpha_5-\alpha_2)^2\alpha_6},\ y=-\frac{\alpha_1\alpha_2\alpha_5}{(\alpha_5-\alpha_2)^2\alpha_6},\ r=\frac{\alpha_1^2\alpha_5^2}{(\alpha_5-\alpha_2)^3\alpha_6},$ we have  the representative $\langle \nabla_1+\nabla_5+\nabla_6\rangle.$
\end{enumerate}

\item $\alpha_9=0,\ \alpha_8=0, \alpha_5=0,$ then $\alpha_2\neq0,\ \alpha_9^*=\alpha_8^*=\alpha_5^*=0$ and we have the following 
\begin{longtable}{lcllcl}
$\alpha^*_1$&$=$&$x(x\alpha_1-(u-v)\alpha_2+y\alpha_4),$ &
$\alpha^*_2$&$=$&$x r\alpha_2,$ \\
$\alpha_4^*$&$=$&$x(x+y)\alpha_4$ &
$\alpha_6^*$&$=$&$x(x+y)^2\alpha_6.$\\
\end{longtable}
Then putting $u=\frac{x\alpha_1+v\alpha_2+y\alpha_4}{\alpha_2},$ we have $\alpha_1^*=0.$

\begin{enumerate}
    \item If  $\alpha_4=0$, then choosing $x=1,\ y=0,\  r=\frac{\alpha_6}{\alpha_2},$ we have the representative 
    $\langle \nabla_2+\nabla_6\rangle.$
    \item If $\alpha_4\neq0,$ then choosing $x=\frac{\alpha_4}{\alpha_6},\ y=0,\ r=\frac{\alpha_4^2}{\alpha_2\alpha_6},$ we have the representative 
    $\langle \nabla_2+\nabla_4+\nabla_6\rangle.$
\end{enumerate}

\end{enumerate}

Summarizing all cases, we have the following distinct orbits 

\begin{center}  
$\langle \alpha\nabla_2+ \nabla_4+\nabla_6+\nabla_9\rangle,$ 
$\langle \nabla_1+\nabla_4+\nabla_6+\nabla_9\rangle,$
$\langle \nabla_1+\nabla_2+\nabla_4+\nabla_6+\nabla_9\rangle,$ 
$\langle \nabla_6+\nabla_9\rangle,$ 
$\langle \nabla_2+ \nabla_6+\nabla_9\rangle,$ 
$\langle \nabla_1+\alpha \nabla_2+\nabla_6+\nabla_9\rangle,$ 
$\langle \nabla_6+\nabla_8\rangle,$ 
$\langle \nabla_2+\nabla_5+\nabla_6+\nabla_8\rangle,$ 
$\langle \alpha\nabla_2+\nabla_6+\nabla_8\rangle,$ 
$\langle \nabla_1+\nabla_2+\nabla_6+\nabla_8\rangle,$ 
$\langle \alpha\nabla_5+\nabla_6+\nabla_8\rangle,$ 
$\langle \nabla_1-\nabla_5+\nabla_6+\nabla_8\rangle,$
$\langle \nabla_2+\nabla_5+\nabla_6\rangle,$ 
$\langle \nabla_1+\nabla_2+\nabla_5+\nabla_6\rangle,$ 
$\langle \nabla_5+\nabla_6\rangle,$
$\langle \nabla_1+\nabla_5+\nabla_6\rangle,$ 
$\langle \nabla_2+\nabla_6\rangle,$ 
$\langle \nabla_2+\nabla_4+\nabla_6\rangle,$
\end{center}
which give the following new algebras (see Section \ref{secteoA}):

\begin{longtable}{llllllllllllllllll}
$\mathbb{L}_{28}^{\alpha},$ &
$\mathbb{L}_{29},$ & $\mathbb{L}_{30},$ &
$\mathbb{L}_{31},$ & 
$\mathbb{L}_{32},$ & 
$\mathbb{L}_{33}^{\alpha},$ & $\mathbb{L}_{34},$ & $\mathbb{L}_{35},$ & 
$\mathbb{L}_{36}^\alpha,$ \\
$\mathbb{L}_{37},$ & 
$\mathbb{L}_{38}^{\alpha},$ & 
$\mathbb{L}_{39},$ & 
$\mathbb{L}_{40},$ & 
$\mathbb{L}_{41},$ & 
$\mathbb{L}_{42},$ & 
$\mathbb{L}_{43},$ & 
$\mathbb{L}_{44},$ & 
$\mathbb{L}_{45}.$ & 
\end{longtable}

 \subsubsection{Central extensions of ${\mathfrak N}_{07}$}
	Let us use the following notations:
	\begin{longtable}{lllllll} $\nabla_1 = [\Delta_{11}],$ & $\nabla_2 = [\Delta_{22}],$ &$\nabla_3 = [\Delta_{32}],$ & $\nabla_4 = [\Delta_{41}].$
	\end{longtable}	
	
Take $\theta=\sum\limits_{i=1}^4\alpha_i\nabla_i\in {\rm H^2}({\mathfrak N}_{07}).$
	The automorphism group of ${\mathfrak N}_{07}$ consists of invertible matrices of the form
	$$\phi=
	\begin{pmatrix}
	x &  0  & 0 & 0\\
	0 &  y  & 0 & 0\\
	z &  u  & xy & 0\\
    t &  v  & 0 & xy
	\end{pmatrix}.
	$$

 Since
	$$
	\phi^T\begin{pmatrix}
	\alpha_1 & 0  & 0 & 0\\
	0 &  \alpha_2 & 0 & 0\\
	0 &   \alpha_3   & 0 & 0\\
	\alpha_4 &  0   & 0 & 0
	\end{pmatrix} \phi=	\begin{pmatrix}
	\alpha_1^* & \alpha^*  & 0 & 0\\
	\alpha^{**} &  \alpha_2^*-\alpha^* & 0 & 0\\
	0 &   \alpha_3^*   & 0 & 0\\
	\alpha_4^* &  0   & 0 & 0
	\end{pmatrix},
	$$
	 we have that the action of ${\rm Aut} ({\mathfrak N}_{07})$ on the subspace
$\langle \sum\limits_{i=1}^4\alpha_i\nabla_i  \rangle$
is given by
$\langle \sum\limits_{i=1}^4\alpha_i^{*}\nabla_i\rangle,$
where
\begin{longtable}{lcllcl}
$\alpha^*_1$&$=$&$x(x \alpha _1+t \alpha_4),$ &
$\alpha^*_2$&$=$&$y(y\alpha _2+(z+u) \alpha_3),$ \\
$\alpha^*_3$&$=$&$x y^2 \alpha _3,$&
$\alpha_4^*$&$=$&$x^2 y \alpha _4.$\\
\end{longtable}

We are interested only in the cases with $\alpha_3\alpha_4\neq 0.$
 Then  choosing 
\begin{center}$x=\alpha_3,$  $y=\alpha_4,$ $z=0,$
$t=-\frac{\alpha_1\alpha_3}{\alpha_4}$ and $u=-\frac{\alpha_2\alpha_4}{\alpha_3}$, \end{center}
we have the representative
 $\langle \nabla_3+\nabla_4\rangle$
 and obtain the algebra
  \begin{longtable}{llllllllllllll}
$\mathbb{L}_{46} : $ & $e_1e_2=e_3$ & $e_2e_1=e_4$ & $e_2e_2=-e_3$ & $e_3e_2=e_5$ & $e_4e_1=e_5.$
\end{longtable}

\subsubsection{Central extensions of ${\mathfrak N}_{08}^{\alpha \neq 1}$}
	Let us use the following notations:
	\begin{longtable}{lllllll} $\nabla_1 = [\Delta_{21}],$ & $\nabla_2 = [\Delta_{22}],$ &$\nabla_3 = [\Delta_{31}],$ & $\nabla_4 =[\Delta_{42}].$ 
	\end{longtable}	
	
Take $\theta=\sum\limits_{i=1}^4\alpha_i\nabla_i\in {\rm H^2}({\mathfrak N}_{08}^{\alpha \neq 1}).$
	The automorphism group of ${\mathfrak N}_{08}^{\alpha \neq 1}$ consists of invertible matrices of the form
	$$\phi_1=
	\begin{pmatrix}
	x &  0  & 0 & 0\\
	0 &  x  & 0 & 0\\
	t &  v  & x^2 & 0\\
    u &  w  & 0 & x^2
	\end{pmatrix}, \quad 
	\phi_2(\alpha \neq 0)=
	\begin{pmatrix}
	0 &  \alpha x  & 0 & 0\\
	x &  0  & 0 & 0\\
	t &  v  & 0 & -\alpha^2x^2\\
    u &  w  & -x^2 & 0
	\end{pmatrix}.
	$$

 Since
	$$
	\phi_1^T\begin{pmatrix}
	0 & 0  & 0 & 0\\
	\alpha_1 &  \alpha_2 & 0 & 0\\
	\alpha_3 &  0 & 0 & 0\\
	0 &  \alpha_4   & 0 & 0
	\end{pmatrix} \phi_1=	\begin{pmatrix}
    \alpha^* & \alpha^{**}  & 0 & 0\\
	\alpha_1-\alpha\alpha^* &  \alpha_2^*-\alpha^{**} & 0 & 0\\
	\alpha_3^* &  0 & 0 & 0\\
	0 &  \alpha_4^*   & 0 & 0
	\end{pmatrix},
	$$
	 we have that the action of ${\rm Aut} ({\mathfrak N}_{08}^{\alpha \neq 1})$ on the subspace
$\langle \sum\limits_{i=1}^4\alpha_i\nabla_i  \rangle$
is given by
$\langle \sum\limits_{i=1}^4\alpha_i^{*}\nabla_i\rangle,$
where
\begin{longtable}{lcllcl}
$\alpha^*_1 $&$=$&$x(x \alpha _1+(v+t \alpha ) \alpha _3),$ &
$\alpha^*_2 $&$=$&$x(x \alpha _2+(w+u \alpha ) \alpha _4),$ \\
$\alpha^*_3 $&$=$&$x^3 \alpha_3,$ &
$\alpha_4^* $&$=$&$x^3 \alpha_4.$\\
\end{longtable}

We are interested only in the cases with $\alpha_3\alpha_4\neq 0$. 
Then  choosing 
\begin{center}
$x=\alpha_3,$ 
$w=0$,
$t=0,$
$v=-\alpha_1$,
and
$u=-\alpha_2\alpha_3\alpha_4^{-1}$,\end{center}
we have the representative
$\langle \nabla_3+\lambda \nabla_4\rangle_{\lambda \neq  0}.$

 \subsubsection{Central extensions of ${\mathfrak N}_{08}^{1}$} \label{sec1.3.7}
	Let us use the following notations:
	\begin{longtable}{lllllll} $\nabla_1 = [\Delta_{21}],$ & $\nabla_2 = [\Delta_{22}],$ &$\nabla_3 = [\Delta_{31}],$ & $\nabla_4 = [\Delta_{42}],$ 
	& $\nabla_5 = [\Delta_{32}+\Delta_{41}].$
	\end{longtable}	
	
Take $\theta=\sum\limits_{i=1}^5\alpha_i\nabla_i\in {\rm H^2}({\mathfrak N}_{08}^{1}).$
	The automorphism group of ${\mathfrak N}_{08}^{1}$ consists of invertible matrices of the form
	$$\phi=
	\begin{pmatrix}
	x &  y  & 0 & 0\\
	x+y-z &  z  & 0 & 0\\
	t &  v  & x(z-y) & y(z-y)\\
    u &  w  & (x+y-z)(z-y) & z(z-y)
	\end{pmatrix}.
	$$

 Since
	$$
	\phi^T\begin{pmatrix}
	0 & 0  & 0 & 0\\
	\alpha_1 &  \alpha_2 & 0 & 0\\
	\alpha_3 &  \alpha_5 & 0 & 0\\
	\alpha_5 &  \alpha_4   & 0 & 0
	\end{pmatrix} \phi=	\begin{pmatrix}
    \alpha^* & \alpha^{**}  & 0 & 0\\
	\alpha_1-\alpha^* &  \alpha_2^*-\alpha^{**} & 0 & 0\\
	\alpha_3^* &  \alpha_5^* & 0 & 0\\
	\alpha_5^* &  \alpha_4^*   & 0 & 0
	\end{pmatrix},
	$$
	 we have that the action of ${\rm Aut} ({\mathfrak N}_{08}^{1})$ on the subspace
$\langle \sum\limits_{i=1}^5\alpha_i\nabla_i  \rangle$
is given by
$\langle \sum\limits_{i=1}^5\alpha_i^{*}\nabla_i\rangle,$
where
\begin{longtable}{lcl}
$\alpha^*_1$&$=$&$(x+y-z)((x+y-z) \alpha _2+u \alpha _4+t \alpha _5)+x((x+y-z) \alpha _1+t \alpha _3+u \alpha _5)+$\\
&&\multicolumn{1}{r}{$(x+y-z)(z \alpha _2+w \alpha _4+v \alpha_5)+x(z \alpha _1+v \alpha _3+w \alpha_5),$} \\
$\alpha^*_2$&$=$&$y (x+y) \alpha _1+(x+y) z \alpha _2+t y \alpha _3+v y \alpha _3+$\\
&&\multicolumn{1}{r}{$u z \alpha _4+w z \alpha _4+u y \alpha _5+w y \alpha _5+t z \alpha _5+v z \alpha _5,$} \\
$\alpha^*_3$&$=$&$(z-y)(x^2 \alpha _3+(x+y-z)((x+y-z) \alpha _4+2 x \alpha _5)),$ \\
$\alpha_4^*$&$=$&$(z-y)(y^2 \alpha _3+z(z \alpha _4+2 y \alpha _5)),$\\
$\alpha_5^*$&$=$&$(z-y)((x+y-z)(z \alpha _4+y \alpha _5)+x(y \alpha _3+z \alpha _5)).$\\
\end{longtable}

We are interested only in the cases with $(\alpha_3\alpha_4,\alpha_5)\neq0.$

\begin{enumerate}
\item $\alpha_3\neq 0,$ $\alpha_4=0$ and  $\alpha_5\neq 0,$ then

\begin{enumerate}
    \item If $\alpha_3\neq -2\alpha_5 ,$ then choosing 
\begin{center}
$x=2i\alpha_5^2 ,$
    $y=2\alpha_5^2,$
    $z=((i-1)\alpha_3+2i\alpha_5)\alpha_5,$
    $w=0,$
    $t=0,$
    $u=(2+2i)(\alpha_2\alpha_3-\alpha_1\alpha_5)$    and    
    $v=-(2+2i)\alpha_2\alpha_5,$\end{center}
we obtain the following representative $\langle \nabla_3+\nabla_4\rangle.$

    \item If $\alpha_3= -2\alpha_5 ,$ then choosing 
\begin{center}
$x=\alpha_5,$
    $y=\alpha_5,$
    $z=0,$
    $w=0,$
    $t=0,$
    $u=-2(\alpha_1\alpha_2)$    and    $v=-2\alpha_2,$\end{center}
we obtain the following representative $\langle \nabla_3-\nabla_4\rangle.$

\end{enumerate}

\item If $\alpha_3= 0$  and  $\alpha_5\neq 0,$ then 
choosing  $z=0$ we have $\alpha_4^*=0$ and $\alpha_5^*\neq 0,$ that gives the first case with $\alpha_3\neq0.$

\item $\alpha_3\alpha_4\neq 0$ and  $\alpha^2_5=\alpha_3\alpha_4,$ then
\begin{enumerate}
    \item If $\alpha_3\neq-\alpha_5,$ then choosing 
\begin{center}
    $x=0,$
    $z=-y\alpha_3\alpha_5^{-1},$
    $w=0,$
    $t=0,$
    $u=-y\alpha_2\alpha_3\alpha_5^{-2}$    and    $v=0,$\end{center}
we obtain the following representatives
$\langle \nabla_3 \rangle$
or
$\langle \nabla_2+\nabla_3 \rangle,$
depending on $\alpha_2\alpha_3=\alpha_1\alpha_5$ or not.
These orbits give us split extensions.

    \item If $\alpha_3=-\alpha_5,$ then choosing 
\begin{center}
    $y=0,$
    $z=1,$
    $w=x(\alpha_2-x(\alpha_1+\alpha_2)){\alpha_5^{-1}},$
    $t=0,$
    $u=0$    and    $v=0,$\end{center}
we obtain the following representatives
$\langle -\nabla_3-\nabla_4+\nabla_5\rangle$
or
$\langle \nabla_2-\nabla_3-\nabla_4+\nabla_5\rangle,$
depending on $\alpha_1=-\alpha_2$ or not.

\end{enumerate}

\item $\alpha_3\alpha_4\neq 0$ and  $\alpha^2_5 \neq \alpha_3\alpha_4,$ then
\begin{enumerate}
   \item If $\alpha_4=-\alpha_5,$ then choosing 
\begin{center}
    $x=\alpha_5(\alpha_3+\alpha_5),$
    $y=\alpha_5(\alpha_3+\alpha_5),$
    $z=\alpha_5(\alpha_3+\alpha_5)+i\sqrt{\alpha_3(\alpha_3+\alpha_5)^3},$
    $w=0,$
    $t=0,$
    $u=2\alpha_2\alpha_3-2\alpha_1\alpha_5$    and    
    $v=-2(\alpha_1+\alpha_5)\alpha_5,$\end{center}
we obtain the representative
$\langle \nabla_3+   \nabla_4\rangle.$

    \item If $\alpha_4\neq-\alpha_5,$ then choosing 
\begin{center}
    $x=\alpha_4(\alpha_3\alpha_4-\alpha_5^2),$
    $y=0,$
    $z=(\alpha_4+\alpha_5)(\alpha_3\alpha_3-\alpha_5^2),$
    $w=0,$
    $t=0,$
    $u=-2\alpha_2\alpha_3\alpha_4+\alpha_1\alpha_4\alpha_5$    and    
    $v=\alpha_4(\alpha_2\alpha_5-\alpha_1\alpha_4),$\end{center}
we obtain the family of representatives
$\langle \nabla_3+ \lambda  \nabla_4\rangle_{\lambda\neq0}.$

\end{enumerate}

\end{enumerate}

Let us now consider the family of representatives $\langle \nabla_3+ \lambda  \nabla_4\rangle_{\lambda\neq0,-1}.$
Choosing 
\begin{center}
    $x=\sqrt{\lambda}+\lambda,$
    $y=\sqrt{\lambda}-\lambda,$
    $z=1+\sqrt{\lambda},$
    $w=0,$
    $t=0,$
    $u=0$    and    
    $v=0,$\end{center}
we obtain the  representative
$\langle \nabla_3+   \nabla_4\rangle.$

Thus,  we have the following orbits 
\begin{center}
$\langle -\nabla_3-\nabla_4+\nabla_5\rangle,$
$\langle \nabla_2-\nabla_3-\nabla_4+\nabla_5\rangle,$
$\langle \nabla_3-   \nabla_4\rangle$ and
$\langle \nabla_3+   \nabla_4\rangle.$

\end{center}


Summarizing, all orbits in case of $\alpha\neq1$ and $\alpha=1$, we have the following new algebras (see Section \ref{secteoA}):
$\mathbb{L}_{47}^{\alpha, \beta\neq 0},$ 
$\mathbb{L}_{48}$ and 
$\mathbb{L}_{49}.$   Note that $\mathbb{L}_{47}^{\alpha,\beta}\simeq \mathbb{L}_{47}^{\alpha,\alpha^3{\beta}^{-1}}$ and $\mathbb{L}_{47}^{1,\beta\neq -1}\simeq \mathbb{L}_{47}^{1,1}$


 \subsubsection{Central extensions of ${\mathfrak N}_{12}$}
	Let us use the following notations:
	\begin{longtable}{lllllll} $\nabla_1 = [\Delta_{11}],$ & $\nabla_2 = [\Delta_{22}],$ & $\nabla_3 = [-\Delta_{13}+\Delta_{14}+\Delta_{31}],$ \\ 
	$\nabla_4 = [\Delta_{32}],$ & $\nabla_5 = [\Delta_{41}],$ & $\nabla_6 = [\Delta_{23}-\Delta_{24}+\Delta_{42}].$
	\end{longtable}	
	
Take $\theta=\sum\limits_{i=1}^6\alpha_i\nabla_i\in {\rm H^2}({\mathfrak N}_{12}).$
	The automorphism group of ${\mathfrak N}_{12}$ consists of invertible matrices of the form
	$$\phi_1=
	\begin{pmatrix}
	x &  0  & 0 & 0\\
	0 &  y  & 0 & 0\\
	z &  v  & xy & 0\\
    u &  t  & 0 & xy
	\end{pmatrix}, \quad 
	\phi_2=
	\begin{pmatrix}
	0 &  x  & 0 & 0\\
	y &  0  & 0 & 0\\
	z &  v  & 0 & x y\\
    u &  t  & x y & 0
	\end{pmatrix}.
	$$
	
 Since
	$$
	\phi_1^T\begin{pmatrix}
	\alpha_1 & 0  & -\alpha_3 & \alpha_3\\
	0 &  \alpha_2 & \alpha_6 & -\alpha_6\\
	\alpha_3 &   \alpha_4   & 0 & 0\\
	\alpha_5 &  \alpha_6   & 0 & 0
	\end{pmatrix} \phi_1=	\begin{pmatrix}
	\alpha_1^* & \alpha^*  & -\alpha_3^* & \alpha_3^*\\
	\alpha^{**} &  \alpha_2^* & \alpha_6^* & -\alpha_6^*\\
	\alpha_3^* &   \alpha_4^*  & 0 & 0\\
	\alpha_5^* &  \alpha_6^*   & 0 & 0
	\end{pmatrix},
	$$
	 we have that the action of ${\rm Aut} ({\mathfrak N}_{12})$ on the subspace
$\langle \sum\limits_{i=1}^6\alpha_i\nabla_i  \rangle$
is given by
$\langle \sum\limits_{i=1}^6\alpha_i^{*}\nabla_i\rangle,$
where
\begin{longtable}{lcllcl}
$\alpha^*_1$&$=$&$x(x \alpha _1+u \left(\alpha _3+\alpha _5\right)),$ &
$\alpha^*_2$&$=$&$y \left(y \alpha _2+v \left(\alpha _4+\alpha _6\right)\right)$ \\
$\alpha^*_3$&$=$&$x^2 y \alpha _3,$ &
$\alpha_4^*$&$=$&$x y^2 \alpha _4,$\\
$\alpha_5^*$&$=$&$x^2 y \alpha _5,$&
$\alpha_6^*$&$=$&$x y^2 \alpha _6.$\\
\end{longtable}

We are interested only in the cases with $(\alpha_3,\alpha_4\alpha_5, \alpha_6)\neq (0,0,0).$

\begin{enumerate}
 
\item If $\alpha_3+\alpha_5\neq 0$ and $\alpha_4+\alpha_6\neq 0$, then  choosing $u=-\frac{x\alpha_1}{\alpha_3+\alpha_5}$ and $ v=-\frac{y\alpha_2}{\alpha_4+\alpha_6},$ we have $\alpha_1^*=0$ and $\alpha_2^*=0$.
\begin{enumerate}
    \item If $(\alpha_3,\alpha_6)=(0,0),$ then $\alpha_4\neq0$ and $\alpha_5\neq0$ and choosing $x=\alpha_4$ and $y=\alpha_5,$
    we have the representative $\langle \nabla_4+\nabla_5\rangle.$
    \item If $(\alpha_3,\alpha_6)\neq (0,0),$ then without loss of generality (maybe with an action of a suitable $\phi_2$), we can suppose $\alpha_3\neq0$ and consider following subcases: 
    \begin{enumerate}
        \item if $\alpha_4=0,$ then $\alpha_6\neq0$ and choosing $x=\sqrt[3]{\frac{\alpha_6}{\alpha_3^2}}, y=\sqrt[3]{\frac{\alpha_3}{\alpha_6^2}},$ we have the representative $\langle \nabla_3+\lambda \nabla_5+\nabla_6\rangle_{\lambda\neq -1};$
        \item if $\alpha_4\neq0,$ then choosing $x=\sqrt[3]{\frac{\alpha_4}{\alpha_3^2}}, y=\sqrt[3]{\frac{\alpha_3}{\alpha_4^2}},$ we have the representative $\langle \nabla_3+\nabla_4+\lambda_1\nabla_5+\lambda_2\nabla_6\rangle_{\lambda_1\neq -1,\lambda_2\neq-1}.$
        In case of $\lambda_1\lambda_2\neq0,$ through the second automorphism, we get the following relationship $O(\lambda_1,\lambda_2) \simeq O(\frac{1}{\lambda_2},\frac{1}{\lambda_1}).$
    \end{enumerate}
\end{enumerate}
 
 \item If $(\alpha_3+\alpha_5,\alpha_4+\alpha_6)\neq (0,0)$ and $(\alpha_3+\alpha_5)(\alpha_4+\alpha_6)=0,$ then without loss of generality (maybe with an action of a suitable $\phi_2$), we can suppose  $\alpha_3+\alpha_5=0,\ \alpha_4+\alpha_6\neq0$.
 Choosing $v=-\frac{y\alpha_2}{\alpha_4+\alpha_6},$ we have $\alpha_2^*=0$.

\begin{enumerate}
    \item If $\alpha_1=0,\alpha_4=0,$ then $\alpha_6\neq0$ and we have the representatives
     $\langle \nabla_6\rangle$ and $\langle \nabla_3-\nabla_5+\nabla_6\rangle,$ 
    depending on $\alpha_3=0$  or not.
    \item If $\alpha_1=0,\ \alpha_4\neq0,$ then consider the following subcases:
    \begin{enumerate}
        \item if $\alpha_3=0,$ then $\alpha_6\neq0$ and choosing $x=\frac{1}{\alpha_4}, y=1,$ we have the representative $\langle  \nabla_4+\lambda\nabla_6\rangle_{\lambda\neq-1};$
        \item if $\alpha_3\neq0,$ then choosing $x=\sqrt[3]{\frac{\alpha_4}{\alpha_3^2}}, y=\sqrt[3]{\frac{\alpha_3}{\alpha_4^2}},$
        we have the representative $\langle  \nabla_3+\nabla_4-\nabla_5+\lambda\nabla_6\rangle_{\lambda\neq-1}.$
    \end{enumerate}
    \item If $\alpha_1\neq0,$ then consider the following cases:
    \begin{enumerate}
        \item if $\alpha_3=0,\ \alpha_6\neq0,$ then choosing $x=\sqrt{\frac{1}{\alpha_1}}, y=\sqrt{\frac{\sqrt{\alpha_1}}{\alpha_6}},$ we have the representative 
        $\langle \nabla_1+\lambda\nabla_4+\nabla_6\rangle_{\lambda\neq-1};$
        \item if $\alpha_3\neq0,\ \alpha_6=0,$ then choosing $x=\frac{\alpha_1\alpha_4}{\alpha_3^2}, y=\frac{\alpha_1}{\alpha_3},$
        we have the representative
        $\langle \nabla_1+\nabla_3+\nabla_4-\nabla_5\rangle;$
        \item if $\alpha_3\neq0,\ \alpha_6\neq0,$ then choosing $x=\frac{\alpha_1\alpha_6}{\alpha_3^2}, y=\frac{\alpha_1}{\alpha_3},$
        we have the representative 
        $\langle \nabla_1+\nabla_3+\lambda\nabla_4-\nabla_5+\nabla_6\rangle_{\lambda\neq-1}.$
    \end{enumerate}
\end{enumerate}

  \item If  $\alpha_3+\alpha_5=0$ and $\alpha_4+\alpha_6=0$, then we have that $(\alpha_3, \alpha_4) \neq (0,0)$ and without loss of generality, we can suppose $\alpha_4\neq 0$.
\begin{enumerate}
    \item If $\alpha_1=\alpha_2=\alpha_3=0,$ then choosing $x=\frac{1}{\alpha_4}, y=1,$ we have the representative
    $\langle \nabla_4-\nabla_6\rangle.$
    \item If $\alpha_1=\alpha_2=0,\ \alpha_3\neq0,$ then choosing $x=\sqrt[3]{\frac{\alpha_4}{\alpha_3^2}}, y=\sqrt[3]{\frac{\alpha_3}{\alpha_4^2}},$ we have the representative 
    $\langle \nabla_3+\nabla_4-\nabla_5-\nabla_6\rangle.$
    \item If $\alpha_2=0,\ \alpha_1\neq0,$ then consider the following cases: 
    \begin{enumerate}
        \item if $\alpha_3=0,$ then choosing 
        $x=-\alpha_1\alpha_4,$ $y=\alpha_1,$ we have the representative 
        $\langle \nabla_1-\nabla_4+\nabla_6\rangle;$
        \item if $\alpha_3\neq0,$ then choosing $x=-\frac{\alpha_1\alpha_4}{\alpha_3^2}, y=\frac{\alpha_1}{\alpha_3},$ we have the representative 
        $\langle \nabla_1+\nabla_3-\nabla_4-\nabla_5+\nabla_6\rangle.$
    \end{enumerate}
    \item If $\alpha_2\neq0,$ then consider the following cases: 
    \begin{enumerate}
        \item if $\alpha_1=\alpha_3=0,$ then choosing $x=-\frac{\alpha_2}{\alpha_4}, y=\frac{1}{\sqrt{\alpha_2}},$ we have the representative 
        $\langle \nabla_2-\nabla_4+\nabla_6\rangle;$
        \item if $\alpha_3=0,\ \alpha_1\neq0,$ then choosing $x=-\frac{\alpha_2}{\alpha_4}, y=\sqrt{\frac{\alpha_1\alpha_2}{\alpha_4^2}},$ we have the representative 
        $\langle \nabla_1+\nabla_2-\nabla_4+\nabla_6\rangle;$
        \item if $\alpha_3\neq0,$ then choosing $x=-\frac{\alpha_2}{\alpha_4}, y=\frac{\alpha_2\alpha_3}{\alpha_4^2},$ we have the representative 
        $\langle \lambda\nabla_1+\nabla_2+\nabla_3-\nabla_4-\nabla_5+\nabla_6\rangle.$
    \end{enumerate}
\end{enumerate}
\end{enumerate}

Summarizing all cases, we have the following distinct orbits
\begin{center} 
$\langle \nabla_4+\nabla_5\rangle,$   
$\langle \nabla_1+\nabla_2+\nabla_3-\nabla_5\rangle,$
$\langle \nabla_3+\nabla_4+\alpha \nabla_5+\beta\nabla_6\rangle^{
O(\alpha,\beta)\simeq O(\beta^{-1},\alpha^{-1})},$ 
$\langle \nabla_6\rangle,$ 
$\langle \nabla_3-\nabla_5+\nabla_6\rangle,$  
$\langle \nabla_4+\lambda\nabla_6\rangle,$ 
$\langle \nabla_1+\lambda\nabla_4+\nabla_6\rangle,$ 
$\langle \nabla_1+\nabla_3+\nabla_4-\nabla_5\rangle,$ 
$\langle \nabla_1+\nabla_3+\lambda \nabla_4-\nabla_5+\nabla_6\rangle,$
$\langle \nabla_2-\nabla_4+\nabla_6\rangle,$  
$\langle \nabla_1+\nabla_2-\nabla_4+\nabla_6\rangle,$  
$\langle \lambda\nabla_1+\nabla_2+\nabla_3-\nabla_4-\nabla_5+\nabla_6\rangle,$ 
\end{center}
which gives the following new algebras (see Section \ref{secteoA}):
\begin{longtable}{llllllllllllllllll}
$\mathbb{L}_{50},$ & 
$\mathbb{L}_{51},$ & 
$\mathbb{L}_{52}^{\alpha, \beta},$ &
$\mathbb{L}_{53},$ &
$\mathbb{L}_{54},$ &
$\mathbb{L}_{55}^\alpha,$ & 
$\mathbb{L}_{56}^\alpha,$ & 
$\mathbb{L}_{57},$ & 
$\mathbb{L}_{58}^\alpha,$ & 
$\mathbb{L}_{59},$ & 
$\mathbb{L}_{60},$ & 
$\mathbb{L}_{61}^\alpha.$  
\end{longtable}

Note that 
$\mathbb{L}_{52}^{\alpha,\beta}\simeq \mathbb{L}_{52}^{{\beta}^{-1},{\alpha}^{-1}}.$

 \subsubsection{Central extensions of ${\mathfrak N}_{13}$}
	Let us use the following notations:
	\begin{longtable}{lllllll} $\nabla_1 = [\Delta_{21}],$ & $\nabla_2 = [\Delta_{22}],$ & $\nabla_3 = [\Delta_{32}-\Delta_{41}],$ &
	$\nabla_4 = [\Delta_{31}-2\Delta_{32}+\Delta_{42}].$
	\end{longtable}	
	
Take $\theta=\sum\limits_{i=1}^4\alpha_i\nabla_i\in {\rm H^2}({\mathfrak N}_{13}).$
	The automorphism group of ${\mathfrak N}_{13}$ consists of invertible matrices of the form
	$$\phi_1=
	\begin{pmatrix}
	x &  0  & 0 & 0\\
	0 & x  & 0 & 0\\
	z &  u  & x^2 & 0\\
    t &  v  & 0 & x^2
	\end{pmatrix}, \quad 
	\phi_2=
	\begin{pmatrix}
	0 &  x  & 0 & 0\\
	x &  0  & 0 & 0\\
	z &  u  & -x^2 & 2x^2\\
    t &  v  & 0 & x^2
	\end{pmatrix}.
	$$

 Since
	$$
	\phi^T\begin{pmatrix}
	0 & 0  & 0 & 0\\
	\alpha_1 &  \alpha_2 & 0 & 0\\
	\alpha_4 &   \alpha_3-2\alpha_4 & 0 & 0\\
	-\alpha_3&  \alpha_4 & 0 & 0
	\end{pmatrix} \phi =	\begin{pmatrix}
	\alpha^* & \alpha^{**}  & 0 & 0\\
	\alpha_1^*-\alpha^{**} & \alpha_2^*+\alpha^*+2\alpha^{**} & 0 & 0\\
	\alpha_4^* &   \alpha_3^*-2\alpha_4^* & 0 & 0\\
	-\alpha_3^*&  \alpha_4^* & 0 & 0
	\end{pmatrix},
	$$
	 we have that the action of ${\rm Aut} ({\mathfrak N}_{13})$ on the subspace
$\langle \sum\limits_{i=1}^4\alpha_i\nabla_i  \rangle$
is given by
$\langle \sum\limits_{i=1}^4\alpha_i^{*}\nabla_i\rangle,$
where
 \begin{longtable}{clcllcl}
\mbox{for }$\phi_1:$ & 
$\alpha^*_1$&$=$&$x(x \alpha_1-(v+z) \alpha _3+(u+z) \alpha_4)),$ &
 $\alpha^*_3$&$=$&$x^3 \alpha _3,$ \\
& $\alpha^*_2$&$=$&$x(x\alpha_2+(t+u-2z)\alpha _3+(v+3 z-2 t-2u) \alpha_4)),$ &
$\alpha_4^*$&$=$&$x^3 \alpha _4;$\\
\\
\mbox{for }$\phi_2:$ & 
$\alpha^*_1$&$=$&$x(x \alpha_2+(u+z) \alpha _3+(t+v-2u-2 z) \alpha_4)),$ &
 $\alpha^*_3$&$=$&$-x^3 (2\alpha _3-3\alpha_4),$ \\
& $\alpha^*_2$&$=$&$x(x(2\alpha_1+\alpha_2)+(v+z-2t)\alpha _3+( t-u) \alpha_4)),$ &
$\alpha_4^*$&$=$&$-x^3 (\alpha_3-2\alpha _4).$\\
\end{longtable}

We are interested only in the cases with $(\alpha_3,\alpha_4)\neq (0,0).$ Without loss of generality (maybe with an action of a suitable $\phi_2$), we can suppose $\alpha_3\neq0.$ 

\begin{enumerate}
 
 \item If $\alpha_3\neq\alpha_4,$ then choosing
\begin{center}
$\phi=\phi_1,$ 
    $x=(\alpha_3-\alpha_4)^2,$
 $t=0,$ 
 $u=-\alpha_2\alpha_3-\alpha_1\alpha_4,$ 
 $v=\alpha_1(\alpha_3-2\alpha_4)-\alpha_2\alpha_4$ and 
 $z=0,$
\end{center}  
 we   obtain the representative  $\langle \nabla_3+\alpha \nabla_4\rangle_{\alpha \neq 1}.$
 \item If $\alpha_3=\alpha_4,$ then choosing
\begin{center}
$\phi=\phi_1,$ 
    $x=2 \alpha_3,$
 $t=\alpha_1+\alpha_2,$ 
 $u=-\alpha_1+\alpha_2,$ 
 $v=0$ and 
 $z=0,$
\end{center}  
 we   obtain the representative  $\langle \nabla_3+\alpha \nabla_4\rangle_{\alpha \neq 1}.$

\end{enumerate}

Therefore, we have the representative  $\langle \nabla_3+\alpha \nabla_4\rangle.$ 
We obtain the following new algebras $\mathbb{L}_{62}^{\alpha}.$
Note that
$\mathbb{L}_{62}^{\alpha\neq \frac{2}{3}}  \simeq \mathbb{L}_{62}^\frac{2\alpha-1}{3\alpha-2}.$

  \subsubsection{Central extensions of ${\mathfrak N}_{14}^{\alpha\neq -1}$}
	Let us use the following notations:
	\begin{longtable}{lllllll} $\nabla_1 = [\Delta_{11}],$ & $\nabla_2 = [\Delta_{21}],$ & $\nabla_3 = [\Delta_{32}],$ &
	$\nabla_4 = [\alpha\Delta_{31}+\Delta_{42}].$
	\end{longtable}	
	
Take $\theta=\sum\limits_{i=1}^4\alpha_i\nabla_i\in {\rm H^2}({\mathfrak N}_{14}^{\alpha \neq -1}).$
	The automorphism group of ${\mathfrak N}_{14}^{\alpha \neq -1}$ consists of invertible matrices of the form
	$$\phi=
	\begin{pmatrix}
	x &  z  & 0 & 0\\
	0 &  y  & 0 & 0\\
	w &  u  & y^2 & 0\\
    t &  v  & (1+\alpha)yz & xy
	\end{pmatrix}.
	$$
	Since
	$$
	\phi^T\begin{pmatrix}
	\alpha_1 & 0  & 0 & 0\\
	\alpha_2 &  0 & 0 & 0\\
	\alpha\alpha_4 &   \alpha_3 & 0 & 0\\
	0&  \alpha_4 & 0 & 0
	\end{pmatrix} \phi=	\begin{pmatrix}
	\alpha_1^* & \alpha^{*}  & 0 & 0\\
	\alpha_2^*+\alpha\alpha^{*} & \alpha^{**} & 0 & 0\\
	\alpha\alpha_4^* &   \alpha_3^* & 0 & 0\\
	0&  \alpha_4^* & 0 & 0
	\end{pmatrix},
	$$
we have that the action of ${\rm Aut} ({\mathfrak N}_{14}^{\alpha\neq -1})$ on the subspace
$\langle \sum\limits_{i=1}^4\alpha_i\nabla_i  \rangle$
is given by
$\langle \sum\limits_{i=1}^4\alpha_i^{*}\nabla_i\rangle,$
where
\begin{longtable}{lcl}
$\alpha^*_1$&$=$&$x(x \alpha _1+w \alpha  \alpha_4),$ \\
$\alpha^*_2$&$=$&$xz(1-\alpha)\alpha_1+x y \alpha _2-\alpha  \left(w y \alpha _3+(t y-u x+w z \alpha ) \alpha _4\right),$ \\
$\alpha^*_3$&$=$&$y^2 \left(y \alpha _3+z (1+2 \alpha ) \alpha _4\right),$ \\
$\alpha_4^*$&$=$&$x y^2 \alpha _4.$\\
\end{longtable}

We are interested only in the cases with $\alpha_4\neq 0$ for $\alpha\neq0$ 
and in the cases with  $(\alpha_3,\alpha_4)\neq 0$ for $\alpha=0.$
Let us consider the following cases.
\begin{enumerate}
    \item  If $\alpha\neq0$ and $\alpha\neq -\frac{1}{2},$  then choosing 
\begin{center}
    $x=1,$
    $y=\alpha(1+2\alpha)\alpha_4,$
    $z=-\alpha\alpha_3\alpha_4,$
    $w=-\alpha_1(\alpha\alpha_4)^{-1},$ 
    $t=0,$
    $u=-\alpha _2\alpha_4 -2 \alpha(\alpha_1\alpha_3+\alpha_2\alpha_4),$    
\end{center}    
    we have the representative  $\langle\nabla_4\rangle_{\alpha\neq-\frac{1}{2}}.$

    \item If  $\alpha= -\frac{1}{2},$ then choosing 
\begin{center}
    $x=1,$
    $z=0,$
    $w=2\alpha_1\alpha_4^{-1},$ 
    $t=0,$
    $u=2y(\alpha_1\alpha_3+\alpha_2\alpha_4)\alpha_4^{-2},$    
\end{center}    
    we have the representatives  
    $\langle\nabla_4\rangle_{\alpha = -\frac{1}{2}}$ and 
$\langle\nabla_3+\nabla_4\rangle_{\alpha=-\frac{1}{2}}$ depending on $\alpha_3=0$ or not.

 \item If $\alpha=0,$ then choosing 
 $y=\alpha_4$ and $z=-\alpha_3,$ we get $\alpha_3^*=0$ and obtain a split extension.
\end{enumerate}

  \subsubsection{Central extensions of ${\mathfrak N}_{14}^{-1}$}
\label{sec1.3.11}	Let us use the following notations:
	\begin{longtable}{lll} $\nabla_1 = [\Delta_{11}],$ & $\nabla_2 = [\Delta_{21}],$ & $\nabla_3 = [\Delta_{32}],$ \\ $\nabla_4 = [\Delta_{31}-\Delta_{42}]$ &
	$\nabla_5 = [\Delta_{14}-\Delta_{41}],$ & $\nabla_6 = [\Delta_{24}-\Delta_{42}].$
	\end{longtable}	
	
Take $\theta=\sum\limits_{i=1}^6\alpha_i\nabla_i\in {\rm H^2}({\mathfrak N}_{14}^{-1}).$
	The automorphism group of ${\mathfrak N}_{14}^{-1}$ consists of invertible matrices of the form
	$$\phi=
	\begin{pmatrix}
	x &  z  & 0 & 0\\
	0 &  y  & 0 & 0\\
	w &  u  & y^2 & 0\\
    t &  v  & 0 &  x y
	\end{pmatrix}.
	$$

 Since
	$$
	\phi^T\begin{pmatrix}
	\alpha_1 & 0  & 0 & \alpha_5\\
	\alpha_2 &  0 & 0 & \alpha_6\\
	\alpha_4 &   \alpha_3 & 0 & 0\\
	-\alpha_5&  -\alpha_4-\alpha_6 & 0 & 0
	\end{pmatrix} \phi=	\begin{pmatrix}
	\alpha_1^* & \alpha^*  & 0 & \alpha_5^*\\
	\alpha_2^*-\alpha^* &  \alpha^{**} & 0 & \alpha_6^*\\
	\alpha_4^* &   \alpha_3^* & 0 & 0\\
	-\alpha_5^*&  -\alpha_4^*-\alpha_6^* & 0 & 0
	\end{pmatrix},
	$$
we have that the action of ${\rm Aut} ({\mathfrak N}_{14}^{-1})$ on the subspace
$\langle \sum\limits_{i=1}^6\alpha_i\nabla_i  \rangle$
is given by
$\langle \sum\limits_{i=1}^6\alpha_i^{*}\nabla_i\rangle,$
where
\begin{longtable}{lcllcl}
$\alpha^*_1$&$=$&$x(x \alpha _1+w \alpha_4),$ &
$\alpha^*_2$&$=$&$x (2 z \alpha_1+y \alpha_2+u \alpha_4)+w (y \alpha_3+z \alpha_4)-t (z \alpha_5+y (\alpha_4+\alpha_6)),$ \\
$\alpha^*_3$&$=$&$y^2(y \alpha _3+z \alpha_4),$ &
$\alpha_4^*$&$=$&$x y^2 \alpha _4,$\\
$\alpha_5^*$&$=$&$x^2 y \alpha _5,$&
$\alpha_6^*$&$=$&$xy(z\alpha_5+y\alpha_6).$\\
\end{longtable}

We are interested only in  the cases $(\alpha_4, \alpha_3\alpha_5,\alpha_3\alpha_6)\neq (0,0,0).$

\begin{enumerate}

\item If $\alpha_5=\alpha_6=0,$ then $\alpha_4\neq0$ and choosing 
\begin{center}$x=\alpha_4,$
$y=\alpha_4^2,$
$z=-\alpha_3\alpha_4,$
$w=-\alpha_1,$ 
$u=2\alpha_1\alpha_3-\alpha_2\alpha_4$ 
and
$t=0,$\end{center}
we  obtain  the representative $\langle\nabla_4\rangle.$
    
\item If $\alpha_5=0,$ $\alpha_6\neq0,$ then  we consider following subcases:
\begin{enumerate}
\item if $\alpha_4=0,$ then $\alpha_3\neq0$ and choosing 
\begin{center}
$y=-\alpha_6 x \alpha_3^{-1},$
$z=0,$
$w=-\alpha_2 x \alpha_3^{-1}$ and
$t=0,$\end{center}
we have the representatives $\langle\nabla_3+\nabla_6\rangle$ and  $\langle\nabla_1+\nabla_3+\nabla_6\rangle$
    depending on $\alpha_1=0$ or not;
    
\item if $\alpha_4\neq0,$ then choosing 
\begin{center}
$x=\alpha_4,$
$y=\alpha_4,$
$z=-\alpha_3,$
$w=-\alpha_1,$ 
$u=-\alpha_2+2\alpha_1\alpha_3\alpha_4^{-1}$ 
and
$t=0,$\end{center}
we  obtain   the representative $\langle\beta \nabla_4+\nabla_6\rangle.$
\end{enumerate}

\item If $\alpha_5\neq0$ and $\alpha_4\neq 0,$ then choosing 
\begin{center}
$x=\alpha_4^3,$
$y=\alpha_4^2\alpha_5,$
$z=-\alpha_4^2\alpha_6,$
$w=-\alpha_1\alpha_4^2,$ 
$u=\alpha_1\alpha_3\alpha_5-\alpha_2\alpha_4\alpha_5+\alpha_1\alpha_4\alpha_6$ 
and
$t=0,$\end{center}
we  obtain  the family of representatives 
$\langle \beta \nabla_4+\nabla_5+\nabla_6\rangle_{\beta\neq0}.$

\item If $\alpha_5\neq0$ and $\alpha_4=0,$ then $\alpha_3\neq0,$
then choosing 
\begin{center}
$y=x\sqrt{\frac{\alpha_5}{\alpha_3}},$
$z=-\frac{x\alpha_6}{\sqrt{\alpha_3\alpha_5}}$ and
$w=\frac{x(2\alpha_1\alpha_6-\alpha_2\alpha_5)}{\alpha_3\alpha_5},$ 
\end{center} 
we 
have the representatives $\langle\nabla_3+\nabla_5\rangle$ and  $\langle\nabla_1+\nabla_3+\nabla_5\rangle$
    depending on $\alpha_1=0$ or not.
 
\end{enumerate}

Summarizing all cases of the central extension of the algebra ${\mathfrak N}_{14}^{\alpha},$ we have the following distinct orbits
\begin{longtable} {llll}
$\langle \nabla_4\rangle_{\alpha \neq 0}$, &
$\langle \nabla_3+\nabla_4\rangle_{\alpha= -\frac 1 2},$
\end{longtable}
and in case of $\alpha=-1$
\begin{center}  
$\langle \nabla_3+\nabla_6\rangle$, 
$\langle \nabla_1+\nabla_3+\nabla_6\rangle,$   $\langle \beta\nabla_4+\nabla_6\rangle_{\beta \neq 0},$ 
$\langle \nabla_3+\nabla_5\rangle$, 
$\langle \nabla_1+\nabla_3+\nabla_5\rangle,$   $\langle \beta\nabla_3+\nabla_4+\nabla_5\rangle,$
\end{center}
which give the following new algebras (see Section \ref{secteoA}): 
$\mathbb{L}_{63}^{\alpha \neq 0},$ 
$\mathbb{L}_{64},$ 
$\mathbb{L}_{65},$  
$\mathbb{L}_{66},$  
$\mathbb{L}_{67}^{\alpha\neq 0},$  
$\mathbb{L}_{68},$ 
$\mathbb{L}_{69}$ and  
$\mathbb{L}_{70}^{\alpha}.$ 
Note that the algebra $\mathbb{L}_{67}^{0}$ is a symmetric Leibniz algebra with two dimensional annihilator.

 \subsubsection{Central extensions of ${\mathfrak L}_{1}$}
	Let us use the following notations:
	\begin{longtable}{lllllll} $\nabla_1 = [\Delta_{14}],$ & $\nabla_2 = [\Delta_{31}],$ & $\nabla_3 = [\Delta_{41}],$ &
	$\nabla_4 = [\Delta_{44}].$
	\end{longtable}	
	
Take $\theta=\sum\limits_{i=1}^4\alpha_i\nabla_i\in {\rm H^2}({\mathfrak L}_{1}).$
	The automorphism group of ${\mathfrak L}_{1}$ consists of invertible matrices of the form
	$$\phi=
	\begin{pmatrix}
	x &  0  & 0 & 0\\
	y &  x^2 & 0 & 0\\
	u &  xy & x^3 & t\\
    v &  0  & 0 & z
	\end{pmatrix}.
	$$

 Since
	$$
	\phi^T\begin{pmatrix}
	0 & 0  & 0 & \alpha_1\\
	0 &  0 & 0 & 0\\
	\alpha_2 & 0 & 0 & 0\\
	\alpha_3&  0 & 0 & \alpha_4
	\end{pmatrix} \phi=	\begin{pmatrix}
	\alpha^* & 0  & 0 & \alpha_1^*\\
	\alpha^{**} &  0 & 0 & 0\\
	\alpha_2^* & 0 & 0 & 0\\
	\alpha_3^* &  0 & 0 & \alpha_4^*
	\end{pmatrix},
	$$
	 we have that the action of ${\rm Aut} ({\mathfrak L}_{1})$ on the subspace
$\langle \sum\limits_{i=1}^4\alpha_i\nabla_i  \rangle$
is given by
$\langle \sum\limits_{i=1}^4\alpha_i^{*}\nabla_i\rangle,$
where
\begin{longtable}{lcllcl}
$\alpha^*_1$ & $=$ & $z(x \alpha _1+v \alpha_4),$ &
$\alpha^*_2$ & $=$ & $x^4 \alpha _2,$ \\
$\alpha^*_3$ & $=$ & $t x \alpha _2+x z \alpha _3+v z \alpha _4,$ &
$\alpha_4^*$ & $=$ & $z^2 \alpha _4.$\\
\end{longtable}

We are interested only in the cases with $\alpha_2\neq0$ and $(\alpha_1,\alpha_3,\alpha_4)\neq (0,0,0).$ Then putting
$t=-\frac{z(x\alpha_3+v\alpha_4)}{x\alpha_2},$ we have $\alpha_3^*=0.$ Thus, we can suppose $\alpha_3=0$ and 
\begin{longtable}{lcllcllcl}
$\alpha^*_1$ & $=$ & $z(x \alpha _1+v \alpha_4),$ &
$\alpha^*_2$ & $=$ & $x^4 \alpha _2,$ &
$\alpha^*_4$ & $=$ & $z^2 \alpha _4.$
\end{longtable}

\begin{enumerate}
    \item If $\alpha_4=0,$ then $\alpha_1\neq 0$ and  choosing $x=\frac{1}{\sqrt[4]{\alpha_2}},$  $z=\frac{\sqrt[4]{\alpha_2}}{\alpha_1},$ we have the representative  $\langle \nabla_1+\nabla_2\rangle$.
    \item If $\alpha_4\neq0,$ then choosing $x=\frac{1}{\sqrt[4]{\alpha_2}},$  $ z=\frac{1}{\sqrt{\alpha_4}}$ and $v=-\frac{\alpha_1}{\alpha_4\sqrt[4]{\alpha_2}},$ we have the representative  $\langle \nabla_2+\nabla_4\rangle$.
\end{enumerate}

Therefore, we have the following distinct orbits
\begin{center} 
$\langle \nabla_1+\nabla_2\rangle,$   $\langle \nabla_2+\nabla_4\rangle,$
\end{center}
which give the following new algebras:
$\mathbb{L}_{71} $  and 
$\mathbb{L}_{72}. $

 \subsubsection{Central extensions of ${\mathfrak L}_{4}$}
	Let us use the following notations:
	\begin{longtable}{lllllll} $\nabla_1 = [\Delta_{12}],$ & $\nabla_2 = [\Delta_{21}],$ & $\nabla_3 = [\Delta_{22}],$ &
	$\nabla_4 = [\Delta_{41}].$
	\end{longtable}	
	
Take $\theta=\sum\limits_{i=1}^4\alpha_i\nabla_i\in {\rm H^2}({\mathfrak L}_{4}).$
	The automorphism group of ${\mathfrak L}_{4}$ consists of invertible matrices of the form
	$$\phi=
	\begin{pmatrix}
	x &  0  & 0 & 0\\
	y &  x+y & 0 & 0\\
	z &  z & x(x+y) & 0\\
    v &  t  & xz & x^2(x+y)
	\end{pmatrix}.
	$$

 Since
	$$
	\phi^T\begin{pmatrix}
	0 & \alpha_1  & 0 & 0\\
	\alpha_2 &  \alpha_3 & 0 & 0\\
	0 & 0 & 0 & 0\\
	\alpha_4&  0 & 0 & 0
	\end{pmatrix} \phi=	\begin{pmatrix}
	\alpha^* & \alpha_1^*  & 0 & 0\\
	\alpha_2^*+\alpha^{*} &  \alpha_3^* & 0 & 0\\
	\alpha^{**} & 0 & 0 & 0\\
	\alpha_4^* &  0 & 0 & 0
	\end{pmatrix},
	$$
	 we have that the action of ${\rm Aut} ({\mathfrak L}_{4})$ on the subspace
$\langle \sum\limits_{i=1}^4\alpha_i\nabla_i  \rangle$
is given by
$\langle \sum\limits_{i=1}^4\alpha_i^{*}\nabla_i\rangle,$
where
\begin{longtable}{lcllcl}
$\alpha^*_1$& $=$& $(x+y)(x \alpha_1+y \alpha_3),$ &
$\alpha^*_2$& $=$& $x(x\alpha_2+y(\alpha_3-\alpha_1)+(t-v)\alpha_4),$ \\
$\alpha^*_3$& $=$& $(x+y)^2 \alpha _3,$ &
$\alpha_4^*$& $=$& $x^3 (x+y) \alpha _4.$\\
\end{longtable}

We are interested only in the cases with $\alpha_4\neq0.$ Then putting $t=\frac{y \alpha _1-x \alpha _2-y \alpha _3+v \alpha _4}{\alpha_4},$ we have $\alpha_2^*=0.$  

\begin{enumerate}
    \item If $\alpha_3=0,$ then we have the representatives $\langle \nabla_4\rangle$ and $\langle \nabla_1+\nabla_4\rangle,$ depending on $\alpha_1=0$ or not.
    
    \item If $\alpha_3\neq0,$ then we have the following subcases: 
    \begin{enumerate}
        \item if $\alpha_1\neq \alpha_3,$ then choosing $x=1,$  $y=-\alpha_1\alpha_3^{-1},$ we have the representative $\langle \nabla_3+\nabla_4\rangle;$
        \item if $\alpha_1=\alpha_3,$ then choosing $x=1,$  $y=\alpha_4\alpha_1^{-1}-1,$ we have the representative $\langle \nabla_1+\nabla_3+\nabla_4\rangle.$
    \end{enumerate}
\end{enumerate}

Summarizing, we have the following distinct orbits
\begin{longtable} {llll}
$\langle \nabla_4\rangle,$ & $\langle \nabla_1+\nabla_4\rangle,$ & $\langle \nabla_3+\nabla_4\rangle,$ & $\langle \nabla_1+\nabla_3+\nabla_4\rangle,$
\end{longtable}
which give the following new algebras (see Section \ref{secteoA}):
$\mathbb{L}_{73},$ 
$\mathbb{L}_{74},$  
$\mathbb{L}_{75}$ and 
$\mathbb{L}_{76}.$


 \subsubsection{Central extensions of ${\mathfrak L}_{6}$}
	Let us use the following notations:
	\begin{longtable}{lllllll} $\nabla_1 = [\Delta_{21}],$ & $\nabla_2 = [\Delta_{22}],$ & $\nabla_3 = [\Delta_{31}],$ &
	$\nabla_4 = [\Delta_{32}+\Delta_{41}].$
	\end{longtable}	
	
Take $\theta=\sum\limits_{i=1}^4\alpha_i\nabla_i\in {\rm H^2}({\mathfrak L}_{6}).$
	The automorphism group of ${\mathfrak L}_{6}$ consists of invertible matrices of the form
	$$\phi=
	\begin{pmatrix}
	x     &  0   & 0           & 0\\
	x^2-x &  x^2 & 0           & 0\\
	y     &  y   & x^3         & 0\\
    z     &  t   & x^4-x^3+x y & x^4
	\end{pmatrix}.
	$$

 Since
	$$
	\phi^T\begin{pmatrix}
	0        & 0         & 0 & 0\\
	\alpha_1 &  \alpha_2 & 0 & 0\\
	\alpha_3 & \alpha_4  & 0 & 0\\
	\alpha_4 &  0        & 0 & 0
	\end{pmatrix} \phi=	\begin{pmatrix}
	\alpha^*              & \alpha^{**}          & 0 & 0\\
	\alpha_1^*+\alpha^{*} & \alpha_2^*+\alpha^{**} & 0 & 0\\
	\alpha_3^*+\alpha^{**}& \alpha_4^* & 0 & 0\\
	\alpha_4^*            &  0        & 0 & 0
	\end{pmatrix},
	$$
	 we have that the action of ${\rm Aut} ({\mathfrak L}_{6})$ on the subspace
$\langle \sum\limits_{i=1}^4\alpha_i\nabla_i  \rangle$
is given by
$\langle \sum\limits_{i=1}^4\alpha_i^{*}\nabla_i\rangle,$
where
\begin{longtable}{lcllcl}
$\alpha^*_1$&$=$&$x(x (\alpha _1-\alpha _2)+ x^2 \alpha _2+(t-z) \alpha_4),$ &
$\alpha^*_2$&$=$&$x^3 \alpha _2,$ \\
$\alpha^*_3$&$=$&$x^3((1-x)\alpha_2+x(\alpha_3-2(1-x)\alpha_4)),$ &
$\alpha_4^*$&$=$&$x^5 \alpha _4.$\\
\end{longtable}

We are interested only in the cases with $\alpha_4\neq0.$ Then putting $t=\frac{z\alpha_4-x\alpha_1+(1-x)x\alpha_2}{\alpha_4},$ we have $\alpha_1^*=0.$ 
\begin{enumerate}
        \item  If $\alpha_2=0$ and $\alpha_3\neq 2\alpha_4,$ then choosing $x=\frac{2\alpha_4-\alpha_3}{\alpha_4}$, we have the representative  $\langle \nabla_4\rangle.$
        \item  If $\alpha_2=0$ and $\alpha_3=2\alpha_4,$ then we have the representative $\langle 2\nabla_3+\nabla_4\rangle.$
    
    \item If $\alpha_2\neq0$, then we have the representative $\langle \nabla_2+\alpha \nabla_3+\nabla_4\rangle.$
\end{enumerate}

Thus, we have the following distinct orbits
\begin{center}  
$\langle \nabla_4\rangle,$   
$\langle 2\nabla_3+\nabla_4\rangle,$  
$\langle \nabla_2+\alpha \nabla_3+\nabla_4\rangle,$
\end{center}
which give the following new algebras (see Section \ref{secteoA}):
$\mathbb{L}_{77},$ 
$\mathbb{L}_{78}$ and 
$\mathbb{L}_{79}^{\alpha}.$ 
 

 \subsubsection{Central extensions of ${\mathfrak L}_{7}$}
	Let us use the following notations:
	\begin{longtable}{lllllll} $\nabla_1 = [\Delta_{21}],$ & $\nabla_2 = [\Delta_{22}],$ & $\nabla_3 = [\Delta_{31}],$ &
	$\nabla_4 = [\Delta_{32}+\Delta_{41}].$
	\end{longtable}	
	
Take $\theta=\sum\limits_{i=1}^4\alpha_i\nabla_i\in {\rm H^2}({\mathfrak L}_{7}).$
	The automorphism group of ${\mathfrak L}_{7}$ consists of invertible matrices of the form
	$$\phi=
	\begin{pmatrix}
	x  &  0   & 0      & 0\\
	y  &  x^2 & 0      & 0\\
	z  &  0   & x^2    & 0\\
    t  &  u   & x(y+z) & x^3
	\end{pmatrix}.
	$$

 Since
	$$
	\phi^T\begin{pmatrix}
	0        & 0         & 0 & 0\\
	\alpha_1 &  \alpha_2 & 0 & 0\\
	\alpha_3 & \alpha_4  & 0 & 0\\
	\alpha_4 &  0        & 0 & 0
	\end{pmatrix} \phi=	\begin{pmatrix}
	\alpha^*              & \alpha^{**}& 0 & 0\\
	\alpha_1^*            & \alpha_2^* & 0 & 0\\
	\alpha_3^*+\alpha^{**}& \alpha_4^* & 0 & 0\\
	\alpha_4^*            &  0         & 0 & 0
	\end{pmatrix},
	$$
	 we have that the action of ${\rm Aut} ({\mathfrak L}_{7})$ on the subspace
$\langle \sum\limits_{i=1}^4\alpha_i\nabla_i  \rangle$
is given by
$\langle \sum\limits_{i=1}^4\alpha_i^{*}\nabla_i\rangle,$
where
\begin{longtable}{lcllcl}
$\alpha^*_1$ & $=$ & $x(x^2 \alpha_1+x y \alpha_2+u \alpha_4),$ &
$\alpha^*_2$ & $=$ & $x^4 \alpha_2,$ \\
$\alpha^*_3$ & $=$ & $x^2(x \alpha_3+y(2\alpha_4-\alpha_2)),$ &
$\alpha_4^*$ & $=$ & $x^4 \alpha_4.$\\
\end{longtable}

We are interested only in the cases with $\alpha_4\neq0.$ Then putting $u=-\frac{x(x \alpha_1+y \alpha_2)}{\alpha_4},$ we obtain that $\alpha^*_1=0$.

\begin{enumerate}
           \item If $\alpha_2\neq 2\alpha_4$, then choosing $y=\frac{x\alpha_3}{\alpha_2-2\alpha_4},$ we have the representative $\langle \alpha\nabla_2+\nabla_4\rangle_{\alpha\neq2}.$          
           \item If $\alpha_2=2\alpha_4$, then  we have the representatives $\langle 2\nabla_2+\nabla_4\rangle$ or $\langle 2\nabla_2+\nabla_3+\nabla_4\rangle,$
        depending on $\alpha_3=0$ or not.
 
\end{enumerate}

Summarizing, we have the following distinct orbits
\begin{longtable} {llll}
$\langle \alpha \nabla_2+\nabla_4\rangle,$ & $\langle 2\nabla_2+\nabla_3+\nabla_4\rangle.$\\
\end{longtable}
which give the following new algebras (see Section \ref{secteoA}): $\mathbb{L}_{80}^{\alpha}$ and $\mathbb{L}_{81}.$
 

 \subsubsection{Central extensions of ${\mathfrak L}_{13}$}
Since ${\rm dim} \  {\rm H^2}({\mathfrak L}_{13})=1,$ we obtain only one new algebra (see Section \ref{secteoA}):	 $\mathbb{L}_{82}.$

\subsection{The algebraic classification of  nilpotent Leibniz algebras}
The algebraic classification of complex $5$-dimensional nilpotent  Leibniz algebras consists of three parts:
\begin{enumerate}
    \item $5$-dimensional algebras with identity $xyz=0$ (also known as $2$-step nilpotent algebras) is the intersection of all varieties of algebras defined by a family of polynomial identities   of degree three or more; for example, it is in intersection of associative, Zinbiel, Leibniz, etc. algebras. All these algebras can be obtained as central extensions of zero-product algebras. The geometric classification of $2$-step nilpotent algebras is given in \cite{ikp20}. It is the reason why we are not interested in it.
    
    \item $5$-dimensional nilpotent symmetric Leibniz (non-$2$-step nilpotent) algebras, which are central extensions of nilpotent symmetric Leibniz algebras with non-zero product of a smaller dimension given in \cite{ak21}.

    \item $5$-dimensional nilpotent non-symmetric Leibniz  algebras given above and summarized in Theorem A below.

  \end{enumerate}

\subsubsection{Classification theorem}\label{secteoA}
Now we are ready summarize all results related to the algebraic classification of complex $5$-dimensional nilpotent   Leibniz   algebras.

\begin{theoremA}

Let ${\mathbb L}$ be a complex $5$-dimensional nilpotent non-symmetric Leibniz algebra.
Then ${\mathbb L}$ is   isomorphic to one algebra from the following list:

{\footnotesize 
\begin{longtable}{llllllllll}
$\mathbb{L}_{01}$ & $: $ & $e_1e_1=e_2$ & $e_2e_1=e_4$ & $e_3e_1=e_5 $\\
$\mathbb{L}_{02}$ & $: $ & $e_1e_1=e_2$  & $e_1e_3=e_4$ & $e_2e_1=e_4$  & $e_3e_1=e_5 $\\
$\mathbb{L}_{03}^\alpha$ & $: $ & $e_1e_1=e_2$ & $e_1e_3=e_5$ & $e_2e_1=e_4$ & $e_3e_1=\alpha e_5 $\\
$\mathbb{L}_{04}$ & $: $ & $e_1e_1=e_2$ & $e_2e_1=e_4$ & $e_3e_1=e_5$ & $e_3e_3=e_4$\\
$\mathbb{L}_{05}^\alpha$ & $: $ & $e_1e_1=e_2$ & $e_1e_3=e_5$ &$e_2e_1=e_4$ & $e_3e_1=\alpha e_5$ & $e_3e_3=e_4$\\
$\mathbb{L}_{06}$ & $: $ & $e_1e_1=e_2$ & $e_2e_1=e_4$ & $e_3e_3=e_5 $\\
$\mathbb{L}_{07}$ & $: $ & $e_1e_1=e_2$ & $e_1e_3=e_4$ & $e_2e_1=e_4$ & $e_3e_3=e_5 $\\
$\mathbb{L}_{08}$ & $: $ & $e_1e_1=e_2$ & $e_1e_3=e_4+e_5$ & $e_2e_1=e_4$ & $e_3e_3=e_5$\\
$\mathbb{L}_{09}$ & $: $ & $e_1e_1=e_3+e_5$ & $e_1e_2=e_3$ & $e_2e_1=e_4$ & $e_3e_1=e_4$ & $e_3e_2=e_4$ \\
$\mathbb{L}_{10}$ & $: $ & $e_1e_1=e_3+e_5$ & $e_1e_2=e_3$ & $e_2e_1=e_5$ & $e_3e_1=e_4$ & $e_3e_2=e_4$ \\
$\mathbb{L}_{11}$ & $: $ & $e_1e_1=e_3+e_4+e_5$ & $e_1e_2=e_3$ & $e_2e_1=e_5$ & $e_3e_1=e_4$ & $e_3e_2=e_4$ \\
$\mathbb{L}_{12}$ & $: $ & $e_1e_1=e_3$ & $e_1e_2=e_3$ & $e_2e_1=e_5$ & $e_3e_1=e_4$ & $e_3e_2=e_4$ \\
$\mathbb{L}_{13}$ & $: $ & $e_1e_1=e_3+e_4$ & $e_1e_2=e_3$ & $e_2e_1=e_5$ & $e_3e_1=e_4$ & $e_3e_2=e_4$ \\
$\mathbb{L}_{14}^{\alpha}$ & $: $ & $e_1e_1=e_3+e_4+\alpha e_5$ & $e_1e_2=e_3$ & $e_2e_2=e_5$ & $e_3e_1=e_4$ & $e_3e_2=e_4$ \\
$\mathbb{L}_{15}$ & $: $ & $e_1e_1=e_3$ & $e_1e_2=e_3$ & $e_2e_2=e_5$ & $e_3e_1=e_4$ & $e_3e_2=e_4$ \\
$\mathbb{L}_{16}$ & $: $ & $e_1e_1=e_3+e_5$ & $e_1e_2=e_3$ & $e_2e_1=e_5$ \\
&& $e_2e_2=e_5$& $e_3e_1=e_4$ & $e_3e_2=e_4$ \\
$\mathbb{L}_{17}$ & $: $ & $e_1e_1=e_3+e_5$ & $e_1e_2=e_3$ & $e_2e_1=e_4+e_5$ \\
&& $e_2e_2=e_5$& $e_3e_1=e_4$ & $e_3e_2=e_4$ \\
$\mathbb{L}_{18}$ & $: $ & $e_1e_1=e_3$ & $e_1e_2=e_3$ & $e_2e_1=e_5$ \\
&& $e_2e_2=e_5$& $e_3e_1=e_4$ & $e_3e_2=e_4$ \\
$\mathbb{L}_{19}$ & $: $ & $e_1e_1=e_3+e_4$ & $e_1e_2=e_3$ & $e_2e_1=e_4+e_5$ \\
&& $e_2e_2=e_5$& $e_3e_1=e_4$ & $e_3e_2=e_4$ \\
$\mathbb{L}_{20}$ & $: $ & $e_1e_1=e_3+e_4$ & $e_1e_2=e_3$ & $e_2e_1=e_5$ \\
&& $e_2e_2=e_5$& $e_3e_1=e_4$ & $e_3e_2=e_4$ \\
$\mathbb{L}_{21}$ & $: $ & $e_1e_1=e_3$ & $e_1e_2=e_3$ & $e_2e_1=e_4+e_5$ \\
&& $e_2e_2=e_5$& $e_3e_1=e_4$ & $e_3e_2=e_4$ \\
$\mathbb{L}_{22}$ & $: $ & $e_1e_1=e_2$ & $e_1e_4=e_5$ & $e_2e_1=e_5$ & $e_3e_3=e_5$ \\
$\mathbb{L}_{23}$ & $: $ & $e_1e_1=e_2$ & $e_2e_1=e_5$ & $e_3e_4=e_5$ & $e_4e_3=-e_5$ \\
$\mathbb{L}_{24}^\alpha$ & $: $ & $e_1e_1=e_2$ & $e_2e_1=e_5$ & $e_3e_3=e_5$ & $e_3e_4=e_5$ & $e_4e_4=\alpha e_5$\\
$\mathbb{L}_{25}$ & $: $ & $e_1e_1=e_2$ & $e_1e_4=e_5$ & $e_2e_1=e_5$ & $e_3e_3=e_5$ & $e_3e_4=e_5$\\
$\mathbb{L}_{26}$ & $: $ & $e_1e_1=e_2$ & $e_2e_1=e_5$ & $e_3e_3=e_5$ & $e_3e_4=e_5$ & $e_4e_3= e_5$\\
$\mathbb{L}_{27}$ & $: $ & $e_1e_1=e_3$ & $e_2e_2=e_4$ & $e_3e_1=e_5$ & $e_4e_2=e_5$\\
$\mathbb{L}_{28}^{\alpha}$ & $: $ & $e_1e_1=e_3$  & $e_1e_2=e_3$ & $e_1e_4=\alpha e_5$ &  $e_2e_2=e_5$ \\
& & $e_3e_1=e_5$ & $e_3e_2=e_5$ & $e_4e_4=e_5$ \\
$\mathbb{L}_{29}$ & $: $ & $e_1e_1=e_3$  & $e_1e_2=e_3+e_5$ & $e_2e_2=e_5$ \\
&& $e_3e_1=e_5$ & $e_3e_2=e_5$ & $e_4e_4=e_5$ \\
$\mathbb{L}_{30}$ & $: $ & $e_1e_1=e_3$  & $e_1e_2=e_3+e_5$ & $e_1e_4=e_5$ &  $e_2e_2=e_5$ \\
& & $e_3e_1=e_5$ & $e_3e_2=e_5$ & $e_4e_4=e_5$ \\
$\mathbb{L}_{31}$ & $: $ & $e_1e_1=e_3$  & $e_1e_2=e_3$ &  $e_3e_1=e_5$ & $e_3e_2=e_5$ & $e_4e_4=e_5$ \\
$\mathbb{L}_{32}$ & $: $ & $e_1e_1=e_3$  & $e_1e_2=e_3$ & $e_1e_4=e_5$\\
& & $e_3e_1=e_5$ & $e_3e_2=e_5$ & $e_4e_4=e_5$ \\
$\mathbb{L}_{33}^{\alpha}$ & $: $ & $e_1e_1=e_3$  & $e_1e_2=e_3+e_5$ & $e_1e_4=\alpha e_5$ \\
&&  $e_3e_1=e_5$ &  $e_3e_2=e_5$ & $e_4e_4=e_5$ \\
$\mathbb{L}_{34}$ & $: $ & $e_1e_1=e_3$  & $e_1e_2=e_3$ &  $e_3e_1=e_5$ & $e_3e_2=e_5$ & $e_4e_2=e_5$ \\
$\mathbb{L}_{35}$ & $: $ & $e_1e_1=e_3$  & $e_1e_2=e_3$ &  $e_1e_4=e_5$ & $e_2e_4=e_5$ \\
&& $e_3e_1=e_5$ &$e_3e_2=e_5$ & $e_4e_2=e_5$ \\
$\mathbb{L}_{36}^\alpha$ & $: $ & $e_1e_1=e_3$  & $e_1e_2=e_3$ & $e_1e_4=\alpha e_5$\\
& & $e_3e_1=e_5$ & $e_3e_2=e_5$ & $e_4e_2=e_5$ \\
$\mathbb{L}_{37}$ & $: $ & $e_1e_1=e_3$  & $e_1e_2=e_3+e_5$ & $e_1e_4=e_5$ \\
&& $e_3e_1=e_5$ & $e_3e_2=e_5$ & $e_4e_2=e_5$ \\
$\mathbb{L}_{38}^{\alpha}$ & $: $ & $e_1e_1=e_3$  & $e_1e_2=e_3$ &  $e_2e_4=\alpha e_5$\\
& & $e_3e_1=e_5$ &$e_3e_2=e_5$ & $e_4e_2=e_5$ \\
$\mathbb{L}_{39}$ & $: $ & $e_1e_1=e_3$  & $e_1e_2=e_3+e_5$ &  $e_2e_4=-e_5$ \\
&& $e_3e_1=e_5$ &$e_3e_2=e_5$ & $e_4e_2=e_5$ \\
$\mathbb{L}_{40}$ & $: $ & $e_1e_1=e_3$  & $e_1e_2=e_3$     &  $e_1e_4=e_5$ \\
&& $e_2e_4=e_5$ &$e_3e_1=e_5$ & $e_3e_2=e_5$ \\
$\mathbb{L}_{41}$ & $: $ & $e_1e_1=e_3$  & $e_1e_2=e_3+e_5$ &  $e_1e_4=e_5$\\
&& $e_2e_4=e_5$ &$e_3e_1=e_5$ & $e_3e_2=e_5$ \\
$\mathbb{L}_{42}$ & $: $ & $e_1e_1=e_3$  & $e_1e_2=e_3$     &  $e_2e_4=e_5$ &$e_3e_1=e_5$ & $e_3e_2=e_5$ \\
$\mathbb{L}_{43}$ & $: $ & $e_1e_1=e_3$  & $e_1e_2=e_3+e_5$ &  $e_2e_4=e_5$ &$e_3e_1=e_5$ & $e_3e_2=e_5$ \\
$\mathbb{L}_{44}$ & $: $ & $e_1e_1=e_3$  & $e_1e_2=e_3$ &  $e_1e_4=e_5$ & $e_3e_1=e_5$ & $e_3e_2=e_5$ \\
$\mathbb{L}_{45}$ & $: $ & $e_1e_1=e_3$  & $e_1e_2=e_3$ &  $e_1e_4=e_5$ \\
&& $e_2e_2=e_5$ &$e_3e_1=e_5$ & $e_3e_2=e_5$ \\
$\mathbb{L}_{46}$ & $: $ & $e_1e_2=e_3$ & $e_2e_1=e_4$ & $e_2e_2=-e_3$ & $e_3e_2=e_5$ & $e_4e_1=e_5$\\
$\mathbb{L}_{47}^{\alpha, \beta}$ & $: $ & $e_1e_1=e_3$  & $e_1e_2=e_4$ & $e_2e_1=-\alpha e_3$ \\
&&  $e_2e_2=-e_4$ & $e_3e_1=e_5$ & $e_4e_2=\beta e_5$ &  \\
$\mathbb{L}_{48}$ & $: $ & $e_1e_1=e_3$ & $e_1e_2=e_4$ & $e_2e_1=-e_3$  & $e_2e_2=-e_4$ \\
&& $e_3e_1=-e_5$ & $e_3e_2=e_5$ & $e_4e_1=e_5$  & $e_4e_2=-e_5$ \\
$\mathbb{L}_{49}$ & $: $ & $e_1e_1=e_3$ & $e_1e_2=e_4$ & $e_2e_1=-e_3$  & $e_2e_2=-e_4+e_5$ \\
&& $e_3e_1=-e_5$ & $e_3e_2=e_5$ & $e_4e_1=e_5$  & $e_4e_2=-e_5$ \\
$\mathbb{L}_{50}$ & $: $  & $e_1e_2=e_3$ &  $e_2e_1=e_4$ &  $e_3e_2=e_5$ & $e_4e_1=e_5$ \\
$\mathbb{L}_{51}$ & $: $  & $e_1e_1=e_5$  & $e_1e_2=e_3$ & $e_1e_3=-e_5$ & $e_1e_4=e_5$ \\
&& $e_2e_1=e_4$ & $e_2e_2=e_5$&   $e_3e_1=e_5$ & $e_4e_1=-e_5$  \\
$\mathbb{L}_{52}^{\alpha, \beta}$ & $: $ & $e_1e_2=e_3$ & $e_1e_3=-e_5$ & $e_1e_4=e_5$ & $e_2e_1=e_4$ &  $e_2e_3=\beta e_5$  \\
&& $e_2e_4=-\beta e_5$ & $e_3e_1=e_5$ & $e_3e_2=e_5$ &  $e_4e_1=\alpha e_5$&  $e_4e_2=\beta e_5$  & \\
$\mathbb{L}_{53}$ & $: $  &  $e_1e_2=e_3$ & $e_2e_1=e_4$ & $e_2e_3=e_5$ & $e_2e_4=-e_5$ & $e_4e_2=e_5$\\
$\mathbb{L}_{54}$ & $: $ & $e_1e_2=e_3$ & $e_1e_3=-e_5$ & $e_1e_4=e_5$ & $e_2e_1=e_4$ & $e_2e_3=e_5$ \\
& & $e_2e_4=-e_5$  & $e_3e_1=e_5$ &  $e_4e_1=-e_5$ & $e_4e_2=e_5$\\
$\mathbb{L}_{55}^\alpha$ & $: $ & $e_1e_2=e_3$ & $e_2e_1=e_4$ & $e_2e_3=\alpha e_5$ \\
&& $e_2e_4=-\alpha e_5$  & $e_3e_2=e_5$ & $e_4e_2=\alpha e_5$\\
$\mathbb{L}_{56}^\alpha$ & $: $ & $e_1e_1=e_5$ & $e_1e_2=e_3$ & $e_2e_1=e_4$ & $e_2e_3=e_5$  \\
&& $e_2e_4=-e_5$ & $e_3e_2=\alpha e_5$ &   $e_4e_2=e_5$\\
$\mathbb{L}_{57}$ & $: $ & $e_1e_1=e_5$ & $e_1e_2=e_3$ & $e_1e_3=-e_5$ & $e_1e_4=e_5$ \\ && $e_2e_1=e_4$ &  $e_3e_1=e_5$ &  $e_3e_2=e_5$ &  $e_4e_1=-e_5$\\
$\mathbb{L}_{58}^\alpha $ & $: $ & $e_1e_1=e_5$  & $e_1e_2=e_3$ &   $e_1e_3=-e_5$ \\
&& $e_1e_4=e_5$ & $e_2e_1=e_4$ & $e_2e_3=e_5$ \\
&& $e_2e_4=-e_5$ &$e_3e_1=e_5$ & $e_3e_2=\alpha e_5$ & $e_4e_1=-e_5$ &  $e_4e_2=e_5$\\
$\mathbb{L}_{59}$ & $: $ & $e_1e_2=e_3$ & $e_2e_1=e_4$ & $e_2e_2=e_5$ &  $e_2e_3=e_5$  \\
&& $e_2e_4=-e_5$  & $e_3e_2=-e_5$ & $e_4e_2=e_5$\\
$\mathbb{L}_{60}$ & $: $ & $e_1e_1=e_5$ & $e_1e_2=e_3$ & $e_2e_1=e_4$ &  $e_2e_2=e_5$  \\
&& $e_2e_3=e_5$   & $e_2e_4=-e_5$ &$e_3e_2=-e_5$ &  $e_4e_2=e_5$\\
$\mathbb{L}_{61}^\alpha$ & $: $ &  $e_1e_1=\alpha e_5$ & $e_1e_2=e_3$ & $e_1e_3=-e_5$ & $e_1e_4=e_5$ \\
&& $e_2e_1=e_4$ & $e_2e_2=e_5$   
&  $e_2e_3=e_5$ & $e_2e_4=-e_5$\\
& & $e_3e_1=e_5$ & $e_3e_2=-e_5$ & $e_4e_1=-e_5$ & $e_4e_2=e_5$\\
$\mathbb{L}_{62}^{\alpha}$ & $: $ & $e_1e_1=e_4$ & $e_1e_2=e_3$ & $e_2e_1=-e_3$ &  $e_2e_2=2e_3+e_4$\\
& & $e_3e_1=\alpha e_5$ &  $e_3e_2=(1-2\alpha)e_5$  &  $e_4e_1=-e_5$ & $e_4e_2=\alpha e_5$ &  \\
$\mathbb{L}_{63}^{\alpha}$ & $: $  & $e_1e_2=e_4$ &  $e_2e_1=\alpha e_4$ & $e_2e_2=e_3$ &  $e_3e_1= \alpha e_5$ &  $e_4e_2= e_5$ \\
$\mathbb{L}_{64}$ & $: $  & $e_1e_2=e_4$ &  $e_2e_1= - \frac 1 2 e_4$ & $e_2e_2=e_3$ \\
&&  $e_3e_1= - \frac 1 2 e_5$ &  $e_3e_2= e_5$ &  $e_4e_2= e_5$ \\
$\mathbb{L}_{65}$ & $: $  & $e_1e_2=e_4$ &  $e_2e_1= - e_4$ & $e_2e_2=e_3$ \\
&&  $e_2e_4= e_5$ &  $e_3e_2= e_5$ &  $e_4e_2= -e_5$ \\
$\mathbb{L}_{66}$ & $: $  & $e_1e_1=e_5$ & $e_1e_2=e_4$ &  $e_2e_1= - e_4$ & $e_2e_2=e_3$ \\ 
&&  $e_2e_4= e_5$ &  $e_3e_2= e_5$   &  $e_4e_2= -e_5$ \\
$\mathbb{L}_{67}^{\alpha\neq 0}$ & $: $  & $e_1e_2=e_4$ &  $e_2e_1= - e_4$ & $e_2e_2=e_3$ \\
&&  $e_2e_4= e_5$ &  $e_3e_1= -\alpha e_5$ &  $e_4e_2= (\alpha-1)e_5$ \\
$\mathbb{L}_{68}$ & $: $ & $e_1e_2=e_4$ &  $e_1e_4= e_5$ &  $e_2e_1= - e_4$ \\
&& $e_2e_2=e_3$ &   $e_3e_2= e_5$ &  $e_4e_1= -e_5$ \\
$\mathbb{L}_{69}$ & $: $ & $e_1e_1=e_5$ & $e_1e_2=e_4$ &  $e_1e_4= e_5$ &  $e_2e_1= - e_4$ \\ 
&& $e_2e_2=e_3$ &   $e_3e_2= e_5$   &  $e_4e_1= -e_5$ \\
$\mathbb{L}_{70}^{\alpha}$ & $: $ & $e_1e_2=e_4$ &  $e_1e_4= e_5$ &  $e_2e_1= - e_4$ & $e_2e_2=e_3$\\ 
& &  $e_3e_1= - e_5$ &   $e_3e_2= \alpha e_5$   &  $e_4e_1= -e_5$ & $e_4e_2= e_5$ \\
$\mathbb{L}_{71}$ & $: $  & $e_1e_1=e_2$ &  $e_1e_4=e_5$ & $e_2e_1=e_3$ & $e_3e_1=e_5$ &   \\
$\mathbb{L}_{72}$ & $: $  & $e_1e_1=e_2$ &  $e_2e_1=e_3$ & $e_3e_1=e_5$ &  $e_4e_4=e_5$ \\
$\mathbb{L}_{73}$ & $: $  & $e_1e_1=e_3$ &  $e_2e_1=e_3$ & $e_3e_1=e_4$ &  $e_4e_1=e_5$ \\
$\mathbb{L}_{74}$ & $: $  & $e_1e_1=e_3$ & $e_1e_2=e_5$ &  $e_2e_1=e_3$ & $e_3e_1=e_4$ &  $e_4e_1=e_5$  \\
$\mathbb{L}_{75}$ & $: $  & $e_1e_1=e_3$ &  $e_2e_1=e_3$ & $e_2e_2=e_5$ & $e_3e_1=e_4$ &  $e_4e_1=e_5$  \\
$\mathbb{L}_{76}$ & $: $  & $e_1e_1=e_3$ & $e_1e_2=e_5$ &  $e_2e_1=e_3$\\
& & $e_2e_2=e_5$ & $e_3e_1=e_4$ &  $e_4e_1=e_5$   \\
$\mathbb{L}_{77}$ & $: $  & $e_1e_1=e_3$ & $e_1e_2=e_4$ &  $e_2e_1=e_3$ \\ 
& & $e_2e_2=e_4$ & $e_3e_1=e_4$ &  $e_3e_2=e_5$ & $e_4e_1=e_5$ \\
$\mathbb{L}_{78}$ & $: $  & $e_1e_1=e_3$ & $e_1e_2=e_4$ & $e_2e_1=e_3$ \\ 
& & $e_2e_2=e_4$ &  $e_3e_1=e_4+2e_5$ & $e_3e_2=e_5$ & $e_4e_1=e_5$ \\
$\mathbb{L}_{79}^{\alpha}$ & $: $  & $e_1e_1=e_3$ & $e_1e_2=e_4$ & $e_2e_1=e_3$ \\
& & $e_2e_2=e_4+e_5$ &  $e_3e_1=e_4+\alpha e_5$ & $e_3e_2=e_5$ & $e_4e_1=e_5$ \\
$\mathbb{L}_{80}^{\alpha}$ & $: $  & $e_1e_1=e_3$ &  $e_1e_2=e_4$ & $e_2e_2=\alpha e_5$ \\
&& $e_3e_1=e_4$ &  $e_3e_2=e_5$ & $e_4e_1=e_5$ \\
$\mathbb{L}_{81}$ & $: $  & $e_1e_1=e_3$ &  $e_1e_2=e_4$ &  $e_2e_2=2e_5$ \\
&& $e_3e_1=e_4+e_5$ & $e_3e_2=e_5$ & $e_4e_1=e_5$\\
$\mathbb{L}_{82}$ & $: $  & $e_1e_1=e_2$ &  $e_2e_1=e_3$ & $e_3e_1=e_4$ & $e_4e_1=e_5$\\
\end{longtable} }
Note that $\mathbb{L}_{47}^{\alpha,\beta\neq 0}\simeq \mathbb{L}_{47}^{\alpha,\alpha^3{\beta}^{-1}}$, $\mathbb{L}_{47}^{1,\beta\neq 0,-1}\simeq \mathbb{L}_{47}^{1,1}$, $\mathbb{L}_{52}^{\alpha\neq 0,\beta\neq0}\simeq \mathbb{L}_{52}^{{\beta}^{-1},{\alpha}^{-1}},$ $\mathbb{L}_{62}^{\alpha\neq \frac{2}{3}}  \simeq \mathbb{L}_{62}^\frac{2\alpha-1}{3\alpha-2};$ 
and  $\mathbb{L}_{67}^{0}$ is a symmetric Leibniz algebra with a two-dimensional annihilator.

\end{theoremA}

\section{The geometric classification of nilpotent  Leibniz algebras}

\subsection{Definitions and notation}
Given an $n$-dimensional vector space $\mathbb V$, the set ${\rm Hom}(\mathbb V \otimes \mathbb V,\mathbb V) \cong \mathbb V^* \otimes \mathbb V^* \otimes \mathbb V$ is a vector space of dimension $n^3$. This space has the structure of the affine variety $\mathbb{C}^{n^3}.$ Indeed, let us fix a basis $e_1,\dots,e_n$ of $\mathbb V$. Then any $\mu\in {\rm Hom}(\mathbb V \otimes \mathbb V,\mathbb V)$ is determined by $n^3$ structure constants $c_{ij}^k\in\mathbb{C}$ such that
$\mu(e_i\otimes e_j)=\sum\limits_{k=1}^nc_{ij}^ke_k$. A subset of ${\rm Hom}(\mathbb V \otimes \mathbb V,\mathbb V)$ is {\it Zariski-closed} if it can be defined by a set of polynomial equations in the variables $c_{ij}^k$ ($1\le i,j,k\le n$).

Let $T$ be a set of polynomial identities.
The set of algebra structures on $\mathbb V$ satisfying polynomial identities from $T$ forms a Zariski-closed subset of the variety ${\rm Hom}(\mathbb V \otimes \mathbb V,\mathbb V)$. We denote this subset by $\mathbb{L}(T)$.
The general linear group ${\rm GL}(\mathbb V)$ acts on $\mathbb{L}(T)$ by conjugations:
$$ (g * \mu )(x\otimes y) = g\mu(g^{-1}x\otimes g^{-1}y)$$
for $x,y\in \mathbb V$, $\mu\in \mathbb{L}(T)\subset {\rm Hom}(\mathbb V \otimes\mathbb V, \mathbb V)$ and $g\in {\rm GL}(\mathbb V)$.
Thus, $\mathbb{L}(T)$ is decomposed into ${\rm GL}(\mathbb V)$-orbits that correspond to the isomorphism classes of algebras.
Let $O(\mu)$ denote the orbit of $\mu\in\mathbb{L}(T)$ under the action of ${\rm GL}(\mathbb V)$ and $\overline{O(\mu)}$ denote the Zariski closure of $O(\mu)$.

Let $\bf A$ and $\bf B$ be two $n$-dimensional algebras satisfying the identities from $T$, and let $\mu,\lambda \in \mathbb{L}(T)$ represent $\bf A$ and $\bf B$, respectively.
We say that $\bf A$ degenerates to $\bf B$ and write $\bf A\to \bf B$ if $\lambda\in\overline{O(\mu)}$.
Note that in this case we have $\overline{O(\lambda)}\subset\overline{O(\mu)}$. Hence, the definition of degeneration does not depend on the choice of $\mu$ and $\lambda$. If $\bf A\not\cong \bf B$, then the assertion $\bf A\to \bf B$ is called a {\it proper degeneration}. We write $\bf A\not\to \bf B$ if $\lambda\not\in\overline{O(\mu)}$.

Let $\bf A$ be represented by $\mu\in\mathbb{L}(T)$. Then  $\bf A$ is  {\it rigid} in $\mathbb{L}(T)$ if $O(\mu)$ is an open subset of $\mathbb{L}(T)$.
 Recall that a subset of a variety is called irreducible if it cannot be represented as a union of two non-trivial closed subsets.
 A maximal irreducible closed subset of a variety is called an {\it irreducible component}.
It is well known that any affine variety can be represented as a finite union of its irreducible components in a unique way.
The algebra $\bf A$ is rigid in $\mathbb{L}(T)$ if and only if $\overline{O(\mu)}$ is an irreducible component of $\mathbb{L}(T)$.




\subsection{Method of the description of  degenerations of algebras}

In the present work we use the methods applied to Lie algebras in \cite{BC99,GRH,GRH2,S90}.
First of all, if $\bf A\to \bf B$ and $\bf A\not\cong \bf B$, then $\mathfrak{Der}(\bf A)<\mathfrak{Der}(\bf B)$, where $\mathfrak{Der}(\bf A)$ is the Lie algebra of derivations of $\bf A$. We compute the dimensions of algebras of derivations and check the assertion $\bf A\to \bf B$ only for such $\bf A$ and $\bf B$ that $\mathfrak{Der}(\bf A)<\mathfrak{Der}(\bf B)$.


To prove degenerations, we construct families of matrices parametrized by $t$. Namely, let $\bf A$ and $\bf B$ be two algebras represented by the structures $\mu$ and $\lambda$ from $\mathbb{L}(T)$ respectively. Let $e_1,\dots, e_n$ be a basis of $\mathbb  V$ and $c_{ij}^k$ ($1\le i,j,k\le n$) be the structure constants of $\lambda$ in this basis. If there exist $a_i^j(t)\in\mathbb{C}$ ($1\le i,j\le n$, $t\in\mathbb{C}^*$) such that $E_i^t=\sum\limits_{j=1}^na_i^j(t)e_j$ ($1\le i\le n$) form a basis of $\mathbb V$ for any $t\in\mathbb{C}^*$, and the structure constants of $\mu$ in the basis $E_1^t,\dots, E_n^t$ are such rational functions $c_{ij}^k(t)\in\mathbb{C}[t]$ that $c_{ij}^k(0)=c_{ij}^k$, then $\bf A\to \bf B$.
In this case  $E_1^t,\dots, E_n^t$ is called a {\it parametrized basis} for $\bf A\to \bf B$.
To simplify our equations, we will use the notation $A_i=\langle e_i,\dots,e_n\rangle,\ i=1,\ldots,n$ and write simply $A_pA_q\subset A_r$ instead of $c_{ij}^k=0$ ($i\geq p$, $j\geq q$, $k< r$).

Since the variety of $5$-dimensional nilpotent  Leibniz algebras  contains infinitely many non-isomorphic algebras, we have to do some additional work.
Let $\bf A(*):=\{\bf A(\alpha)\}_{\alpha\in I}$ be a series of algebras, and let $\bf B$ be another algebra. Suppose that for $\alpha\in I$, $\bf A(\alpha)$ is represented by the structure $\mu(\alpha)\in\mathbb{L}(T)$ and $B\in\mathbb{L}(T)$ is represented by the structure $\lambda$. Then we say that $\bf A(*)\to \bf B$ if $\lambda\in\overline{\{O(\mu(\alpha))\}_{\alpha\in I}}$, and $\bf A(*)\not\to \bf B$ if $\lambda\not\in\overline{\{O(\mu(\alpha))\}_{\alpha\in I}}$.

Let $\bf A(*)$, $\bf B$, $\mu(\alpha)$ ($\alpha\in I$) and $\lambda$ be as above. To prove $\bf A(*)\to \bf B$ it is enough to construct a family of pairs $(f(t), g(t))$ parametrized by $t\in\mathbb{C}^*$, where $f(t)\in I$ and $g(t)\in {\rm GL}(\mathbb V)$. Namely, let $e_1,\dots, e_n$ be a basis of $\mathbb V$ and $c_{ij}^k$ ($1\le i,j,k\le n$) be the structure constants of $\lambda$ in this basis. If we construct $a_i^j:\mathbb{C}^*\to \mathbb{C}$ ($1\le i,j\le n$) and $f: \mathbb{C}^* \to I$ such that $E_i^t=\sum\limits_{j=1}^na_i^j(t)e_j$ ($1\le i\le n$) form a basis of $\mathbb V$ for any  $t\in\mathbb{C}^*$, and the structure constants of $\mu_{f(t)}$ in the basis $E_1^t,\dots, E_n^t$ are such rational functions $c_{ij}^k(t)\in\mathbb{C}[t]$ that $c_{ij}^k(0)=c_{ij}^k$, then $\bf A(*)\to \bf B$. In this case  $E_1^t,\dots, E_n^t$ and $f(t)$ are called a parametrized basis and a {\it parametrized index} for $\bf A(*)\to \bf B$, respectively.

We now explain how to prove $\bf A(*)\not\to\bf  B$.
Note that if $\mathfrak{Der} \ \bf A(\alpha)  > \mathfrak{Der} \  \bf B$ for all $\alpha\in I$ then $\bf A(*)\not\to\bf B$.
One can also use the following  Lemma, whose proof is the same as the proof of \cite[Lemma 1.5]{GRH}.

\begin{lemma}\label{gmain}
Let $\mathfrak{B}$ be a Borel subgroup of ${\rm GL}(\mathbb V)$ and $\mathcal{R}\subset \mathbb{L}(T)$ be a $\mathfrak{B}$-stable closed subset.
If $\bf A(*) \to \bf B$ and for any $\alpha\in I$ the algebra $\bf A(\alpha)$ can be represented by a structure $\mu(\alpha)\in\mathcal{R}$, then there is $\lambda\in \mathcal{R}$ representing $\bf B$.
\end{lemma}

\subsection{The geometric classification of $5$-dimensional nilpotent 
  Leibniz algebras}
The main result of the present section is the following theorem.

\begin{theoremB}\label{geo3}
The variety of $5$-dimensional nilpotent  Leibniz     algebras  has 
dimension  $24$   and it has 
$10$  irreducible components
defined by  
 
\begin{center}
$\mathcal{C}_1=\overline{\{\mathcal{O}({\mathfrak V}_{4+1})\}},$ \
$\mathcal{C}_2=\overline{\{\mathcal{O}({\mathfrak V}_{3+2})\}},$ \
$\mathcal{C}_3=\overline{\{\mathcal{O}(\mathbb{S}_{21}^{\alpha,\beta})\}},$ \  
$\mathcal{C}_4=\overline{\{\mathcal{O}(\mathbb{S}_{22}^{\alpha})\}},$   \
$\mathcal{C}_5=\overline{\{\mathcal{O}(\mathbb{S}_{41}^{\alpha})\}},$   \\
$\mathcal{C}_6=\overline{\{\mathcal{O}(\mathbb{L}_{28}^{\alpha})\}},$ \
$\mathcal{C}_7=\overline{\{\mathcal{O}(\mathbb{L}_{47}^{\alpha, \beta})\}},$ \  
$\mathcal{C}_8=\overline{\{\mathcal{O}(\mathbb{L}_{52}^{\alpha, \beta})\}},$  \ 
$\mathcal{C}_9=\overline{\{\mathcal{O}(\mathbb{L}_{79}^{\alpha})\}},$   \
$\mathcal{C}_{10}=\overline{\{\mathcal{O}(\mathbb{L}_{82}))\}}.$   

\end{center}

In particular, there is only one  rigid algebra in this variety.
 
\end{theoremB}

\begin{proof}
Thanks to \cite{ak21} the variety of $5$-dimensional  nilpotent symmetric  Leibniz algebras has only six irreducible components defined by 

\begin{longtable}{lllllll}
${\mathfrak V}_{4+1}$ & $:$&  
$e_1e_2=e_5$& $e_2e_1=\lambda e_5$ &$e_3e_4=e_5$&$e_4e_3=\mu e_5$\\

${\mathfrak V}_{3+2}$ &$ :$&
$e_1e_1 =  e_4$& $e_1e_2 = \mu_1 e_5$ & $e_1e_3 =\mu_2 e_5$& 
$e_2e_1 = \mu_3 e_5$  & $e_2e_2 = \mu_4 e_5$  \\
& & $e_2e_3 = \mu_5 e_5$  & $e_3e_1 = \mu_6 e_5$  & \multicolumn{2}{l}{$e_3e_2 = \lambda e_4+ \mu_7 e_5$ } & $e_3e_3 =  e_5$  \\

${\mathfrak V}_{2+3}$ &$ :$&
$e_1e_1 = e_3 + \lambda e_5$& $e_1e_2 = e_3$ & $e_2e_1 = e_4$& $e_2e_2 = e_5$\\

$\mathbb{S}_{21}^{\alpha,\beta}$ &$:$& 
$e_1e_1=\alpha e_5$ &  $e_1 e_2 =e_3+e_4+\beta e_5$  & $e_1e_3=e_5$ & $e_2 e_1 =-e_3$ \\&& $e_2e_2=e_5$  & $e_2e_3 =e_4$ & $e_3e_1=-e_5$ & $e_3e_2=-e_4$\\
 
$\mathbb{S}_{22}^{\alpha}$ &$:$& $e_1e_1=e_5$ &  $e_1 e_2 =e_3$  & $e_1e_3=e_5$ & $e_2 e_1 =-e_3$ \\&& $e_2e_2=\alpha e_5$  &  $e_2e_4=e_5$ & $e_3e_1=-e_5$ & $e_4e_4=e_5$\\

$\mathbb{S}_{41}^{\alpha}$ &$:$& $e_1e_1=e_5$ & $e_1 e_2 =e_3$  & $e_1e_3=e_5$ & $e_2 e_1 =-e_3$ & $e_2e_2=\alpha e_5$ \\&& $e_2e_3=e_4$ & $e_2e_4=e_5$ &  $e_3e_1=-e_5$ & $e_3e_2=-e_4$ & $e_4e_2=-e_5$\\

\end{longtable}

Thanks to \cite{kppv,ikp20}, 
all $5$-dimensional split  Leibniz  algebras are in orbit closure 
of families 
${\mathfrak V}_{2+3}$, ${\mathfrak V}_{3+2}$ and ${\mathfrak V}_{4+1};$
and 
$\mathfrak{L}_5 \oplus \mathbb C,$
$\mathfrak{L}_{11} \oplus \mathbb C,$
$\mathfrak{L}_{13} \oplus \mathbb C$.  Note that
$\mathfrak{L}_{11} \oplus \mathbb C$ is a symmetric Leibniz algebra.

After carefully  checking  the dimensions of orbit closures of the more important for us algebras, we have 

\begin{center}  
$\dim  \mathcal{O}({\mathfrak V}_{3+2})=24,$ \, 
$\dim \mathcal{O}(\mathbb{L}_{28}^{\alpha})=
\dim \mathcal{O}(\mathbb{L}_{47}^{\alpha, \beta})= 
\dim \mathcal{O}(\mathbb{L}_{52}^{\alpha, \beta})=   
\dim \mathcal{O}(\mathbb{L}_{79}^{\alpha})=22,$   \\

$\dim \mathcal{O}(\mathbb{S}_{21}^{\alpha,\beta})=\dim \mathcal{O}(\mathbb{S}_{22}^{\alpha}) = 21,$\,
$\dim \mathcal{O}(\mathbb{S}_{41}^{\alpha})=
\dim \mathcal{O}({\mathfrak V}_{4+1})=
\dim \mathcal{O}(\mathbb{L}_{82})=20.$

 \end{center}

 Hence, 
${\mathfrak V}_{3+2},$
$\mathbb{L}_{28}^{\alpha},$
$\mathbb{L}_{47}^{\alpha, \beta},$
$\mathbb{L}_{52}^{\alpha, \beta},$   
$\mathbb{L}_{79}^{\alpha}$ give $5$ irreducible components.
It is known that if $\bf A(*) \to \bf B$ and $\bf B$ is one-generated, 
than $\bf A$ is one-generated.
There is only one one-generated 
$5$-dimensional nilpotent Leibniz algebra.
Hence, $\mathbb{L}_{82}$ gives an irreducible component and is rigid.

Below we have listed all   reasons for the rest of the necessary non-degenerations.

\begin{longtable}{lcl|l}
\hline
    \multicolumn{4}{c}{Non-degenerations reasons} \\
\hline

$\mathbb{L}_{28}  $&$   \not \rightarrow  $ & $\left\{\begin{array}{l}
\mathbb{S}_{41}^{\alpha},  \mathfrak{V}_{4+1},\\ 
\mathbb{S}_{21}^{\alpha,\beta},
\mathbb{S}_{22}^{\alpha}\\
\end{array} \right.$ & 

$\mathcal R=\left\{\begin{array}{ll}
A_1A_1 \subset A_4,\  A_1A_3+A_2A_1 \subset A_5,\
A_1A_4 +A_4A_3 = 0, \\
c_{42}^5 c_{11}^4 = c_{12}^4 c_{41}^5\\
\mbox{new basis}: \ f_1=e_1, \ f_2=e_2, \ f_3=e_4, \ f_4=e_3, \ f_5=e_5 
\end{array}
\right\}$\\

\hline
$\mathbb{L}_{47}^{\alpha,\beta } $&$   \not \rightarrow  $ & $\left\{\begin{array}{l}
\mathbb{S}_{41}^{\alpha}, 
\mathfrak{V}_{4+1},\\ \mathbb{S}_{21}^{\alpha,\beta},
\mathbb{S}_{22}^{\alpha}\\
\end{array}\right.$ & 
$\mathcal R=\left\{\begin{array}{ll}
A_1A_1 \subset A_3, \  A_1A_2 \subset A_4, \
A_1A_3 = 0, \ A_3A_1 \subset A_5, \\
c_{21}^{5} c_{42}^5+c_{21}^3 c_{32}^5 = c_{22}^4 c_{41}^5, \ 
c_{11}^4 c_{42}^5 +c_{11}^3 c_{32}^5 = c_{12}^4 c_{41}^5\\ 
\end{array}
\right\}$\\

\hline

$\mathbb{L}_{52}^{\alpha,\beta } $&$   \not \rightarrow  $ & $\left\{\begin{array}{l}
\mathbb{S}_{41}^{\alpha}, 
\mathfrak{V}_{4+1},\\ \mathbb{S}_{21}^{\alpha,\beta},
\mathbb{S}_{22}^{\alpha}\\
\end{array}\right.$ & 
$\mathcal R=\left\{\begin{array}{ll}
A_1A_1 \subset A_3, \  A_2A_1 \subset A_4, \\  A_1A_3+A_3A_1+A_2A_2  \subset A_5, \ A_3^2=0,   \\  
c_{21}^{4} c_{11}^3+c_{11}^3c_{12}^4=c_{11}^4c_{12}^3,\
c_{24}^5 c_{11}^4 =- c_{11}^3c_{23}^5, \\
c_{11}^3c_{23}^5=c_{42}^5c_{11}^4,\
c_{14}^5 c_{11}^4 =- c_{11}^3c_{13}^5, \
c_{14}^5c_{23}^5= c_{13}^5c_{24}^5\\
\end{array}\right\}$\\

\hline

$\mathbb{L}_{79}^{\alpha} $&$  \not \rightarrow  $ & $\left\{\begin{array}{l}
\mathbb{S}_{41}^{\alpha},
\mathfrak{V}_{4+1},\\ \mathbb{S}_{21}^{\alpha,\beta},
\mathbb{S}_{22}^{\alpha}\\
\end{array}\right.$ & 
$\mathcal R=\left\{\begin{array}{ll}
A_1A_1 \subset A_3, \ A_1A_2+A_3A_1 \subset A_4, \
A_3A_2 \subset A_5, \\
A_1A_3 +A_4A_2= 0, \
c_{22}^4 c_{11}^3=c_{12}^4 c_{21}^3,\ 
c_{41}^5c_{12}^4 = c_{11}^3c_{32}^5\\
\end{array}\right\}$\\
\hline

\end{longtable}

The rest of the degenerations is given below on the following two tables and it completes the proof of the Theorem.
 
 {\footnotesize 
\begin{longtable}{lclll} \hline

$\mathbb{L}_{47}^{\frac{1}{2},0}$& $\to$ &
 $\mathfrak{L}_5 \oplus \mathbb C$ &
$E_1^t=-2t e_1-2te_2$ & $E_2^t=2 t e_2+2t^2e_3$ \\ 

\multicolumn{3}{l}{$E_3^t=2 t^2 e_3+4t^2e_4-4t^3 e_5$} &  $E_4^t=-4t^{3}e_5$ & $E_5^t=4 t e_4-4 t^2 e_5$
\\  \hline

$\mathbb{L}_{05}^{{t}^{-1}}$ & $\to$ & ${\mathbb L}_{02}$ & 
$E_1^t=t e_1-te_2+te_3$ & $E_2^t=t^2e_2+(t+t^2)e_5$ \\ 

\multicolumn{3}{l}{$E_3^t=(t^{3}-t^2)e_2+t^2e_3$} &  $E_4^t=t^{3}e_4$ & $E_5^t=t^{2}e_5$
\\  \hline

$\mathbb{L}_{05}^{\frac{1}{t}}$ & $ \to$  & ${\mathbb L}_{04}$ &
$ E_1^t=(1-t^2)e_1+(t-t^3)e_3 $& $E_2^t=(1-t^2)^2e_2+(t-t^3)e_3$\\
\multicolumn{3}{l}{$E_3^t=(1-t^2)e_3$} & $E_4^t=(1-t^2)^{2}e_4$ & $E_5^t=\frac{(1-t^2)^2}{t}e_5$
\\  \hline

$\mathbb{L}_{47}^{\frac{(\alpha-1-t)t}{\alpha (1+2t)},\frac{(\alpha-1-t)^3}{(\alpha+\alpha^2) (1+2t)}} $&$ \to $&$ {\mathbb L}_{05}^{\alpha}$ & 
$ E_1^t=(1-\alpha+t)e_1-\alpha e_2$ &

$E_2^t=\frac{(1-\alpha+t)^2}{1+2t}e_2-\alpha e_4 $\\ 

\multicolumn{4}{l}{$E_3^t=-te_2-\frac{\alpha (1-\alpha+t)^2(1+t)}{(1+\alpha)(1+2t)t}e_3+\frac{\alpha(1+t)}{t}e_4$ \quad $E_4^t=\frac{(1-\alpha+t)^3}{(1+\alpha)(1+2t)}e_5$} & $ E_5^t=-te_4$
\\  \hline

$\mathbb{L}_{79}^{3-\frac{2t+1}{\sqrt{t^2-t}}}$&$ \to$&$ {\mathbb L}_{08}$ & 
\multicolumn{2}{l}{
   $E_1^t=\sqrt{\frac{t^3}{t-1}}e_1+\frac{t(t+\sqrt{t^2-t})}{t-1}e_2+\frac{t^{3/2}(2t-1-\sqrt{t^2-t})}{\sqrt{t-1}}e_3+\frac{3t^2+2t-1-t\sqrt{t^2-t}}{t-1}e_4 $}\\
\multicolumn{5}{l}{
$\begin{array}{ll}
E_2^t=\sqrt{\frac{t^7}{(t-1)^3}}e_3+\frac{t^3(t+1-2\sqrt{t^2-t})}{(t-1)^2}e_4+\frac{(2t^4+3t^3)\sqrt{t-1}-(3t^2-2t+1)t^2\sqrt{t}}{\sqrt{(t-1)^5}}e_4 &  E_3^t=\frac{t^2}{t-1}e_2+\frac{t^{5/2}(1-\sqrt{t^2-t})}{\sqrt{(t-1)^3}}e_3\\
E_4^t=\frac{t^5}{(t-1)^2}e_4  &   E_5^t=\sqrt{\frac{t^{11}}{(t-1)^5}}e_5
\end{array}$}
\\  \hline

$\mathbb{L}_{47}^{1-t, \ (t-1)^3t^3}$&$ \to $&$  {\mathbb L}_{09}$ &
\multicolumn{2}{l}{$E_1^t=(t^2-t^3)e_1+(t^2-t)e_2+(t^2-t^3)^2e_3+te_4$}\\

\multicolumn{5}{l}{$\begin{array}{ll}
E_2^t=(t^2-t^3)e_1+t^2e_2+(t^2-t^3)^2e_3+te_4 & E_3^t=(1-t)^2t^3e_3\\ 
E_4^t=(1-t)^3t^5e_5 & E_5^t=-t^2e_4+(1-t)^3t^5e_5
\end{array}$}
\\  \hline

$\mathbb{L}_{13}$ & $ \to $ & $ {\mathbb L}_{11}$ &
$
E_1^t=\frac{t^2}{1-t}e_1+\frac{t^2+t}{t-1}e_2$ & $E_2^t=\frac{t}{t-1}e_2$\\
\multicolumn{4}{l}{$E_3^t=-\frac{t^3}{(t-1)^2}e_3+\frac{t^5}{(t-1)^3}e_4-\frac{t^4}{(t-1)^2}e_5$ \quad  
$E_4^t=\frac{t^4}{(1-t)^3}e_4$} &   $E_5^t=-\frac{t^3}{(t-1)^2}e_5$
\\  \hline

$\mathbb{L}_{14}^{\frac{1}{t-1}} $ & $\to $ & $ {\mathbb L}_{13}$ & 
$
E_1^t=\frac{(t-1)^2}{t^2}e_1+\frac{t-1}{t^2}e_2$ & $E_2^t=\frac{t-1}{t}e_2$ \\
\multicolumn{3}{l}{$E_3^t=\frac{(t-1)^3}{t^3}e_3+\frac{(t-1)^2}{t^3}e_5$}& $E_4^t=\frac{(t-1)^4}{t^4}e_4$ & $E_5^t=\frac{(t-1)^2}{t^3}e_5$
\\  \hline

\multicolumn{5}{l}{$\mathbb{L}_{47}^{\frac{\alpha-t-t^2}{\alpha+1+t}, \ -\frac{t(t^2+t-\alpha)^3}{(\alpha+1+t)^2(\alpha t^2-t-1)}} \to \mathbb{L}_{14}^{\alpha}$} \\

\multicolumn{5}{l}{ 
{\tiny $E_1^t=(\alpha-t-t^2)e_1-(t^2+\alpha t+t)e_2+\frac{\alpha^3 (1+t)^2-\alpha^2(3t+9 t^2 +10t^3+3t^4)+\alpha(t+4t^2+12t^3+18t^4+12t^5+3 t^6)-t^3(1+t)^5}{(1+t)(t^2+t-\alpha)(\alpha t^2-t-1)}e_3+$}}\\

\multicolumn{5}{r}{\tiny $\frac{(1+\alpha+t)(\alpha(t+t^2-\alpha t^3)+(1+t)^2 (\alpha-t-t^2)^3)}{(t+t^2)(t^2+t-\alpha)^2(\alpha -t-t^2)}e_4$}\\

\multicolumn{5}{l}{\tiny 
$\begin{array}{ll} 
E_2^t=(\alpha-t-t^2)e_1+(\alpha+1+t)e_2+\frac{(\alpha+1+t)(\alpha t^2-t-1)}{(\alpha-t-t^2)^3}e_4 &   E_3^t=(1+t)(t^2+t-\alpha)^2e_3+\alpha (1+t)(\alpha+1+t)e_4 \\
 E_4^t=\frac{(1+t)^2(t^2+t-\alpha)^3}{\alpha t^2-t-1} e_5  & E_5^t=-(1+t)^2(\alpha+1+t)e_4+te_5
 \end{array}$}
\\  \hline

\multicolumn{3}{l}{$\mathbb{L}_{47}^{1-t^2,\ \frac{(t^2-1)^3}{t}}\to  \mathbb{L}_{17}$} & 
\multicolumn{2}{l}{$E_1^t=(t^3-t)e_1-(t^2+t)e_2-t(t^2-1)^2e_3-(t^3+t^2)e_4 $} \\
\multicolumn{5}{l}{
$\begin{array}{ll}
E_2^t=(t^3-t)e_1-te_2-t(t^2-1)^2e_3-(t^3+t^2)e_4 & 
E_3^t=-t^3(t^2-1)^2e_3-(t^4+t^3)e_4+t^3(t^2-1)^3e_5 \\
 E_4^t=t^3(t^2-1)^3e_5 &
E_5^t=-t^4e_4+t^3(t^2-1)^3e_5
\end{array}$}
\\  \hline

$\mathbb{L}_{17} $ &  $\to$ & $ {\mathbb L}_{19}$ &

$E_1^t=te_1+(1-t)e_2$ & 
$E_2^t= e_2$\\
\multicolumn{3}{l}{$E_3^t=te_3-t^2e_4+(1-2t)e_5$} & 
$E_4^t=te_4$ & $E_5^t= e_5$
\\  \hline

$\mathbb{L}_{14}^{0} $ & $\to$ & 
   $ {\mathbb L}_{20}$ & 

$E_1^t=t^2e_1-(t^2-t)e_2$ & $E_2^t= te_2$\\
\multicolumn{3}{l}{$E_3^t=t^3e_3-(t^3-t^2)e_5$} &  $E_4^t=t^4e_4$ & $E_5^t= t^2e_5$\\
\hline

$\mathbb{L}_{79}^{\frac{1+3t^4-t^5}{t^4} }$ &$ \to $&$\mathbb{L}_{21}$ & 
\multicolumn{2}{l}{$E_1^t=(t-t^4)e_1+(t^4-t^3-t+1)e_2-\frac{t^3-1}{t^2}e_3+\frac{t^{10}-t^5-t^4+t^3+t^2-1}{t^6}e_4 $}\\  

\multicolumn{5}{l}{
{\tiny
$\begin{array}{ll}
E_2^t=(t-t^4)e_1+(t^4-t)e_2-\frac{t^3-1}{t^2}e_3+\frac{t^{10}-t^5-t^4+t^3+t^2-1}{t^6}e_4 & E_3^t=t(t^3-1)^2e_3-\frac{(t^2-1)(t^3-1)^2}{t}e_4-\frac{(t^4+t^3-t^2-1)(t^3-1)^2}{t^3} e_5  \\
 E_4^t=\frac{(1-t^3)^3}{t^2} e_5  &  E_5^t=\frac{(t^3-1)^2}{t}e_4-\frac{(t^3-t^2-1)(t^3-1)^2}{t^3}e_5
 \end{array}$}}
\\  \hline

$\mathbb{L}_{38}^{\frac{\sqrt{1-4\alpha}+2\alpha-1}{2\alpha}}$&$ \to $&$  {\mathbb L}_{24}^{\alpha}$ &  $E_1^t=te_1 $ & $E_2^t=t^2e_3$ \\

\multicolumn{3}{l}{$E_3^t=\frac{2t^2}{1-\sqrt{1-4\alpha}}e_2 +\frac{\alpha t}{\sqrt{1-4\alpha}}e_4$} & $ E_4^t=\frac{(1+\sqrt{1-4\alpha})t^2}{1-\sqrt{1-4\alpha}}e_2+\frac{(1-\sqrt{1-4\alpha})\alpha t}{2\sqrt{1-4\alpha}}e_4$ & $E_5^t=t^3e_5$

\\  \hline

$\mathbb{L}_{38}^{\frac{1}{t}} $ & $ \to $ & $ {\mathbb L}_{25} $ &
\multicolumn{2}{l}{$ E_1^t=(t-1)^2(1+t)e_1+(t-t^3)e_2+(1+\frac{1}{t})e_3-\frac{1}{t}e_4$}\\  

\multicolumn{5}{l}{
$\begin{array}{ll}
E_2^t=(1-t)^3 (1+t)^2e_3+(t-1)^2 t (1+t)e_4 &
E_3^t=(t-t^3)e_2+(\frac{1}{t}+t^2)e_3-te_4 \\
 E_4^t=(t-1)^2(1+t)e_4   &  E_5^t=(1-t)^3 (1+t)^2e_5
 \end{array}$}
\\  \hline

$\mathbb{L}_{47}^{-\frac{t^2}{t^2+1},\ \frac{t^3}{t^2+1}}$&$ \to $&$  {\mathbb L}_{27}$ & 
$ E_1^t=-(t^3+t)e_1+(t^3+t)e_2$ & $ E_2^t=(t^2+1)e_2 $\\
\multicolumn{3}{l}{$ E_3^t=(t^4+t^2)e_3$} & $E_4^t=-(t^2+1)^2e_4$ & $E_5^t=-t^3(t^2+1)^2e_5$
\\  \hline

$\mathbb{L}_{30} $ & $ \to  $ & ${\mathbb L}_{29}$ &
\multicolumn{2}{l}{$E_1^t=4t^2(t^2-1)e_1+(1-5t^2+4t^4)e_2+2t^2(t^2-1)e_4 $} \\ 

\multicolumn{5}{l}{
\tiny $\begin{array}{ll}
E_2^t=(1-t^2)e_2+(1-2t^2)^2(t^2-1)e_3-2t^2(t^2-1)e_4 &
E_3^t=4 t^2 (t^2-1)^2e_3-2 t^2(t^2-1)e_4+(1-6t^2+13 t^4-12 t^6+4 t^8)e_5 \\ 
E_4^t=2t(1-t^2)e_4   & E_5^t=4t^{2}(1-t^2)^2e_5
\end{array}$}
\\  \hline

$\mathbb{L}_{28}^{\frac{1}{t}}$ & $ \to $ & ${\mathbb L}_{30} $ &  
$E_1^t=t^2e_1+(1-t^2)e_2+(2t^2-1)e_3 $& $ E_2^t=e_2+(t^2-1)e_3$\\
\multicolumn{3}{l}{$E_3^t=t^2e_3 $}&$  E_4^t=te_4$&$ E_5^t=t^2e_5$
\\  \hline

$\mathbb{L}_{30} $ & $ \to $ & $ {\mathbb L}_{32}$ & 
$ E_1^t=\frac{1}{4t^2}e_1+\frac{2t-1}{4t^2}e_2$ & 
$E_2^t=\frac{1}{2t}e_2+\frac{1-4t}{8t^2}e_3$\\
\multicolumn{3}{l}{$E_3^t=\frac{1}{8t^3}e_3$} & 
$ E_4^t=\frac{1}{4t^2}e_4 $&$ E_5^t=\frac{1}{16t^4}$

\\  \hline
$\mathbb{L}_{28}^{\frac{\alpha}{\sqrt{2t-1}}} $ & $ \to $ & $ {\mathbb L}_{33}^{\alpha}$ & 
$E_1^t=\frac{2t-1}{4t^2}e_1+\frac{1}{4t^2}e_2-\frac{1}{8t^3}e_3$ &$E_2^t=\frac{1}{2t}e_2-\frac{1+2t}{8t^2}e_3$\\ 
\multicolumn{3}{l}{$E_3^t=\frac{2t-1}{8t^3}e_3$} & $E_4^t=\frac{\sqrt{2t-1}}{4t^2}e_4$&  $E_5^t=\frac{2t-1}{16t^4}e_5$
\\  \hline

$\mathbb{L}_{33}^{\frac{1}{\sqrt{(t-1)t}}} $ & $ \to$ & $  {\mathbb L}_{34} $ &
\multicolumn{2}{l}{$E_1^t=(t-1)e_1+(2-t)e_2+(1-t)e_3+\frac{(t-1)\sqrt{t-1}}{\sqrt{t}}e_4$}\\
\multicolumn{3}{l}{ 
$E_2^t=e_2 $ \quad $E_3^t=(t-1)e_3 $}  & $ E_4^t=(t-1)e_3+\sqrt{t^2-t}e_4 $ & $E_5^t=(t-1)e_5$
\\  \hline

$\mathbb{L}_{30} $ & $ \to $ & $ {\mathbb L}_{35}$ &
$E_1^t=te_1-\frac{(t-2)t}{t-1}e_2+\frac{(t^2-2)t}{t-1}e_3+\frac{t^2}{1-t}e_4$ & 
$E_2^t=\frac{t}{t-1}e_2+\frac{(t^2-2)t}{t-1}e_3+\frac{t}{1-t}e_4$
\\ 
\multicolumn{3}{l}{$E_3^t=\frac{t^2}{t-1}e_3 $} & $ E_4^t=\frac{t^2}{1-t}e_4 $& $ E_5^t=\frac{t^3}{(1-t)^2}e_5$
\\\hline

$\mathbb{L}_{38}^{-\alpha-t} $ & $ \to $ & $ {\mathbb L}_{36}^{\alpha} $ &
$ E_1^t=(\alpha+t)e_1-\alpha e_2 $&$ E_2^t=te_3$\\ 
\multicolumn{3}{l}{$E_3^t=(\alpha t+t^2)e_3 $}&$  E_4^t=\alpha te_3+t^2e_4 $&$ E_5^t=(\alpha t^2+t^3)e_5$
\\  \hline

$\mathbb{L}_{36}^{\frac{1}{t+1}} $ & $ \to $ & $ {\mathbb L}_{37} $ &
\multicolumn{2}{l}{$E_1^t=(t^2+t)e_1-t^2e_2+(t^3-t)e_3+(t^2+t)e_4 $} \\
 \multicolumn{3}{l}{$ E_2^t=te_2$ \quad $E_3^t=(t^3+t^2)e_3 $}&$  E_4^t=(t^3+t^2)e_4 $&$ E_5^t=(t^4+t^3)e_5$
\\  \hline

$\mathbb{L}_{52}^{\frac{1}{\alpha}, \  \frac{\alpha t}{1-\alpha t}}$&$\to $&$ {\mathbb L}_{38}^{\alpha}$ & 
$E_1^t=\frac{\alpha t \sqrt{\alpha-\alpha^2 t}}{1-\alpha t}e_1+\sqrt{\alpha-\alpha^2 t}e_2 $&$ E_2^t=\sqrt{\alpha-\alpha^2 t} e_2$\\
\multicolumn{3}{l}{$ E_3^t=\alpha^2 t(e_3+e_4) $} &
$ E_4^t=\alpha^2 t(1+\alpha)e_3+\alpha^2 ((1+\alpha)t-1)e_4$ & $ E_5^t=\frac{\alpha^3 t}{\sqrt{\alpha-\alpha^2 t}}e_5$
\\  \hline

$\mathbb{L}_{38}^{t-1} $ & $ \to $ & $ {\mathbb L}_{39} $ & 
\multicolumn{2}{l}{$ E_1^t=te_1+t^2e_4 $ \quad $E_2^t=te_2-(t^3-t^2)e_3+(t^2-t)e_4$}  \\ 
\multicolumn{3}{l}{
$E_3^t=t^2e_3$} & $ E_4^t=t^2e_4$ &  $E_5^t=t^3e_5$
\\  \hline
 
$\mathbb{L}_{43} $ & $ \to $ & $  {\mathbb L}_{41}$ & 
$ E_1^t=\frac{t^2}{1-t}e_1+\frac{t^2+t}{t-1}e_2-\frac{t^2}{(t-1)^2(t+1)}e_4$\\ \multicolumn{5}{l}{
$E_2^t=\frac{t}{t-1}e_2-\frac{t^2(t^3+t^2-1)}{(t-1)^2(t+1)}e_3-\frac{t^2}{(t-1)^2(t+1)}e_4$ \quad 
$E_3^t=-\frac{t^{3}}{(t-1)^2}e_3$ \quad 
$ E_4^t=-\frac{t^{3}}{(t-1)^2}e_4$ \quad $E_5^t=\frac{t^{4}}{(1-t)^3}e_5$}
\\  \hline

$\mathbb{L}_{43} $ & $ \to $ & $  {\mathbb L}_{45}$ & 
$E_1^t=\frac{t-1}{t}e_1+\frac{1}{t}e_2$ & $E_2^t=e_2+(t-1)e_3 -\frac{(t-1)^2}{t}e_4$\\  \multicolumn{3}{l}{$E_3^t=\frac{t-1}{t}e_3+\frac{t-1}{t^2}e_5$} & $ E_4^t=(t-1)e_4$ & $E_5^t=\frac{t-1}{t}e_5$
\\  \hline

$\mathbb{L}_{47}^{\frac{t-1}{t}, \  \frac{(t-1)^3 (1+t)^2}{t^2}} $&$\to  $&$ {\mathbb L}_{46}$ & 
$ E_1^t=te_2$ & $ E_2^t=(t^2-1)e_1+t^2e_2$ 
\\
\multicolumn{3}{l}{$E_3^t=-(t-1)^2(t+1)e_3$} & $E_4^t=-te_4$ & $E_5^t=(1-t)^3(1+t)^2e_5$
\\  \hline

$\mathbb{L}_{47}^{1, \ 1} $ & $ \to  $ & $ {\mathbb L}_{49} $ &
$ E_1^t=t^2e_1+(t^2-t)e_2-te_3-t^2e_4$ & $E_2^t=te_2$\\ 
\multicolumn{3}{l}{$E_3^t=t^3e_3-t^2e_4+t^4e_5$} & $E_4^t=t^2e_4-t^3e_5$ & $E_5^t=-t^3e_5$
\\  \hline

$\mathbb{L}_{52}^{0, \ \frac{1}{\alpha}} $&$\to$&$  {\mathbb L}_{56}^{\alpha\neq0}$& 
$E_1^t=\frac{t}{\alpha}e_1+\frac{1}{\alpha}e_4$ & $ E_2^t=e_2$\\
\multicolumn{3}{l}{$E_3^t=\frac{t}{\alpha}e_3+\frac{1}{\alpha^2}e_5$} & $E_4^t=\frac{t}{\alpha}e_4-\frac{1}{\alpha^2}e_5$ & $ E_5^t=\frac{t}{\alpha^2}e_5$
\\  \hline

$\mathbb{L}_{52}^{t-1, \ t} $&$\to $&$ {\mathbb L}_{57}$ & 
$ E_1^t=te_1+te_4$ & $E_2^t=te_2$ \\ 
\multicolumn{3}{l}{$E_3^t=t^2e_3+t^3e_5$} & $E_4^t=t^2e_4-t^3e_5$ & $ E_5^t=t^3e_5$
\\  \hline

$\mathbb{L}_{52}^{t-1, \ \frac{1}{\alpha}}$&$ \to$&$  {\mathbb L}_{58}^{\alpha}$ & 
$E_1^t=\frac{t}{\alpha}e_1+\frac{t}{\alpha}e_4$ & $ E_2^t=te_2$\\ 
\multicolumn{3}{l}{$E_3^t=\frac{t^2}{\alpha}e_3 +\frac{t^2}{\alpha^2}e_5$} &   
$E_4^t=\frac{t^2}{\alpha}e_4-\frac{t^2}{\alpha^2}e_5$ & $ E_5^t=\frac{t^3}{\alpha^2}e_5$
\\  \hline

$\mathbb{L}_{52}^{0, \ t-1} $&$ \to $&$ {\mathbb L}_{59}$ & 
$E_1^t=(t-1)te_1 $&$  E_2^t=e_2+(t-1)^2e_3$\\ 
\multicolumn{3}{l}{$E_3^t=(t-1)te_3-(t-1)^3te_5$} & 
$E_4^t=(t-1)te_4+(t-1)^3te_5 $ &$ E_5^t=(t-1)^2te_5$
\\  \hline

$\mathbb{L}_{52}^{0, \ t-1} $&$ \to $&   ${\mathbb L}_{60}$ & 
$E_1^t=te_1+(t-1)e_4 $&$  E_2^t=e_2+(t-1)e_3 $\\
\multicolumn{3}{l}{$ E_3^t=te_3-(t-1)e_5$}& $E_4^t=te_4+ (t-1)e_5 $&$ E_5^t=(t-1)te_5$
\\  \hline

$\mathbb{L}_{52}^{t-1, \ t-1} $&$\to $&$  {\mathbb L}_{61}^{\alpha}$  & 
$
E_1^t=(t-1)te_1+\alpha (t-1)te_4 $&\\  

\multicolumn{5}{l}{
$\begin{array}{ll}
E_2^t=te_2+(t-1)^2t  e_3 &  E_3^t=(t-1)t^2e_3+(\alpha+1-t)(t-1)^2t^2e_5 \\  
 E_4^t=(t-1)t^2e_4-(\alpha+1-t)(t-1)^2t^2e_5  &  E_5^t=(t-1)^2t^3e_5
 \end{array}$}
\\  \hline

$\mathbb{L}_{79}^{\frac{6\alpha-4}{\alpha-1} } $&$ \to$&$  \mathbb{L}_{62}^{\alpha}$ & 
$E_1^t=-2te_1+4te_2-4te_3+\frac{4(3\alpha -2)t}{\alpha-1}e_4$ &$ E_2^t=2te_1-\frac{4(3\alpha -2)t}{\alpha-1}e_4 $\\
\multicolumn{3}{l}{$ E_3^t=4t^2e_3-8t^2e_4-\frac{8(3\alpha -2)t}{\alpha-1} e_5 $} & 
 $E_4^t=-4t^2e_3+16t^2e_4-\frac{8(3\alpha -2)t}{\alpha-1} e_5 $&$ E_5^t=\frac{16t^2}{1-\alpha}e_5$
\\  \hline

$\mathbb{L}_{47}^{\frac{2 (1+\alpha)}{2+\alpha},\ -\frac{(1+\alpha)^3}{2+\alpha }} $&$\to $&$  \mathbb{L}_{63}^{\alpha}$ & 
$E_1^t=te_2 $&$ E_2^t=(1+\alpha)e_1+e_2 $\\
\multicolumn{3}{l}{$E_3^t=\frac{\alpha (1+\alpha)^2}{2+\alpha}e_3+\alpha e_4$}&  
$ E_4^t=te_4 $&$ E_5^t=-\frac{(1+\alpha)^3t}{2+\alpha }e_5$ 
\\  \hline

$\mathbb{L}_{79}^{3-2 \sqrt{2}-\sqrt{2} t} $&$\to$&$  {\mathbb L}_{64}$ & 
\multicolumn{2}{l}{
$E_1^t=\frac{t^3}{2}e_2-\frac{1}{\sqrt{2}}t^3(t+2-\sqrt{2})e_3 $ \quad
$E_2^t=\frac{t^2}{\sqrt{2}}e_1+\frac{(1-\sqrt{2})t^2}{2}e_2-\frac{(t+1-\sqrt{2})t^2}{\sqrt{2}}e_3$}  \\

\multicolumn{5}{l}{
$  E_3^t=\frac{t^4}{2 \sqrt{2}}e_3-\frac{t^4(2t+1-\sqrt{2})}{4}e_4+ \frac{t^4(2\sqrt{2}t^2+(5\sqrt{2}-8)t+5\sqrt{2}-7)}{4}e_5$\quad
 $E_4^t=-\frac{t^5}{2}e_4+ \frac{ t^5(\sqrt{2} t+2 \sqrt{2}-3)}{2}e_5$ \quad $E_5^t=-\frac{t^7}{2 \sqrt{2}} e_5$}
\\  \hline

$\mathbb{L}_{70}^{-\frac{2+t+t^2}{t}} $&$ \to  $&$ {\mathbb L}_{66}$ & 
\multicolumn{2}{l}{$ E_1^t=te_1+te_3 $ \quad $E_2^t=-e_1+te_2-(1+t+t^2)e_3$}\\
\multicolumn{4}{l}{$E_3^t=t^2e_3+t^3e_4+(1+2 t+2 t^2+t^3)e_5$ \quad
 $E_4^t=t^2e_4-(t+t^2+t^3)e_5$} &  $E_5^t=-t^2e_5$ 
 
\\  \hline

$\mathbb{L}_{52}^{\frac{t+1}{\alpha t^2-1},\ \frac{1}{t-1}} $&$\to$&$  {\mathbb L}_{70}^{\alpha}$ & 
$E_1^t=(t-t^2)e_2$&$ E_2^t=(\alpha t^2-1)e_1-(t-1)e_2 $ \\ 
\multicolumn{4}{l}{$ E_3^t=(1-t)(\alpha t^2-1)e_3+(1-t)(\alpha t^2-1)e_4 $
\quad $  E_4^t=(t-t^2)(\alpha t^2-1)e_4 $}&  $ E_5^t=(t^2-t^3)(\alpha t^2-1)e_5$
\\  \hline

$\mathbb{L}_{72}$ & $\to$ & ${\mathbb L}_{71}$ &
$E_1^t=te_1+te_4 $&$ E_2^t=t^2e_2+t^2e_5 $\\
\multicolumn{3}{l}{$E_3^t=t^3e_3$}&$ E_4^t=-t^3e_3+t^3e_4 $&$ E_5^t=t^4e_5$
\\  \hline

$\mathbb{L}_{79}^{3-\frac{1}{t}-2 t} $&$\to$&$  {\mathbb L}_{72}$ & 
\multicolumn{2}{l}{$E_1^t=te_1+(t^2-t)e_2 $ \quad $ E_2^t=t^3e_3+(t^4-t^3)e_4-(t^2-t)^2e_5$}\\
\multicolumn{3}{l}{$ E_3^t=t^4e_4+(t^4-t^3)e_5$}& $E_4^t=t^3e_2-(t^3-t^2)e_4$&$  E_5^t=t^5e_5$
\\  \hline

 $\mathbb{L}_{75} $ & $\to $ & $ {\mathbb L}_{74}$ &
$ E_1^t=\frac{i}{\sqrt{1-t}}e_1-\frac{i}{\sqrt{(1-t)^3}}e_2-\frac{i}{\sqrt{(1-t)^3}}e_4 $ &$ E_2^t=-\frac{i t}{\sqrt{(1-t)^3}}e_2 $\\ 
\multicolumn{3}{l}{$E_3^t=\frac{t}{(t-1)^2}e_3+\frac{t}{(t-1)^3}e_5$} &
$ E_4^t=\frac{i t}{\sqrt{(1-t)^5}}e_4 $&
$E_5^t=\frac{t}{(t-1)^3}e_5$
\\  \hline

$\mathbb{L}_{79}^{3-2t-\frac{1}{t}}$&$ \to $&$ {\mathbb L}_{75}$ & 
\multicolumn{2}{l}{ $E_1^t=e_1+(t-1)e_2-(t-1)e_3-(2t^2-5t+4-\frac{1}{t})e_4$}\\  

\multicolumn{5}{l}{
$E_2^t=te_2-(t-1)e_3-(t^2-4t +3-\frac{1}{t})e_3$ \quad 
$E_3^t=te_3+(t-1)^2e_4$  \quad
$E_4^t=te_4$  \quad  $E_5^t=te_5$}
\\  \hline

$\mathbb{L}_{75} $ & $ \to $ & $ {\mathbb L}_{76}$ & 

$E_1^t=\frac{1}{t}e_1+\frac{1}{t^{3}}e_2-\frac{1}{t^{5}}e_4$ &
$E_2^t=\frac{1}{t^3}e_2-\frac{1}{t^{5}}e_4$\\ 
\multicolumn{3}{l}{$E_3^t= \frac{1}{t^4}e_3$} & 
$E_4^t=\frac{1}{t^5}e_4$ &  $E_5^t=\frac{1}{t^6}e_5$\\
\hline

$\mathbb{L}_{79}^{\frac{3\sqrt{t}-2}{\sqrt{t}}} $&$ \to  $&$ {\mathbb L}_{77}$ & 
$E_1^t=\frac{1}{\sqrt{t}}e_1+\frac{1-\sqrt{t}}{t}e_2+\frac{1}{\sqrt{t}}e_3+\frac{3-2 \sqrt{t}}{t}e_4 $&$ E_2^t=\frac{1}{t}e_2$\\  
\multicolumn{3}{l}{$E_3^t=\frac{1}{t\sqrt{t}}e_3+\frac{1-\sqrt{t}+t}{t^2}e_4+\frac{1+t}{t^2}e_5$} & $ E_4^t=\frac{1}{t^2}e_4+\frac{1}{t^2}e_5 $&$ E_5^t=\frac{1}{t^2\sqrt{t}}e_5$
\\  \hline

$\mathbb{L}_{79}^{0} $&$ \to $&$  {\mathbb L}_{78}$ & 
$E_1^t=\frac{1}{\sqrt{t}}e_1+\frac{1-\sqrt{t}}{t}e_2+\frac{1}{\sqrt{t}}e_3+\frac{1+\sqrt{t}}{t}e_4 $&$ E_2^t=\frac{1}{t}e_2 $\\ 
\multicolumn{3}{l}{ $E_3^t=\frac{1}{t\sqrt{t}}e_3+\frac{1-\sqrt{t}+t}{t^2}e_4+\frac{1+t}{t^2}e_5$} & $ E_4^t=\frac{1}{t^2}e_4+\frac{1}{t^2}e_5 $&$ E_5^t=\frac{1}{t^2\sqrt{t}}e_5$
\\  \hline

$\mathbb{L}_{79}^{\frac{3\sqrt{\alpha}-\alpha-2}{\sqrt{\alpha}}} $&$ \to $&$  {\mathbb L}_{80}^{\alpha\neq0}$ & 
$E_1^t=\frac{t}{\sqrt{\alpha}}e_1+\frac{(1-\sqrt{\alpha})t}{\alpha}e_2-\frac{(\sqrt{\alpha}-1)^2}{\alpha\sqrt{\alpha}}e_4 $  & $ E_2^t=\frac{t^2}{\alpha}e_2+\frac{(\sqrt{\alpha}-1)t^2}{\alpha\sqrt{\alpha}}e_4 $\\  

\multicolumn{3}{l}{$E_3^t=\frac{t^2}{\alpha\sqrt{\alpha}}e_3-\frac{(\sqrt{\alpha}-1)t^2}{\alpha^2}e_4$}  &
$ E_4^t=\frac{t^3}{\alpha^2}e_4+\frac{(\sqrt{\alpha}-\alpha+t^2)t^3}{\alpha^2\sqrt{\alpha}}e_5$  &$ E_5^t=\frac{t^4}{\alpha^2\sqrt{\alpha}}e_5$
\\  \hline


$\mathbb{L}_{79}^{\frac{4+3\sqrt{2+t}}{\sqrt{2+t}}}$&$ \to$&$  {\mathbb L}_{81}$ & 

\multicolumn{2}{l}{$
E_1^t=-\frac{t}{\sqrt{2+t}}e_1+\frac{t(1+\sqrt{2+t})}{2+t}e_2+\frac{t(1+\sqrt{2+t})^2}{\sqrt{(2+t)^3}}e_4$}\\

\multicolumn{5}{l}{
$\begin{array}{ll}
E_2^t=\frac{t^2}{2+t}e_2+\frac{t^3}{\sqrt{(2+t)^3}}e_3-\frac{t^2 \left(t+t^2 \left(1-\sqrt{2+t}\right)-2 \left(1+\sqrt{2+t}\right)\right)}{\sqrt{(2+t)^5}}e_4  & E_3^t=-\frac{t^2}{\sqrt{(2+t)^3}}e_3+\frac{t^2 \left(1+\sqrt{2+t}\right)}{(2+t)^2}e_4\\
E_4^t=\frac{t^3}{(2+t)^2}e_4+\frac{t^3 \left(1+\sqrt{2+t}\right)}{(2+t)^2}e_5  & 
E_5^t=-\frac{t^4}{\sqrt{(2+t)^5}}e_5
\end{array}$}
\\  \hline

$\mathbb{L}_{47}^{\frac{t^2-t+1}{t},\ \frac{(t^2-t+1)^3}{t^2-t}} $&$ \to  $&$ {\mathfrak V}_{2+3}$ & 
$ E_1^t=(t^3-t^2+t)e_1+te_2 $&\\
\multicolumn{5}{l}{$E_2^t=(t^3-t^2+t)e_1+t^2e_2-(t^2-t+1)^2e_3+(t-1)(t+(t-1)\lambda )e_4$}\\
\multicolumn{5}{l}{$ E_3^t=(t^2-t)(t^2-t+1)^2e_3+(t^4-t^3)e_4-\lambda (t^2-t)(t^2-t+1)^3e_5 $} \\ 
\multicolumn{5}{l}{$ E_4^t=(t^4-2t^3+t^2)e_4+\lambda (t-1)(t^2-t+1)^3e_5 $ \quad $ E_5^t=(t^2-t)(t^2-t+1)^3e_5$}
\\  \hline

\end{longtable}
 }

For the rest of degenerations, in  case of  $E_1^t,\dots, E_5^t$ is a {\it parametric basis} for ${\bf A}\to {\bf B},$ it will be denoted by
${\bf A}\xrightarrow{(E_1^t,\dots, E_5^t)} {\bf B}$.

{\small
\begin{longtable}{lcl|lcl} \hline

$\mathbb{L}_{04}$ & $\xrightarrow{ (t^{-1} e_1, t^{-2}e_2, t^{-1}e_3, t^{-3}e_4, t^{-2}e_5)}$ & ${\mathbb L}_{01}$ &

$\mathbb{L}_{05}^{\alpha}$ & $\xrightarrow{ (t^{-1}e_1, t^{-2}e_2,  t^{-1}e_3,  E t^{-3}e_4, t^{-2}e_5)}$ & ${\mathbb L}_{03}^{\alpha}$ 
\\  \hline

$\mathbb{L}_{07}$ & $\xrightarrow{ (t^{-1} e_1, t^{-2}e_2, t^{-1}e_3, t^-{3}e_4, t^{-2}e_5)}$ & ${\mathbb L}_{06}$ & 
 
$\mathbb{L}_{08}$ & $\xrightarrow{ (t^{-1}e_1, t^{-2}e_2, t^{-2}e_3, t^{-3}e_4, t^{-4}e_5)}$ & ${\mathbb L}_{07}$   
\\  \hline

$\mathbb{L}_{11}$ & $\xrightarrow{ (t^{-1} e_1, t^{-1}e_2, t^{-2}e_3, t^{-3}e_4, t^{-2}e_5)}$ & ${\mathbb L}_{10}$ 
&   

$\mathbb{L}_{13}$ & $\xrightarrow{ (t^{-1} e_1, t^{-1}e_2, t^{-2}e_3, t^{-3}e_4, t^{-2}e_5)}$ & ${\mathbb L}_{12}$ 

\\  \hline

$\mathbb{L}_{14}^{t} $ & $\xrightarrow{ (\frac{1}{t}e_1,   \frac{1}{t}e_2,   \frac{1}{t^2}e_3, \frac{1}{t^3}e_4 , \frac{1}{t^2}e_5)} $ & $ {\mathbb L}_{15} $ &

$\mathbb{L}_{17}$ & $\xrightarrow{ (t^{-1}e_1, t^{-1}e_2, t^{-2}e_3, t^{-3}e_4, t^{-2}e_5)}$ & ${\mathbb L}_{16}$  
\\  \hline

$\mathbb{L}_{19}$ & $\xrightarrow{ (t^{-1}e_1, t^{-1}e_2, t^{-2}e_3, t^{-3}e_4, t^{-2}e_5)}$ & ${\mathbb L}_{18}$ &

$\mathbb{L}_{25} $ & $ \xrightarrow{ (t^2e_1 , t^4e_2 , t^3e_3 , t^4e_4 , t^6e_5)}  $ & $ {\mathbb L}_{22}$ 
\\  \hline

$\mathbb{L}_{24}^{\frac14}$ & $\xrightarrow{ (t e_1, t^{2}e_2, -te_3+2te_4, -2t^{2}e_4, t^{3}e_5)}$ & ${\mathbb L}_{23}$ & 

$\mathbb{L}_{24}^{-\frac{1}{t^2}} $ & $ \xrightarrow{ (e_1 , e_2+te_3 , e_3 , (1-t)e_3-te_4 , e_5)} $ & $ {\mathbb L}_{26}$ 
\\  \hline

$\mathbb{L}_{32}$ & $\xrightarrow{ (t^{-2} e_1, t^{-2}e_2, t^{-4}e_3, t^{-3}e_4, t^{-6}e_5)}$ & ${\mathbb L}_{31}$ & 

$\mathbb{L}_{41}$ & $\xrightarrow{ (t^{-1}e_1, t^{-1}e_2, t^{-2}e_3, t^{-2}e_4, t^{-3}e_5)}$ & ${\mathbb L}_{40}$  
\\  \hline

$\mathbb{L}_{43}$ & $\xrightarrow{ (t^{-1} e_1, t^{-1}e_2, t^{-2}e_3, t^{-2}e_4, t^{-3}e_5)}$ & ${\mathbb L}_{42}$ &

$\mathbb{L}_{33}^{-\frac{1}{\sqrt{t}}} $ & $ \xrightarrow{ (e_1+\frac{1}{\sqrt{t}}e_4,  e_2-\frac{1}{t}e_3+\frac{1}{\sqrt{t}}e_4,  e_3 , -e_3+\sqrt{t}e_4 , e_5)}
 $ & $  {\mathbb L}_{43}$ 
\\  \hline

$\mathbb{L}_{45}$ & $\xrightarrow{ (t^{-1}e_1, t^{-1}e_2, t^{-2}e_3, t^{-2}e_4, t^{-3}e_5)}$ & ${\mathbb L}_{44}$ &

$\mathbb{L}_{49}$ & $\xrightarrow{ (t^{-1}e_1, t^{-1}e_2, t^{-2}e_3, t^{-2}e_4, t^{-3}e_5)}$ & ${\mathbb L}_{48}$  

\\  \hline

$\mathbb{L}_{52}^{\frac{1}{t}, \ t} $ & $ \xrightarrow{ (te_1 , e_2 , te_3 , te_4 , te_5)} $ & $ {\mathbb L}_{50}$ & 

$\mathbb{L}_{52}^{t-1, \ 0} $ & $ \xrightarrow{ (e_1+e_4 , te_2+e_3 , te_3-e_5 , te_4+e_5 , te_5)} 
  $ & $ {\mathbb L}_{51}$ 
\\  \hline

$\mathbb{L}_{52}^{0, \ \frac{1}{t}} $ & $ \xrightarrow{ (te_1 , te_2 , t^2e_3 , t^2e_4 , t^2e_5)} $ & $ {\mathbb L}_{53}$ &

$\mathbb{L}_{52}^{-1, \ \frac{1}{t}} $ & $ \xrightarrow{ (e_1 , te_2 , te_3 , te_4 , te_5)}  $ & $ {\mathbb L}_{54}$ 
\\  \hline

$\mathbb{L}_{52}^{0, \ \alpha} $ & $ \xrightarrow{ (te_1 , e_2 , te_3 , te_4 , te_5)}
 $ & $ {\mathbb L}_{55}^{\alpha}$ &

$\mathbb{L}_{66}$ & $\xrightarrow{ (t^{-1}e_1, t^{-1}e_2, t^{-2}e_3, t^{-2}e_4, t^{-3}e_5)}
$ & ${\mathbb L}_{65}$  
\\  \hline

$\mathbb{L}_{70}^{\frac{1}{\alpha}}$&$\xrightarrow{ (te_1 , e_1+\alpha e_2 , \alpha^2 e_3 , \alpha te_4 , \alpha te_5)}$&$  {\mathbb L}_{67}^{\alpha\neq0}$ &

$\mathbb{L}_{69}$ & $\xrightarrow{ (t^{-1}e_1, t^{-1}e_2, t^{-2}e_3, t^{-2}e_4, t^{-3}e_5)}$ & ${\mathbb L}_{68}$  

\\  \hline

$\mathbb{L}_{70}^{\frac{1}{t^2}} $&$\xrightarrow{ (te_1-t^3e_3+te_4 , t^2e_2 , t^4e_3 , t^3e_4 , t^4e_5)}
$&$  {\mathbb L}_{69}$ & 

$\mathbb{L}_{74}$ & $\xrightarrow{ (t^{-1}e_1, t^{-1}e_2, t^{-2}e_3, t^{-3}e_4, t^{-4}e_5)}$ & ${\mathbb L}_{73}$  

\\  \hline

\multicolumn{6}{c}{
$\mathbb{L}_{82} $ \quad $ \xrightarrow{ (e_1, e_2, e_3,e_4, {t^{-1}}e_5)}
$ \quad $\mathfrak{L}_{13} \oplus \mathbb C$}
\\  \hline
\end{longtable}
}

\end{proof}

\end{document}